\documentclass{amsart}

\usepackage{amssymb}
\usepackage{amsthm}
\usepackage[style=alphabetic, url=false, isbn = false, backend=biber]{biblatex}
\usepackage{tikz-cd}
\usepackage{mathtools}
\usepackage{framed}
\usepackage[mathscr]{eucal}
\usepackage{fullpage}
\usepackage{hyperref}
\usepackage{enumerate}
\usepackage{wasysym}
\usepackage{xcolor}
\usepackage{scalerel}
\usepackage{mathdots}
\hypersetup{colorlinks=true,citecolor=blue,urlcolor =black,linkbordercolor={1 0 0}}
\mathtoolsset{centercolon}

\newtheorem{theorem}{Theorem}[section]

\newtheorem{proposition}[theorem]{Proposition}
\newtheorem{lemma}[theorem]{Lemma}
\newtheorem{corollary}[theorem]{Corollary}
\newtheorem{conjecture}[theorem]{Conjecture}

\theoremstyle{definition}

\theoremstyle{remark}
\newtheorem*{remark}{Remark}

\newcommand{\on}{\operatorname}
\newcommand{\abs}[1]{\left|#1\right|}

\newcommand{\emptyinput}{\underline{\,\,\,\,}}

\newcommand{\Db}{\mathrm{D}^b}

\newcommand{\Cat}{\mathrm{Cat}}

\newcommand{\colim}{\on{colim}}

\newcommand{\im}{\on{im}}
\newcommand{\coker}{\on{coker}}
\newcommand{\Hom}{\on{Hom}}
\newcommand{\End}{\on{End}}

\newcommand{\homsheaf}{\cH\!\on{om}}

\newcommand{\Endsheaf}{\cE\!\on{nd}}
\newcommand{\Sym}{\on{Sym}}

\newcommand{\cO}{\mathcal{O}}
\newcommand{\cA}{\mathcal{A}}
\newcommand{\cB}{\mathcal{B}}
\newcommand{\cC}{\mathscr{C}}

\newcommand{\cE}{\mathcal{E}}
\newcommand{\cF}{\mathcal{F}}
\newcommand{\cG}{\mathcal{G}}
\newcommand{\cH}{\mathcal{H}}

\newcommand{\cL}{\mathcal{L}}
\newcommand{\cM}{\mathcal{M}}

\newcommand{\cQ}{\mathcal{Q}}
\newcommand{\cR}{\mathcal{R}}
\newcommand{\cS}{\mathcal{S}}

\newcommand{\cU}{\mathcal{U}}
\newcommand{\cV}{\mathcal{V}}

\newcommand{\cZ}{\mathcal{Z}}

\newcommand{\bbC}{\mathbb{C}}

\newcommand{\bbP}{\mathbb{P}}

\newcommand{\bbS}{\mathbb{S}}

\newcommand{\bbZ}{\mathbb{Z}}

\newcommand{\Gr}{\on{Gr}}
\newcommand{\OGr}{\on{OGr}}

\newcommand{\Spec}{\on{Spec}}

\newcommand{\cl}{\mathcal{C}l}
\newcommand{\Cl}{\mathfrak{C}l}
\newcommand{\textcl}{\mathrm{Cl}}
\newcommand{\even}{\mathrm{even}}

\newcommand{\GL}{\on{GL}}

\newcommand{\SO}{\on{SO}}
\newcommand{\Spin}{\on{Spin}}

\newcommand\iso{\xrightarrow{
		\,\smash{\raisebox{-0.3ex}{\ensuremath{\scriptstyle\sim}}}\,}}
\newcommand{\parity}{\on{parity}}
\newcommand{\sgn}{\on{sgn}}

\title{Semiorthogonal decompositions and components of derived categories of orthogonal Grassmannian fibrations} 
\author{Saket Shah}

\begin{document}
	\maketitle 
	\begin{abstract}
		Kuznetsov showed that for a flat quadric fibration $\cQ$ over a smooth base $S$, $\Db(\cQ)$ admits a semiorthogonal decomposition where one of the components is the derived category of the sheaf of even parts of a Clifford algebra $\Db(S,\cl_0)$. \par 
		As progress towards a generalization, we show that for a quadric fibration with a fairly minor condition on the rank of the quadric fibers, the category $\Db(S,\cl_0)$ embeds fully faithfully into the derived category of the relative orthogonal Grassmannian $\Db(\OGr(k,\cQ))$. When $k = 2$, we use this to produce a semiorthogonal decomposition of $\Db(\OGr(2,\cQ))$ up to a residual category; we compute this residual category in the smooth case and produce a conjecture for in the case of a pencil of quadrics with smooth base locus. 
	\end{abstract}
	\tableofcontents
	\section{Introduction}
	The study of derived categories of quadric hypersurfaces and quadric fibrations goes back to Kapranov, who showed that smooth quadric hypersurfaces always admit a full exceptional collection \cite{kapranovcoll}. For a smooth quadric $Q \subset \bbP(V)$, this collection can be written as 
	\begin{equation}
		\Db(Q) = \begin{cases}
				\langle \cO(-n + 1), \cO(-n+2),\cdots, \cO, \cS_+, \cS_- \rangle &\text{if }\dim Q \text{ is even} \\
				\langle \cO(-n + 1),\cO(-n+2),\cdots,\cO, \cS\rangle &\text{if }\dim Q \text{ is odd}.
			\end{cases}
	\end{equation}
	While most bundles appearing in these decompositions are line bundles pulled back from the ambient $\bbP(V)$, the interesting bundles here are the \textit{spinor bundles} $\cS_{\pm}$ and $\cS$. Constructed by Ottaviani \cite{ott1}, these bundles are equivariant for the transitive action of $\Spin(V)$ on $Q$ and can be viewed as incarnations on $Q$ of the half-spin and spin representations. Kapranov's arguments for the full exceptional collection, on the other hand, construct these spinor bundles using the \textit{Clifford algebra}: 
	\begin{equation}
		\textcl(V) = T(V)/(v \otimes v - Q(v,v)),
	\end{equation}
	where $T(V) = \bigoplus_{k \geq 0} V^{\otimes k}$ is the tensor algebra of $V$. 
	\par  
	When we allow the quadric $Q$ to degenerate in a family $\cQ \xrightarrow{p} S$ with possibly singular fibers, the fibers are no longer homogeneous spaces for Spin groups and it is more difficult to leverage representation-theoretic methods. Nonetheless, the utility of spinor sheaves and Clifford algebras persists. Kuznetsov shows \cite{quadfibs} that even in the relative setting, there is a natural semiorthogonal decomposition (which is moreover $S$-linear in the sense of \cite{basechange}, or any other reasonable sense):
	\begin{equation}
		\Db(\cQ) = \begin{cases}
			\langle p^*\Db(S) \otimes \cO(-n + 1), p^*\Db(S) \otimes \cO(-n+2),\cdots, p^*\Db(S), \Db(S,\cl_0)\rangle &\text{if }\dim Q \text{ is even} \\
			\langle p^*\Db(S) \otimes \cO(-n + 1),p^*\Db(S) \otimes \cO(-n+2),\cdots,p^*\Db(S), \Db(S,\cl_0)\rangle &\text{if }\dim Q \text{ is odd}.
		\end{cases}
	\end{equation}
	where $\cl_0$ is the \textit{sheaf of even parts for the Clifford algebra}. Taking $S$ to be a point and $\cQ\to S$ to be just a single smooth quadric, the observation that $\Db(S,\cl_0)$ consists of one or two exceptional objects (depending on the parity of $\dim \cQ$) recovers the collection of Kapranov. When $S$ is a point but $\cQ\to S$ is not necessarily smooth, Addington gives a different proof that $\Db(S,\cl_0)$ embeds into $\Db(\cQ)$, and at the same time defines a good notion of spinor sheaves on singular quadrics \cite{singularquadrics}. \par 
	In a slightly different setting, spinor bundles also show up as exceptional vector bundles arising in full exceptional collections of orthogonal Grassmannians, which are also homogeneous spaces for the action of $\Spin(V)$ \cite{kuznetsovodd, kuznetsoveven}. The natural question to ask is whether these spinor bundles can be relativized, just as they were in the case of quadric fibrations, to provide embedding of derived categories of sheaves of even parts for Clifford algebras. For smooth fibrations, it may be possible to leverage directly the existence of the fiberwise full exceptional collections, coming from the representation theory of Spin groups, but we will provide a somewhat different approach that allows us to handle the singular fibers. 
	\par 
	With that said, we begin in earnest by fixing a smooth base $S$ over $\bbC$ admitting a flat quadric fibration $p  : \cQ \to S$ of relative dimension $n - 2$. Taking the relative orthogonal Grassmannian of this family gives a map $\tau : \OGr(k,\cQ) \to S$. The first theorem of this paper is the following embedding theorem, under a mild assumption on the singularities of the fibers of $p$, which are used mainly to ensure the relative orthogonal Grassmannian is flat over the base. 
	\begin{theorem}\label{embeddingtheorem}
		Take any integer $2 \leq k \leq n/2$. Suppose that every quadric in the fibration $\cQ \to S$ has rank at least $2k-1$. Then there exists a fully faithful functor $\Phi : \Db(S,\cl_0) \to \Db(\OGr(k,\cQ))$ linear over $\Db(S)$. 
	\end{theorem}
	Rather than constructing spinor sheaves in the style of Addington, we produce an explicit sheaf of $\cl_0$-modules on $\OGr(k,\cQ)$ generalizing the sheaves of $\cl_0$-modules Kuznetsov constructs on $\cQ$ in \cite{quadfibs} and use it as a Fourier-Mukai kernel in order to define the embedding functor. \par
	As an application of the above, we are able to relativize most of the full exceptional collections for orthogonal Grassmannians of 2-dimensional subspaces. To state the result, we define the subcategories on $\OGr(2,\cQ)$ when $n = 2m$ is even: 
	\begin{equation}
		\begin{aligned}
			\cA &\coloneq \langle \tau^*\Db(S), \tau^*\Db(S) \otimes \cU_2^\vee, \tau^*\Db(S) \otimes \Sym^2 \cU_2^\vee, \cdots, \tau^*\Db(S) \otimes \Sym^{m-3} \cU_2^\vee, \Phi(\Db(S,\cl_0))\rangle, \\ 
			\cB &\coloneq \langle \tau^*\Db(S), \tau^*\Db(S) \otimes \cU_2^\vee, \tau^*\Db(S) \otimes \Sym^2 \cU_2^\vee, \cdots,\tau^*\Db(S) \otimes \Sym^{m-2} \cU_2^\vee, \Phi(\Db(S,\cl_0))\rangle.
		\end{aligned}
	\end{equation}
	When $n = 2m + 1$ is odd, we only take one subcategory: 
	\begin{equation}
		\cB \coloneq \langle \tau^*\Db(S), \tau^*\Db(S) \otimes \cU_2^\vee, \tau^*\Db(S) \otimes \Sym^2 \cU_2^\vee, \cdots,\tau^*\Db(S) \otimes \Sym^{m-2} \cU_2^\vee, \Phi(\Db(S,\cl_0))\rangle.
	\end{equation}
	Given these definitions, the semiorthogonal decompositions of $\OGr(2,\cQ)$ can be written as follows: 
	\begin{theorem} \label{sodtheorem}
		Suppose that every quadric in the fibration $\cQ \to S$ has rank at least $3$. Let $\cO_{\OGr(2,\cQ)}(-1)$ denote the determinant of the tautological subbundle on $\OGr(2,\cQ)$. If $n = 2m$, then the subcategories $\cA$ and $\cB$ are $S$-linear semiorthogonal sequences and there is an $S$-linear semiorthogonal decomposition 
		\begin{equation}
			\Db(\OGr(2,\cQ)) = \langle \cR_{\text{even}}, \cA, \cB(1), \cdots, \cB(m-2),\cA(m-1),\cdots,\cA(2m-4)\rangle,
		\end{equation}
		where $\cR_{\text{even}}$ is the residual category defined as the right orthogonal to the other components. If $n = 2m + 1$, then the subcategory $\cQ$ is an $S$-linear semiorthogonal sequence and there is an $S$-linear semiorthogonal decomposition 
		\begin{equation}
			\Db(\OGr(2,\cQ)) = \langle \cR_{\text{odd}}, \cB, \cB(1), \cdots, \cB(2m-3)\rangle,
		\end{equation}
		where $\cR_{\text{odd}}$ is the residual category.
	\end{theorem}
	The key insight is that all of the necessarily cohomological vanishings can be checked on the ambient Grassmannian by exploiting the Koszul resolution. The resulting semiorthogonal sequences are direct generalizations of exceptional collections in \cite{kuznetsoveven} and \cite{kuznetsovodd} respectively. \par 
	 In Theorem \ref{sodtheorem}, we restrict to $k = 2$ as it simplifies the necessary cohomological calculations, but it is possible that the sufficiently motivated reader can use these methods to show the existence of similar semiorthogonal decompositions for $k > 2$. 
	\subsection*{Outline of the paper}
	We review the language of Clifford algebras in the relative setting in Section \ref{backgroundsection}, before constructing the Fourier-Mukai kernel for the embedding in Section \ref{kernelsection}. The key technical cohomological vanishings used in the proofs for both theorems follow from the Borel-Weil-Bott theorem (but as the fibers need not be smooth, we can only use it on the ambient Grassmannian), and are handled in Section \ref{bwbsection}. To understand the smooth case, we make explicit the connection between our kernel and spinor bundles on smooth orthogonal Grassmannians; this is done in Section \ref{smoothsection}. Using the cohomological calculations and deformation to the smooth case, we reduce the proof of Theorem \ref{embeddingtheorem} to a fiberwise linear-algebraic argument, which is done in Section \ref{proofsection}. After a similar reduction to linear algebra, we prove Theorem \ref{sodtheorem} in Section \ref{semiorthogonalitysection}. Finally, in Section \ref{residualsection} we study the residual categories $\cR_{\text{even}}$ and $\cR_{\text{odd}}$ in the smooth case and suggest a conjecture in more generality, motivated by calculations coming from the case of a pencil of quadrics with smooth base locus. 
	\subsection*{Acknowledgments}
	We are thankful to Calvin Yost-Wolff and James Hotchkiss for useful conversations, as well as Pieter Belmans and Alexander Kuznetsov for discussing a prior draft of this paper. Additionally, an unpublished library of Sage code written by Pieter Belmans \cite{BWB} was very helpful in preliminary Borel-Weil-Bott cohomology calculations. As always, many thanks go to my advisor Alex Perry for his continued support. \par 
	During the preparation of this paper, the author was informed that Aporva Varshney was able to independently prove many related results by different methods when the quadric fibration is a nonsingular pencil of quadrics. 
	\section*{Conventions}
	We work over $\bbC$. Fix a line bundle $\cL$ and a rank $n$ vector bundle $\cE$ on a smooth base variety $S$. Throughout, we consider a flat quadric fibration $p : \cQ \to S$ of relative dimension $n - 2$, by which we mean that $\cQ \subset \bbP(\cE) \to S$ is cut out by a section $H^0(\bbP(\cE),\cL^\vee \otimes \cO_{\bbP(\cE)}(2)) = H^0(S, \cL^\vee \otimes \Sym^2 \cE^\vee)$ such that $\cL \to \Sym^2 \cE^\vee$ is fiberwise injective. We denote the associated orthogonal Grassmannian fibration as $\tau : \OGr(k,\cQ) \to S$, where $\OGr(k,\cQ) \subset \Gr(k,\cE)$ is cut out by the associated section of $\cL^\vee \otimes \Sym^2 \cE^\vee$. On this fibration we denote the tautological subbundle by $\cU_k$ and its determinant by $\cO(-1) \coloneq \cO_{\OGr(k,\cQ)}(-1)$. 
	\section{Background} \label{backgroundsection}
	\subsection{Clifford algebras}
	We recall for the reader some of the theory of sheaves of even parts of Clifford algebras, as expounded in Kuznetsov's work on quadric fibrations \cite{quadfibs}. \par 
	In the aforementioned paper, Kuznetsov constructs over $S$ two natural sheaves of noncommutative algebras: 
	\begin{itemize}
		\item A sheaf of $\bbZ$-graded algebras $\Cl_\bullet(\cQ)$ called the \textit{graded Clifford algebra} of $\cQ$, which relativizes the construction of Kapranov \cite{kapranovcoll}. By definition, 
		\begin{equation}
			\Cl_\bullet(\cQ) := T^\bullet(\cE)/(\ker(\Sym^2 \cE \to \cL^\vee)
		\end{equation}
		is a quotient of the graded tensor algebra $T^\bullet(\cE) := \bigoplus_{m=0}^\infty \cE^{\otimes m}$. We will consider the abelian category $\on{qgr} \Cl_\bullet(\cQ)$, the quotient of the category of sheaves of finitely generated graded left $\Cl_\bullet(\cQ)$-modules by the modules of finite rank over $\cO_S$. 
		\item A sheaf of algebras $\cl_0 := \varinjlim \Cl_{2k}(\cQ) \otimes \cL^{k}$ called the \textit{sheaf of even parts of the Clifford algebra of $\cQ$}. We will consider $\on{mod} \cl_0$, the abelian category of finitely generated left $\cl_0$-modules. 
	\end{itemize}
	The most important relation between these two algebras is the following equivalence: 
	\begin{proposition}[\cite{quadfibs}] \label{qgrprop}
		There is an exact equivalence of abelian categories
		\begin{equation}
			q_* : \on{qgr} \Cl_\bullet(\cQ) \simeq \on{mod} \cl_0,
		\end{equation}
		given by $M_\bullet \mapsto \varinjlim M_{2k} \otimes \cL^k$. 
	\end{proposition}
	\begin{remark}
		In fact the equivalence of the proposition above holds also for the categories of left modules; the same argument works without modification. 
	\end{remark}
	In fact, there is a series of $\cl_0$-bimodules which are, on the level of graded $\Cl_\bullet(\cQ)$-modules, given simply by shifts of $\Cl_\bullet(\cQ)$-itself: 
	\begin{equation}
		\cl_k := q_* \Cl_{\bullet+k}(\cQ),
	\end{equation}
	satisfying the compatibilities $\cl_{k+2} \simeq \cl_k \otimes \cL^\vee$. \par 
	It is worth noting that, as is true for any module over a sheaf of noncommutative algebras over $S$, taking the $\cO$-linear dual $\cM^\vee$ of any left (resp. right) $\cl_0$-module $\cM$ admits a natural right (resp. left) $\cl_0$-module structure inherited from that of $\cM$. To be precise, in the case of right $\cl_0$-modules, we take the $\cl_0$-module multiplication map 
	\[\cM \otimes \cl_0\to \cM,\]
	observe that it induces a natural map 
	\[\cl_0 \to \Endsheaf_\cO(\cM,\cM)\]
	and observe that the left module structure on the dual $\cM^\vee := \homsheaf_\cO(\cM,\cO)$ is given by 
	\begin{equation}
		\cl_0 \otimes \homsheaf_\cO(\cM,\cO) \to \Endsheaf_\cO(\cM) \otimes \homsheaf_{\cO}(\cM,\cO) \to \homsheaf_\cO(\cM,\cO).
	\end{equation}
	\par
	When $\cM = \cl_0$ itself, the image of this map is quite easily described by the following observation. 
	\begin{proposition}[\cite{quadfibs}] \label{endomorphismprop}
		The map $\cl_0 \to \Endsheaf_\cO(\cl_0) \simeq (\cl_0)^\vee \otimes \cl_0$ given by right Clifford multiplication of $\cl_0$ onto itself is an isomorphism of $\cl_0$-bimodules onto the left $\cl_0$-linear endomorphisms of $\cl_0$. 
	\end{proposition} 
	\begin{proof}
		At least as right modules, this is \cite[Lemma 3.8]{quadfibs}, but it essentially follows from the definition of the left $\cl_0$-module structure on $(\cl_0)^\vee$ that the isomorphism preserves the left module structures. 
	\end{proof}	
	A useful observation in \cite{quadfibs} is that both $\Cl_\bullet$ and $\cl_0$ admit explicit descriptions in terms of $\cE$ and $\cL$: 
	\begin{equation}
		\Cl_\bullet = \cO_S \oplus \cE \oplus \left(\bigwedge^2 \cE \oplus \cL^\vee\right)\oplus\left(\bigwedge^3 \cE \oplus \cE \otimes \cL^\vee\right) \oplus \left(\bigwedge^4 \cE\oplus \bigwedge^2 \cE \otimes \cL^\vee \oplus \cE \otimes (\cL^\vee)^{\otimes 2}\right) \oplus \cdots 
	\end{equation}
	and 
	\begin{equation}
		\cl_0 = \cO_S \oplus \bigwedge^2 \cE \otimes \cL\oplus\bigwedge^4 \cE \otimes \cL^{\otimes 2} \oplus \cdots 
	\end{equation}
	and in particular the fibers of $\cl_0$ are given precisely by the even subalgebra of the usual Clifford algebra of a single quadratic form. In other words, for any $s \in S$, 
	\begin{equation}
		\cl_0|_s = T^{\even}(\cE|_s)/\ker(\Sym^2 \cE|_s \to \cL^\vee|_s) =: \textcl_{\even},
	\end{equation}
	where $\Sym^2 \cE|_s \to \cL^\vee|_s$ is the fiber of the relative quadratic form. Note that $\textcl_{\even}$ depends on the quadric $\cQ_s$, but we will omit this dependence when it is clear from context. 
	\subsection{Relative orthogonal Grassmannians}
	For a quadric fibration $p : \cQ \to S$ as in our assumptions, by the orthogonal Grassmannian $\OGr(k,\cQ)$ we mean simply the relative Hilbert scheme of lines for $p$. The reason that we will restrict to the case where the rank of the quadrics occurring in the quadric fibration is $\geq 2k - 1$, or equivalently corank $\leq n - 2k + 1$, is for the following technical input.
	\begin{lemma}
		\label{coranklemma}
		% note to the author: the corank assumption changes for higher-dimensional linear spaces.
		Suppose that for any point $s \in S$, the fiber $\cQ_s \subset \bbP(\cE|_s)$  of the quadric fibration $p : \cQ \to S$ has corank $\leq n - 2k + 1$. Then $\OGr(k,\cQ) \subset \Gr(k,\cE)$ is cut out by a regular section of the vector bundle \[\cV := \Sym^2 \cU_k^\vee \otimes \tau^*\cL^\vee\] where $\cU_k$ is the tautological subbundle on $\Gr(k,\cE)$, and the morphism $\tau : \OGr(k,\cQ) \to S$ is flat. Moreover for every $s \in S$ the section of $\cV|_s$ which cuts out $\OGr(k,\cQ_s) \subset \Gr(k,\cE|_s)$ remains regular.  
	\end{lemma} 
	\begin{proof}
		First, it is actually immediate from the definition that $\OGr(k,\cQ) \subset \Gr(k,\cE)$ is cut out by a section of the vector bundle $\cV$. To show that it is cut out by a regular section of $\cV$, it suffices to show that it has codimension $k\cdot \frac{k+1}{2}$ in $\Gr(k,\cE)$, as $\on{rank}\cV = k\cdot \frac{k+1}{2}$. \par 
		We first observe that if for every $s \in S$, the quadric $\cQ_s$ has corank $\leq n - 2k + 1$ in the $n$-dimensional vector space $\cE|_s$, then by Lemma 2.2 of \cite{coskunflagvar}, the variety of isotropic lines $\OGr(k,\cQ_s)$ has pure dimension $k\cdot\frac{2n-3k-1}{2}$, with either 1 irreducible component if the corank is $n - 2k + 1$ or $< n - 2k$, or two irreducible components if the corank is $n - 2k$. As a consequence we conclude that $\dim \OGr(2,\cQ) = k\cdot\frac{2n-3k-1}{2} + \dim S$, while $\dim \Gr(k,\cE) = k(n-k) + \dim S = k\cdot\frac{2n-3k-1}{2} + k\cdot \frac{k+1}{2} + \dim S $, so we may conclude that $\OGr(k,\cQ)$ is cut out by a regular section of $\cV$. In particular this implies that $\OGr(k,\cQ)$ is a local complete intersection variety, therefore CM. By miracle flatness, it follows that the morphism $\tau$ is flat. \par 
		The last claim of the lemma follows immediately from the fact that for each $s \in S$, $\OGr(k,\cQ_s)$ is codimension $k\cdot\frac{k+1}{2}$ inside of $\Gr(k,\cE|_s)$. 
	\end{proof}
	One useful consequence of the lemma above is that we can readily calculate the relative dualizing line bundle $\omega_{\OGr(k,\cQ)/S}$, since $\OGr(k,\cQ)$ is cut out by a regular section of a vector bundle on $\Gr(k,\cE)$. Using the isomorphism $\omega_{\OGr(k,\cQ)/S} \simeq \omega_{\Gr(k,\cE)/S} \otimes \det (\Sym^2 \cU_k^\vee \otimes \tau^*\cL^\vee)|_{\OGr(k,\cQ)}$, we find 
	\begin{equation} \label{dualizing}
		\omega_{\OGr(k,\cQ)/S} \simeq \cO(-n+k+1) \otimes (\tau^*\cL^\vee)^{\otimes k(k+1)/2}.
	\end{equation}
	Since $\cL$ is a line bundle pulled back from the base, on the fibers this reduces to simply $\cO(-n+k+1)$. 
	\section{Construction of the Fourier-Mukai kernels}\label{kernelsection}
	The construction of our kernels is analogous to the sheaves $\cE_{k,l}$ arising in Kuznetsov's proof of this claim in \cite{quadfibs}. \par 
	We first observe that there is a natural graded Clifford multiplication map of graded right $\Cl_\bullet(\cQ)$ modules 
	\[\cE \otimes \Cl_{\bullet}(\cQ) \to \Cl_{\bullet +1 }(\cQ);\]
	shifting degrees and applying the functor $q_*$ therefore yields for any $i \in \bbZ$ a \textit{Clifford multiplication} map 
	\begin{equation}
		\cE\otimes \cl_i \to \cl_{i+1}. 
	\end{equation}
	On the other hand, the orthogonal Grassmannian $\OGr(k,\cQ)$ supports a tautological isotropic subbundle $\cU_k \hookrightarrow \tau^*\cE$ pulled back from the Grassmann bundle $\Gr(k, \cE)$ over $S$, so composing with this inclusion gives for each $i \in \bbZ$ a natural multiplication map 
	\begin{equation}
		\cU_k \otimes \tau^*\cl_i \to \tau^*\cl_{i+1}. 
	\end{equation}
	As a consequence, it is possible to build a long exact sequence which we will use to construct our kernel. To do so, first observe that there is a natural inclusion $\Sym^p \cU_k \subset \Sym^{p-1} \cU_k \otimes \cU_k$ as a direct summand; then Clifford multiplication yields a map
	\[\Sym^p \cU_k \otimes \cl_i\to \Sym^{p-1}\cU_k \otimes \cU_k \otimes \cl_i \to \Sym^{p-1}\cU_k \otimes \cl_{i+1}.\]
	On the other hand, we can interpret $\cO_{\OGr(k,\cQ)}(-1) = \bigwedge^k \cU_k \subset \cU_k^{\otimes k}$, apply Clifford multiplication $k$ times, and get a map 
	\[\cO_{\OGr(k,\cQ)}(-1) \otimes \tau^* \cl_0 \to \tau^* \cl_k.\]
	Finally, we have natural maps 
	\[\Sym^p \cU_k^\vee \otimes \tau^*\cl_{i} \to \Sym^p \cU_k^\vee \otimes \cU_k^\vee \otimes \cU_k \otimes \tau^*\cl_{i} \to \Sym^{p+1} \cU_k^\vee \otimes \tau^*\cl_{i+1}\]
	given first by the monoidal coevaluation in the middle followed by multiplication onto each factor. 
	\begin{proposition}
		\label{LES}
		There is a long exact sequence of right $\tau^*\cl_0$-modules on $\OGr(k,\cQ)$ given by 
%		\begin{equation}
%			\begin{gathered}
			\begin{multline}
				\cdots \to \Sym^p \cU_k \otimes \tau^*\cl_{-p} \to \Sym^{p-1} \cU_{k} \otimes \tau^*\cl_{-p+1} \to \cdots \\
				\cdots \to \cU_k \otimes \tau^*\cl_{-1} \to \tau^*\cl_0 \to \cO_{\OGr(2,\cQ)}(1) \otimes \tau^*\cl_k \to \cU_k^\vee(1) \otimes \tau^*\cl_{k+1} \to \cdots \\
				\cdots \to \Sym^p\cU_k^\vee(1) \otimes \tau^*\cl_{p+k} \to \Sym^{p+1} \cU_k^\vee(1) \otimes \tau^*\cl_{p+k+1} \to  \cdots 
			\end{multline}
%			\end{gathered}
%		\end{equation}
	\end{proposition}
	\begin{proof}
		Since all of the sheaves involved are vector bundles over $\OGr(k,\cQ)$, the exactness may be checked fiberwise on $\OGr(k,\cQ)$, in which case this reduces to checking exactness of a sequence of vector spaces. We fix a point $(U_k,s) \in \OGr(k,\cQ)$ consisting of a point $s \in S$ and a $k$-dimensional isotropic space $U_k$ for the quadratic form $Q := \cQ_s$ on $V := \cE|_s$. Restricted to $\kappa(s)$, the sheaves $\cl_i$ coincide with the even part of the Clifford algebra $\textcl_{\text{even}}$ of $Q$ when $i$ is even and the odd part of the Clifford algebra $\textcl_{\text{odd}}$ when $i$ is odd. As a consequence, the sequence becomes 
		\[\cdots \to \Sym^2 U_k \otimes \textcl_{\text{even}} \to U_k \otimes \textcl_{\text{odd}} \to \textcl_{\text{even}}\to \det(U_k)^\vee \otimes \textcl_{\parity(k)} \to U_k^\vee \otimes \det(U_k)^\vee \otimes \textcl_{\parity(k+1)}\to \cdots\]
		where $\textcl_{\text{even}}$ and $\textcl_{\text{odd}}$ denote respectively the even and odd parts of the $\bbZ/2$-graded Clifford algebra of $Q$. In fact, we will simplify our situation further by showing the exactness of the sequence 
		\[\cdots \to \Sym^2 U_k \otimes \textcl \to U_k \otimes \textcl \to \textcl \to \det(U_k)^\vee \otimes \textcl \to U_k^\vee \otimes \det(U_k)^\vee \otimes \textcl\to \cdots\]
		where we take the \textit{full} $\bbZ/2$-graded Clifford algebra. Observe that all the maps permute the $\bbZ/2$-grading of $\textcl$ in such a way that taking one of the graded pieces recovers our desired sequence. \par 
		We first fix a basis $e_1, e_2, \dots, e_k$ of $U_k$ and extend to a basis $e_1,\dots,e_n$ of $V$. Let $W= \mathrm{span}(e_{k+1},\dots,e_n)$. By the Poincar\'e-Birkhoff-Witt theorem for Clifford algebras (which is an easy consequence of the defining relations), we know that $\textcl$ has a basis given by terms of the form 
		\[ e_{i_1}e_{i_2}\cdots e_{i_\ell}\]
		for $1 \leq i_1 < i_2 < \cdots < i_\ell \leq n$. In particular, we can observe that as vector spaces we have an isomorphism 
		\begin{equation}
			\label{cliffwedgetrick}
			\begin{aligned}
				\textcl &\iso\bigwedge^*\left(U_k \oplus W \right)\simeq \bigwedge^* U_k \otimes \bigwedge^* W \\
				e_{i_1}e_{i_2} \cdots e_{i_\ell} &\mapsto e_{i_1} \wedge e_{i_2} \wedge \cdots e_{i_\ell},
			\end{aligned}
		\end{equation}
		though we stress that this isomorphism does not respect the algebra structure. However, since $U_k$ is isotropic, this map is easily checked to be a map of left $\bigwedge^* U_k \simeq \textcl(U_k)$-modules; equivalently, it respects left multiplication by elements of $U_k$. 
		\par 
		We then proceed by examining the cases.
		\begin{enumerate}[(a)]
			\item \underline{Exactness of }
			\[ \cdots \to \Sym^3 U_k \otimes \textcl \to \Sym^2 U_k \otimes \textcl \to U_k \otimes \textcl \to \textcl\]
			\underline{at all terms besides the rightmost:} \par 
			In fact, by exploiting the isomorphism of \eqref{cliffwedgetrick}, it is possible to break this chain complex into an infinite direct sum of transparently exact terms. Under the isomorphism of \eqref{cliffwedgetrick}, we see that the maps of our desired chain complex become 
			\begin{align*}
				\Sym^p U_k \otimes \bigwedge^* U_k \otimes \bigwedge^* W &\to \Sym^{p-1} U_k \otimes \bigwedge^* U_k \otimes \bigwedge^* W\\
				e_{i_1}\cdots e_{i_p} \otimes \alpha \otimes \beta &\mapsto \sum_{j=1}^p e_{i_1} \cdots \hat{e}_{i_j}\cdots e_{i_p} \otimes e_{i_j}\alpha \otimes \beta.
			\end{align*}
			As a consequence, we see that we can break the complex down as: 
			\begin{equation}
				\begin{tikzcd}
					& & & \Sym^0 U_k \otimes \bigwedge^0 U_k \otimes \bigwedge^* W \arrow[d, phantom, "\oplus"]\\
					& & \Sym^1 U_k \otimes \bigwedge^0 U_k \otimes \bigwedge^* W \arrow[r] \arrow[d, phantom, "\oplus"]& \Sym^0 U_k \otimes \bigwedge^1 U_k \otimes \bigwedge^* W \\
					& \Sym^2 U_k \otimes \bigwedge^0 U_k \otimes \bigwedge^* W \arrow[r] \arrow[d, phantom, "\oplus"]& \Sym^1 U_k \otimes \bigwedge^1 U_k \otimes \bigwedge^* W \arrow[r] & \Sym^0 U_k \otimes \bigwedge^2 U_k \otimes \bigwedge^* W \\
					\cdots \arrow[r] \arrow[d, phantom, "\oplus"]& \Sym^2 U_k \otimes \bigwedge^1 U_k \otimes \bigwedge^* W \arrow[r]& \Sym^1 U_k \otimes \bigwedge^2 U_k \otimes \bigwedge^* W  \arrow[r] & \Sym^0 U_k \otimes \bigwedge^3 U_k \otimes \bigwedge^* W \\
					\cdots \arrow[r] &\Sym^2 U_k \otimes \bigwedge^2 U_k \otimes \bigwedge^* W  \arrow[r] &\Sym^1 U_k \otimes \bigwedge^3 U_k \otimes \bigwedge^*W \arrow[r] & \Sym^0 U_k \otimes \bigwedge^4 U_k \otimes \bigwedge^* W \\
					\iddots & \vdots & \vdots & \vdots 
				\end{tikzcd}
			\end{equation}
			Then each row is, up to stupid truncation at the final term, the tensor by $\bigwedge^* W$ of the chain complex $\Sym^p(U_k \xrightarrow{\mathrm{id}} U_k)$, where we view $U_k \xrightarrow{\mathrm{id}} U_k$ as a complex in cohomological degrees $[0,1]$. But of course this two-term complex is exact, and the symmetric powers of an exact complex are exact. 
			\item \underline{Exactness of }
			\[U_k \otimes \textcl \to \textcl \to \det(U_k)^\vee \otimes \textcl \]
			\underline{at the middle term:} \par 
			This is easy to check explicitly. The image of the first map is given by terms of the form 
			\[e_1\xi_1 + e_2\xi_2 + \cdots + e_k\xi_k,\]
			but the second map, which is given by the 
			\[\xi \mapsto (e_1\wedge e_2 \wedge \cdots \wedge e_k \mapsto e_1e_2\cdots e_k\xi)\]
			for $\xi \in \textcl$, clearly vanishes if and only if when written in the Poincar\'e-Birkhoff-Witt basis, each term of $\xi$ contains $e_i$ for some $1 \leq i \leq k$. 
			\item \underline{Exactness of} \[\det(U_k) \otimes \textcl \to \textcl \to U_k^\vee \otimes \textcl\]
			\underline{at the middle term:} \par 
			This is essentially identical to the previous part: the image of the first map is given by $e_1e_2\cdots e_k\xi$ for $\xi \in \textcl$, while the second map is the map 
			\[\xi \mapsto (u \mapsto u\xi)\]
			for $u \in U_k$, and this homomorphism vanishes if and only if when written in the Poincar\'e-Birkhoff-Witt basis the terms of $\xi$ contain a basis of $U_k$. 
			\item \underline{Exactness of} 
			\[ \textcl \to U_k^\vee \otimes \textcl \to \Sym^2 U_k^\vee \otimes \textcl \to \Sym^3 U_k^\vee \otimes \textcl \to \cdots \]
			\underline{at all terms besides the leftmost:} \par 
			This is similar to case (a). Indeed, applying the isomorphism of \eqref{cliffwedgetrick} turns the terms of our complex into 
			\[\Sym^p U_k^\vee \otimes \bigwedge^* U_k \otimes \bigwedge^* W, \]
			and the same trick as before reduces this to showing the exactness of complexes of the form 
			\[\Sym^p U_k^\vee \otimes \bigwedge^0 U_k \to \Sym^{p+1} U_k^\vee \otimes \bigwedge^1 U_k \to \Sym^{p+2} U_k^\vee \otimes \bigwedge^2 U_k \to \cdots \to \Sym^{p+k} \cU_k^\vee \otimes \bigwedge^p U_k.\]
			On the other hand, the perfect pairing 
			\[\bigwedge^p U_k \otimes \bigwedge^{k-p} U_k \to \det U_k \simeq \bbC\]
			induces an identification of $\bigwedge^p U_2$ with $\bigwedge^{k-p} U_2^\vee$, and applying this isomorphism termwise produces an isomorphic chain complex of the form 
			\[\Sym^p U_k^\vee \otimes \bigwedge^k U_k^\vee \to \Sym^{p+1}U_k^\vee \otimes \bigwedge^{k-1} U_k^\vee \to \cdots \to \Sym^{p+k}U_2^\vee \otimes \bigwedge^0 U_k^\vee,\]
			given by the chain complex $\Sym^{p+k}(U_2^\vee \xrightarrow{\mathrm{id}} U_2^\vee)$ where $U_2^\vee \xrightarrow{\mathrm{id}} U_2^\vee$ is viewed as an exact complex supported in cohomological degrees $[-1,0]$. As before, this sequence is exact. 
		\end{enumerate}
	\end{proof}
	\begin{remark}
		One can make sense of all of the maps in the sequence on the ambient Grassmannian, rather than on the orthogonal Grassmannian. If one takes $k = 1$ and works on the ambient space $\Gr(1,\cE) \simeq \bbP(\cE)$, it is interesting to observe that the long exact sequence of Proposition \ref{LES} becomes a matrix factorization where two subsequent maps compose to multiplication by the defining section of the quadric, recovering the sequences (21) and (22) of \cite{quadfibs}. This matrix factorization also can be related to the matrix factorizations appearing in \cite{singularquadrics}. \par 
		One can also observe that in \cite{quadfibs}, the exact sequences which arise in the computation come from Koszul duality; it would be interesting to know whether the sequences arising here can be viewed through this lens. 
	\end{remark}
	Using the long exact sequence above, we define our Fourier-Mukai kernel as the right $\cl_0$-module given by 
	\begin{equation}
		\label{kerdef}
		\cF := \coker\left(\cU_k \otimes \tau^*\cl_{-1} \to \tau^*\cl_0\right) = \ker\left(\cO_{\OGr(k,\cQ)}(1) \otimes \tau^*\cl_k \to \cU_k^\vee(1) \otimes \tau^*\cl_{k+1}\right).
	\end{equation}
	It is then an immediate consequence of Proposition \ref{LES} above that there exist a left and right resolution for $\cF$ by right $\cl_0$-modules: 
	\begin{gather}
		\cdots \to \Sym^2 \cU_k \otimes \tau^*\cl_{-2}\to \cU_k \otimes \tau^*\cl_{-1}\to \tau^*\cl_0 \to \cF \to 0 \label{leftres} \\ 
		0 \to \cF \to \cO_{\OGr(k,\cQ)}(1) \otimes \tau^*\cl_k \to \cU_k^\vee(1) \otimes \tau^*\cl_{k+1} \to \Sym^2\cU_k^\vee(1) \otimes \tau^*\cl_{k+2} \to \cdots  \label{rightres}
	\end{gather}
	\begin{remark}
		Tensoring this kernel over $\cl_0$ with a particular left ideal sheaf of $\cl_0$ associated to a choice of isotropic subspace yields a reasonable notion of spinor sheaves on singular orthogonal Grassmannians; when $k = 1$, this recovers the construction of \cite{singularquadrics}. We discuss this connection further in the next section by analogy to the smooth case. \par 
		This kernel was also shown in \cite{2dimquadrics} to provide an embedding similar to Theorem \ref{embeddingtheorem} into $\OGr(2,\cQ)$ when the ambient vector space is of dimension $n = 4$. While our strategy is similar, the cohomological calculations differ greatly; in particular, Kuznetsov uses a resolution on the ambient Grassmannian itself which is special to his case. We expect a similar resolution exists in general, but avoid constructing one in our approach. 
	\end{remark}
	We now define our embedding functor as a functor from the (derived) category of left $\cl_0$-modules to the derived category of the total space of the fibration of orthogonal Grassmannians given by 
	\begin{equation} \label{functordefn}
		\begin{aligned}
			\Phi_\cF : \Db(S, \cl_0) &\to \on{D}^{-}(\OGr(k,\cQ)) \\ 
			\cG &\mapsto \cF \otimes^L_{\tau^*\cl_0} \tau^*\cG. 
		\end{aligned}
	\end{equation}
%	\textbf{\textcolor{red}{I need to show that $\cF$ is flat as a right module over $\tau^*\cl_0$ (or admits a finite length resolution by modules over $\tau^*\cl_0$), or else I don't think it's clear that this functor lands in $\Db(\OGr(k,\cQ))$!} Actually, I think it is MAYBE possible to use Kuznetsov's observation that my kernel is (almost) related to his by pushforward. $Rg_*f^*\cE_{0,0} \otimes_{\tau^*\cl_0} \tau^*\cG = Rg_*(f^*\cE_{0,0} \otimes_{g^*\tau^*\cl_0}g^*\tau^*\cG) = Rg_*(f^*(\cE_{0,0} \otimes_{p^*\cl_0} p^*\cG))$ maybe gives it to us? but it's not clear, since $f$ is not flat.} 
	This Fourier-Mukai kernel is particularly convenient to work with for our purposes for the following reason: 
	\begin{lemma}
		$\cF$ is a locally free $\cO_{\OGr(k,\cQ)}$-module. 
	\end{lemma}
	\begin{proof}
		Taking fibers reduces this to showing that $U_k \otimes \textcl_{\text{odd}} \to \textcl_{\text{even}}$ is constant rank, which follows from a straightforward calculation in terms of the Poincar\'e-Birkhoff-Witt basis. 
	\end{proof}
	In fact, it can be checked that $\cF$ is related to one of Kuznetsov's embedding functors in \cite{quadfibs} by $Rg_*f^*$ along the incidence correspondence: 
	\[\begin{tikzcd}
		&\bbP_{\OGr(k,\cQ)}(\cU_k) \arrow[ld, "f"'] \arrow[rd,"g"]& \\
		\cQ& & \OGr(k,\cQ)  
	\end{tikzcd}\]
	To be precise, if $p : \cQ\to S$ is the fibration, Kuznetsov constructs for each $a,b \in \bbZ$ a family of right $p^*\cl_0$-modules $\cE_{a,b}$ which coincide with twists of the sheaf $\cF$ above in the case $k = 1$. In particular, one can check that these $\cE_{a,b}$ admit the long right resolution 
	\[
			0 \to \cE_{a,b} \to \cO_{\cQ}(a+1) \otimes p^*\cl_{b+1} \to \cO_{\cQ}(a+2) \otimes p^*\cl_{b+2} \to \cdots
	\]
	When $a = -1, b = k - 1$, by using this right resolution it follows that we have an isomorphism of right $\tau^*\cl_0$-modules 
	\begin{equation}
		\cF \simeq (Rg_*f^*\cE_{-1,k-1})\otimes \cO_{\OGr(k,\cQ)}(1). 
	\end{equation}
	Using this isomorphism we check that the functor $\Phi_{\cF}$ lands in the bounded derived category $\Db(\OGr(k,\cQ))$. 
	\begin{proposition}
		For any complex $\cG \in \Db(S,\cl_0)$, the object $\cF \otimes^L_{\tau^*\cl_0} \tau^*\cG$ lives in bounded cohomological degrees. 
	\end{proposition}
	\begin{proof}
		It suffices to show that the twist $\cF(-1) \otimes^L_{\tau^*\cl_0} \tau^*\cG$ lives in bounded cohomological degrees. By the projection formula, we see that
		\begin{equation}
			\begin{aligned}
				\cF(-1) \otimes_{\tau^*\cl_0}^L \tau^*\cG &\simeq Rg_*f^*\cE_{-1,k-1} \otimes^L_{\tau^*\cl_0} \tau^*\cG \\
				&\simeq Rg_*(f^*\cE_{-1,k-1} \otimes^L_{g^*\tau^*\cl_0} g^*\tau^*\cG)\\
				&\simeq Rg_*(Lf^*(\cE_{-1,k-1} \otimes^L_{p^*\cl_0} p^*\cG)).
			\end{aligned}
		\end{equation}
		By \cite[Lemma 4.6]{quadfibs} the sheaf $\cE_{-1,k-1}$ is locally flat over $(S,\cl_0)$, and hence the tensor product over $p^*\cl_0$ is underived. \par 
		On the other hand, the morphism $f : \bbP_{\OGr(k,\cQ)}(\cU_k) \to \cQ$ is of finite Tor-dimension. We check this by factoring it as the composition $\bbP_{\OGr(k,\cQ)}(\cU_k) \hookrightarrow \Gr_{\cQ}(k-1,\cU_k/\cO_{\cQ}(-1)) \to \cQ$, where the first map is the closed immersion which sends the pair $(p \in \cQ, U_k \in \OGr(k,\cQ)) \in \bbP_{\OGr(k,\cQ)}(\cU_k)$ to the quotient $(U_k/\bbC v_p) \in \Gr_{\cQ}(k-1,\cU_k/\cO_\cQ(-1))$, where $\bbC v_p \subset V$ is the one-dimensional subspace corresponding to the point $p \in \cQ \subset \bbP(V)$. Clearly the second map is flat, so it suffices to observe that the closed immersion of finite Tor-dimension. Viewing $\Gr_{\cQ}(k-1,\cU_k/\cO_\cQ(-1))$ as the subscheme of the flag variety $\on{Fl}(1,k,V)$ where the 1-dimensional subspace is isotropic for $\cQ$, the variety $\bbP_{\OGr(k,\cQ)}(\cU_k)$ is cut out by the restriction of the quadratic form to the tautological $k$-dimensional subbundle: i.e., by the map of vector bundles
		\[\cL \to \Sym^2 \cE^\vee \to \Sym^2 \cU_k^\vee.\]
		However, this map of vector bundles vanishes when restricted to the one-dimensional subspace (or equivalently, when postcomposed with the map $\Sym^2 \cU_k^\vee \to \cO_{\cQ}(2)$), so it factors through a morphism
		\[\cL \to (\Sym^2 \cU_k^\vee)/\cO_{\cQ}(2).\]
		An easy rank calculation then shows that the codimension of $\bbP_{\OGr(k,\cQ)}(\cU_k)$ in $\Gr_{\cQ}(k-1,\cU_k/\cO_\cQ(-1))$ is precisely $\on{rk} \left(\cL^\vee \otimes (\Sym^2 \cU_k^\vee)/\cO_{\cQ}(2)\right)$, hence this closed immersion is the inclusion of a local complete intersection subvariety, hence it is of finite Tor-dimension. \par 
		Since derived pullbacks for morphisms of finite Tor-dimension and derived pushforwards for proper morphisms both preserve boundedness and coherence, the result follows. 
	\end{proof}
	Kuznetsov suggested to the author that, at least when $k = 2$, one might approach Theorem \ref{embeddingtheorem} by decomposing $f^*\cE_{-1,k-1}$ for some appropriate $i,j$ according to the semiorthogonal decomposition for the projective bundle $\bbP(\cU_2) \to \OGr(2,\cQ)$, and observing that these components are each up to a twist related to $\cF$. Regretfully, we were not able to see through this strategy. Instead we understand $\cF$ and its endomorphisms directly using the resolutions \eqref{leftres} and \eqref{rightres}. 
	\section{Schur functors and the Borel-Weil-Bott theorem} \label{bwbsection}
	When $Q$ is a singular quadric hypersurface, the orthogonal Grassmannian $\OGr(k,Q)$ will no longer be a homogeneous space for the Spin group. This adds additional subtleties in the cohomological calculations necessary for our proof. Nonetheless, by using Koszul resolutions, we will be able to calculate all of the requisite cohomology groups by using the Borel-Weil-Bott theorem on the ambient Grassmannians. There are many good sources for the representation theory we require, e.g. \cite{fultonharris, kuznetsovodd}, but we will briefly remind the reader of it. \par 
	For the remainder of the section, we fix the assumption that the quadrics arising in the quadric fibration all have rank $\geq 2k - 1$ so that Lemma \ref{coranklemma} applies. Recall also the running assumption $2 \leq k \leq n/2$. \par 
	Let $\mu = (\mu_1,\dots,\mu_k) \in \bbZ^k$ be a weakly decreasing sequence, i.e. $\mu_i \geq \mu_j$ for $i > j$. For any $k$-dimensional subspace $U$, we may interpret $\mu$ as a weight for the group $\GL(U)$ and write $\bbS^\lambda U$ for the irreducible representation of highest weight associated to $\mu$; this coincides with the Schur functor associated to $\mu$ when $\mu_i \geq 0$ for all $i$. Since $\bbS^{(\mu_1 + i, \mu_2 + i, \dots, \mu_k + i)} U = \bbS^\mu U \otimes (\det U)^{\otimes i}$, up to scaling by the determinant all irreducible representations of $\GL(U)$ are given by Schur functors. In some cases, these Schur functors have familiar names: 
	\begin{itemize}
		\item For $\mu = (p,0,\dots,0)$, $\bbS^{\mu} U \simeq \Sym^p U$, and 
		\item For $\mu = (\overbrace{1,\dots,1}^{p \text{ times}},0,\dots,0)$, $\bbS^\mu U \simeq \bigwedge^p U$. 
	\end{itemize}
	Highest weight representations and Schur functors also satisfy the following compatibility with dualization of vector bundles:
	\[(\bbS^\mu U)^\vee = \bbS^{(-\mu_k,-\mu_{k-1},\dots,-\mu_1)} U^\vee.\]
	\par 
	Perhaps the most important feature of these Schur functors is that their tensor products are controlled by the Littlewood-Richardson rule \cite{youngtablaeux}: 
	\begin{equation}
		\bbS^\mu U \otimes \bbS^\lambda U \simeq \bigoplus_\nu \bbS^\nu U^{\oplus c_{\mu,\lambda}^\nu}.
	\end{equation}
	We will also need to understand compositions of Schur functors $\bbS^\lambda \bbS^\mu U$. These also decompose as direct sums of Schur functors. While this is a complicated question in general, we will be able to extract the necessary information in our case in an ad hoc manner. \par 
	The first ingredient we need is a characterization of which irreducible subrepresentations occur in a specific representation. 
	\begin{lemma} \label{tensorprod}
		For any irreducible $\GL(U)$-subrepresentation 
		\[\bbS^\nu U \subset \bbS^{(a,b,0,\dots,0)} U \otimes \bigwedge^i \Sym^2 U,\]
		we have $\nu = (a + \alpha_1, b + \alpha_2, \alpha_3, \dots, \alpha_k)$ where $0 \leq \alpha_i \leq k + 1$, $\alpha_1 \leq i + 1$ and $\alpha_1 + \alpha_2 \leq i + 3$. 
	\end{lemma}
	\begin{proof}
		Let $\mu = (a,b,0,\dots,0)$. Evidently $\bbS^\nu U$ occurs as a subrepresentation of $\bbS^\mu U \otimes \bbS^\lambda U$ for some irreducible subrepresentation $\bbS^\lambda U \subset \bigwedge^i \Sym^2 U$. It follows immediately from the Littlewood-Richardson rule that we can write $\nu = (a + \alpha_1, b + \alpha_2, \alpha_3, \dots, \alpha_k)$ for $0 \leq \alpha_i$, since the coefficient $c_{\mu,\lambda}^\nu \neq 0$ only if the Young diagrams associated to $\mu$ and $\lambda$ lie in the Young diagram associated to $\nu$. \par 
		To find that $\alpha_i \leq k + 1$, we first recall that there is a perfect pairing 
		\[\bigwedge^i \Sym^2 U \otimes \bigwedge^{k(k+1)/2-i}\Sym^2 U \to \det \Sym^2 U,\]
		which implies that the map 
		\[\left(\bbS^\mu U \otimes \bigwedge^i \Sym^2 U\right) \otimes \bigwedge^{k(k+1)/2 - i}\Sym^2 U \to \bbS^\mu U \otimes \det \Sym^2 U\]
		is also a perfect pairing. In particular, the composition 
		\[\bbS^\nu U \otimes \bigwedge^{k(k+1)/2-i} \Sym^2 U\hookrightarrow \left(\bbS^\mu U \otimes \bigwedge^i \Sym^2 U\right) \otimes \bigwedge^{k(k+1)/2 - i}\Sym^2 U\to \bbS^\mu U \otimes \det \Sym^2 U \]
		cannot be zero. On the other hand, an easy calculation shows that $\det \Sym^2 U \simeq (\det U)^{\otimes(k+1)}$, so in particular
		\[\bbS^\mu U \otimes \det \Sym^2 U \simeq \bbS^{(a + k + 1, b + k + 1, k+1, \dots, k+1)} U.\]
		By Schur's lemma and the Littlewood-Richardson rule, it follows that there is some $\xi$ such that the coefficient $c_{\nu,\xi}^{(a+k+1,b + k + 1, k + 1, \dots, k + 1)}$ is nonzero, and hence the Young diagram for $\nu$ must be a subdiagram of the one for $(a + k + 1, b + k + 1, k + 1, \dots , k + 1)$. In particular, this implies that $\alpha_i \leq k + 1$ for all $i$. 
		\par 
		To check the remaining inequalities, note that $\nu$ must be a weight of $\bbS^{(a,b,0,\dots,0)} U\otimes \bigwedge^i \Sym^2 U$. However, every weight of this tensor product arise as a sum $\beta + \alpha$ where $\beta$ is a weight of $\bbS^{(a,b,0,\dots,0)}$ and $\alpha$ is a weight of $\bigwedge^i \Sym^2 U$. It then suffices to show that $\beta_1 \leq a$, $\beta_1 + \beta_2 \leq a + b$, $\alpha_1 \leq i + 1$ and $\alpha_1 + \alpha_2 \leq i + 3$ to prove the claim.\par 
		We first analyze the weights of $\bbS^{(a,b,0,\dots,0)} U$. Pick $e_1, \dots, e_k$ to be a basis of weight vectors for $U$, with $e_1$ in the weight space of $\epsilon_1 = (1,0,\dots,0)$, $e_2$ in the weight space of $\epsilon_2 = (0,1,0,\dots,0)$ and so on. Since with respect to the usual ordering on weights, $\epsilon_i > \epsilon_j$ for $i < j$, the highest weight of this representation will have maximal first entry; in particular, any weight $\beta$ of this representation must have $\beta_1 \leq a$. Similarly, the sum of the first two entries for the highest weight will be maximal, so any weight $\beta$ must have $\beta_1 + \beta_2 \leq a + b$. \par 
		Next, we analyze the weights of $\bigwedge^i \Sym^2 U$. The weight vectors of $\bigwedge^i \Sym^2 U$ are all of the form 
		\[e_{j_1}e_{j_2} \wedge e_{j_3}e_{j_4} \wedge \dots \wedge e_{j_{2i-1}}e_{j_{2i}}.\]
		The first entry of the weight for this weight vector will be as large as possible when the weight vector has as many copies of $e_1$ as possible, e.g. for the vector
		\[e_1^2 \wedge e_1e_2 \wedge e_1e_3 \wedge \cdots \wedge e_1e_i.\]
		Since $e_1$ occurs $i + 1$ times in this vector, it follows that any weight $\alpha$ of $\bigwedge^i \Sym^2 U$ must have $\alpha_1 \leq i + 1$. Similarly, the sum of the first two entries of a weight are largest when the weight vector has as many copies of $e_1$ or $e_2$ as possible, e.g. for the vector 
		\[e_1^2 \wedge e_2^2 \wedge e_1e_2 \wedge e_1e_3 \wedge \dots \wedge e_1e_{i-1}.\]
		In this weight vector, $e_1$ and $e_2$ occur a combined total of $6 + (i - 3) = i + 3$ times, so for any weight $\alpha$ of this representation we must have $\alpha_1 + \alpha_2 \leq i + 3$. 
	\end{proof}
	When we use this result to calculate cohomology, we will need to apply the Borel-Weil-Bott theorem. We recall this theorem here. 
	\begin{theorem}[Borel-Weil-Bott]
		Let $\mu = (\mu_1,\dots,\mu_k) \in \bbZ^k$ and $\nu = (\nu_1,\dots,\nu_{n-k}) \in \bbZ^{n-k}$, and let $(\mu,\nu) \in \bbZ^n$ denote their concatenation. Write $\rho = (n,n-1,\dots,1)$. Write $\cU$ for the tautological subbundle on $\Gr(k,n)$, and $\cQ$ for the tautological subbundle. If all of the entries of $(\mu,\nu) + \rho$ are distinct, then let $w \in S_n$ denote the unique permutation which makes it decreasing, and note that we have 
		\[R\Gamma(\bbS^\mu \cU^\vee \otimes \bbS^\nu \cQ^\vee) = \bbS^{w((\mu,\nu)+\rho) - \rho}V^\vee [-\ell(w)]\]
		where $\ell(w)$ is the length of the permutation. If not all entries of $(\mu,\nu) + \rho$ are distinct, then $\bbS^\mu \cU^\vee \otimes \bbS^\nu \cQ^\vee$ has no cohomology. 
 	\end{theorem}
 	With these two results in hand, we are in a good place to prove the necessary cohomological vanishings on the relative orthogonal Grassmannian, which we will log for future use. The proofs of the following lemmas will follow from an explicit (but fairly technical) analysis of the Schur functors arising in Koszul resolutions for tautological vector bundles, paired with the Borel-Weil-Bott theorem above. 
 	\begin{lemma}\label{vanishinglemma1}
 		On $\OGr(k,\cQ)$, for any $1 \leq t < n - 2(k-1)$, we have 
 		\begin{equation}
 			R^i\tau_*(\Sym^p \cU_k \otimes \Sym^q \cU_k(-t)) = 0
 		\end{equation}
 		for all $i > p + q + 2t - 1$. 
 	\end{lemma}
 	\begin{proof}
 		For the sake of clarity in this proof, we write $\cU_k$ for the tautological subbundle on $\Gr(k,\cE)$ or its fibers, and $\cU_k|_{\OGr}$ for its restriction to the orthogonal Grassmannian or its fibers. 
 		\par
 		By Grothendieck's theorem on cohomology and base change, it suffices to show the vanishing of these cohomologies on all fibers, so we may fix a point $s \in S$ and consider the quadric $Q := \cQ_s$ on $V := \cE|_s$ and instead show the vanishing of 
 		\[H^i(\Sym^p \cU_k \otimes \Sym^q \cU_k|_{\OGr}(-t)) = 0\]
 		for $i > p + q + 2t- 1$. By the Pieri rule, there is a decomposition 
 		\[\Sym^p \cU_k \otimes \Sym^q \cU_k \simeq \bigoplus_{\max(p,q) \leq a \leq p + q}\bbS^{(a,p+q-a,0,\dots,0)}\cU_k,\]
 		so it is enough to instead show the vanishing 
 		\[H^i(\bbS^{(a,b,0,\dots,0)}\cU_k|_{\OGr}(-t)) = 0\]
 		for $i > a + b + 2t-1$, where $a, b \geq 0$. \par 
 		Since the fiber $\OGr(k,Q)$ is the zero locus of a regular section of the bundle $\Sym^2 \cU_k^\vee$, we may resolve $\bbS^{(a,b,0,\dots,0)}\cU_k|_{\OGr}(-t)$ on $\Gr(k,V)$ by tensoring the bundle $\bbS^{(a,b,0,\dots,0)}\cU_k$ with the Koszul resolution 
 		\[0 \to \det \Sym^2 \cU_k \to \cdots \to \bigwedge^2 \Sym^2 \cU_k \to \Sym^2 \cU_k \to \cO_{\Gr(k,V)} \to \cO_{\OGr(k,Q)} \to 0.\]
 		This resolves $\bbS^{(a,b,0,\dots,0)}\cU_k|_{\OGr}(-t)$ by 
 		\begin{multline*}
 			0 \to \bbS^{(a,b,0,\dots,0)}\cU_k \otimes \det \Sym^2 \cU_k(-t) \to \bbS^{(a,b,0,\dots,0)}\cU_k \otimes \bigwedge^{k(k+1)/2 - 1} \Sym^2 \cU_k(-t) \to \dots \\
 			\dots \bbS^{(a,b,0,\dots,0)}\cU_k \otimes \bigwedge^2 \Sym^2 \cU_k(-t) \to \bbS^{(a,b,0,\dots,0)}\cU_k \otimes \Sym^2 \cU_k(-t) \to  \bbS^{(a,b,0,\dots,0)}\cU_k(-t). 
 		\end{multline*}
 		By using the hypercohomology spectral sequence, if we can show that 
 		\[H^i(\bbS^{(a,b,0,\dots,0)}\cU_k \otimes \bigwedge^j \Sym^2\cU_k(-t)) = 0\]
 		for all $i > a + b + 2t + j - 1$, then the lemma will be proved. Using Lemma \ref{tensorprod}, this reduces to showing that 
 		\[H^i(\bbS^\nu \cU_k(-t)) = 0\]
 		when $\nu = (a + \alpha_1, b + \alpha_2, \alpha_3, \dots, \alpha_k)$ where $0 \leq \alpha_1,\alpha_2,\cdots,\alpha_k \leq k + 1$, $\alpha_1 \leq j + 1$ and $\alpha_1 + \alpha_2 \leq j + 3$. Since $\nu$ induces a highest weight representation, the sequence of integers defining $\nu$ must be decreasing. Using the compatibility of highest weight representations with dualization of vector bundles and tensors with the determinant, there is an isomorphism 
 		\[\bbS^\nu \cU_k(-t) = \bbS^{(-\alpha_k - t, \dots, -\alpha_3 - t, -b - \alpha_2 - t, -a - \alpha_1 - t)} \cU_k.\]
 		By the Borel-Weil-Bott theorem, we consider the sequence of numbers given by adding $\rho = (n, n-1, \dots, 1)$ to $(-\alpha_k - t, \dots , -b - \alpha_2 - t, - a - \alpha_1 - t, 0, \dots , 0)$: 
 		\[(n - \alpha_k - t,\dots, n - k + 3 - \alpha_3 - t,n - k + 2 - b - \alpha_2 - t,n - k + 1 - a - \alpha_1 - t, n-k, \dots, 1).\]
 		Evidently, the first $k$ terms are strictly decreasing and the last $n - k$ terms are strictly decreasing. Moreover, since $\alpha_3 \leq k + 1$, it follows that $n - k + 3 - \alpha_3 - t \geq n - k + 3 - k - 1 - t \geq n - 2(k-1) - t > 0$, so the first $k - 2$ terms are always positive. Using Borel-Weil-Bott, it follows that the bundle $\bbS^\nu \cU_k(-t)$ on $\Gr(k,V)$ can only have higher cohomology in degrees $n - k$ or $2(n-k)$. In particular, we see that the bundle $\bbS^\nu \cU_k(-t)$ can have cohomology in degree $n - k$ only if 
 		\[n - k + 2 - b - \alpha_2 - t > n - k > 1 > n - k + 1 - a - \alpha_1 - t\]
 		in which case $n - k < a + \alpha_1 + t \leq a + j + 1 + t$; but since all the values are integral, this implies that the cohomology occurs in degree $n - k \leq a + j + t \leq a + b + j + 2t - 1$. Similarly, there can be cohomology in degree $2(n-k)$ only if 
 		\[1 > n - k + 2 - b - \alpha_2 - 1 > n - k + 1 - a - \alpha_1 - 1.\]
 		Using again that everything is integer-valued, this implies that 
 		\[n - k \leq b + \alpha_2 + t  - 2\]
 		and 
 		\[n - k \leq a + \alpha_1 + t - 2,\]
 		hence the cohomology occurs in degree $2(n-k) \leq a + b + \alpha_1 + \alpha_2 + 2t - 4\leq a + b + j + 2t - 1$. This proves the lemma.
 	\end{proof}
 	We will specialize to the case where $t = 1$ in order to prove fully faithfulness of the main embedding functor, and allow for higher values of $t$ in the proofs for semiorthogonality. \par 
 	In the dual situation, we have the following. 
 	\begin{lemma}\label{vanishinglemma2}
 		On $\OGr(k,\cQ)$, for any (possibly negative) $t < n -  2k$, we have 
 		\begin{equation} 
 			R^i\tau_*\left(\Sym^p \cU_k^\vee \otimes \Sym^q \cU_k^\vee(-t)\right)  = 0
 		\end{equation}
 		for all $i > 0$. If in addition $\min(p,q)  < t$, then vanishing holds also for $i = 0$. \par 
 		Moreover, we have isomorphisms 
 		\begin{align}
 			\tau_*(\cO_{\OGr(k,\cQ)}(1)) &\simeq \bigwedge^k \cE^\vee, \label{O1calc}\\
 			\tau_*(\cU^\vee(1)) &\simeq \bbS^{(2,\overbrace{\scriptstyle{1, \dots, 1}}^{k - 1 \text{ times}})} \cE^\vee / \im\left(\bigwedge^{k-1}\cE^\vee \otimes \tau^*\cL \to \bigwedge^{k-1} \cE^\vee \otimes \Sym^2 \cE^\vee \to \bbS^{(2,1,\dots,1)} \cE^\vee\right). \label{U1calc}
 		\end{align}
 	\end{lemma}
 	\begin{proof}
 		For the first part of the lemma, the strategy is roughly the same as the previous Lemma \ref{vanishinglemma1}: we pass to a fiber, taking $V = \cE|_s$ and $Q = \cQ_s$, and check vanishing by examining the Koszul resolution of $\Sym^p\cU_k^\vee \otimes \Sym^q \cU_k^\vee(-t)|_{\OGr}$. In fact, we will show all terms in the Koszul sequence have no higher cohomology. Decomposing the tensor product, we again reduce to showing $\bbS^{(a,b,0,\dots,0)} \cU_k^\vee(-t) \otimes \bigwedge^i \Sym^2\cU_k$ has no higher cohomology, where $a + b = p + q$ and $0 \leq b \leq \min(p,q)$. This bundle can be rewritten as 
 		\begin{align*}
 			\bbS^{(a,b,0,\dots,0)} \cU_k^\vee(-t) \otimes \bigwedge^i \Sym^2\cU_k &\simeq \bbS^{(a,b,0,\dots,0)} \cU_k^\vee(-t) \otimes \bigwedge^{k(k+1)/2-i} \Sym^2 \cU_k^\vee \otimes \det \Sym^2 \cU_k \\
 			&\simeq \bbS^{(a,b,0,\dots,0)} \cU_k^\vee \otimes \bigwedge^{k(k+1)/2-i} \Sym^2 \cU_k^\vee (-t-k-1).
 		\end{align*}
 		Using again Lemma \ref{tensorprod}, the sheaves on $\Gr(k,V)$ arising in the spectral sequence whose cohomology we need to compute are of the form $\bbS^\nu \cU_k^\vee$ where $\nu = (a + \alpha_1 - t - k - 1, b + \alpha_2 - t - k - 1 , \alpha_3 - t - k - 1, \dots, \alpha_k - t - k - 1)$. By the Borel-Weil-Bott theorem, to calculate the cohomology of this equivariant vector bundle on $\Gr(k,V)$, we consider the sequence of numbers 
 		\[(n + a + \alpha_1 - t-k-1, n - 1 + b + \alpha_2 - t - k - 1, \dots, n-k+1+\alpha_k - t- k - 1, n - k, \dots, 1).\] The bundles all have no higher cohomology as long as $n - k + 1 + \alpha_k - t - k - 1 = n - 2k - t + \alpha_k > 0$. But this follows from our assumption that $t < n - 2k$. \par
 		If in addition $b \leq \min(p,q) < t$, we have the inequality 
 		\[n - 1 + b + \alpha_2 - t - k - 1 \leq n - 1 + b - t < n - 1,\]
 		where we recall that $\alpha_2 \leq k + 1$ by Lemma \ref{tensorprod}. Hence the $n - 1$ numbers 
 		\[n - 1 + b + \alpha_2 - t - k - 1, \dots, n-k+1+\alpha_k-t-k-1,n-k,\dots,1\]
 		all lie in the interval $[0,n-2]$, and by the pigeonhole principle cannot all be distinct, so the Borel-Weil-Bott theorem allows us to conclude that the associated vector bundle has no cohomology. 
 		\par  
 		For the calculation of $\tau_*(\cO_{\OGr(k,\cQ)}(1))$, observe first that the pushforward of $\cO_{\Gr(k,\cE)}(1)$ to the base is canonically isomorphic to $\bigwedge^k \cE^\vee$, so that restriction gives a canonical map 
 		\[\bigwedge^k \cE^\vee \to \tau_*(\cO_{\OGr(k,\cQ)}(1)).\]
 		Since all terms in the Koszul resolution have no higher cohomology, to check that this map is an isomorphism we can simply pass to a fiber $s \in S$ and check that 
 		\[H^0(\Sym^2 \cU_k(1)) = 0,\]
 		because the sequence 
 		\[H^0(\Sym^2 \cU_k(1)) \to H^0(\cO_{\Gr(k,V)}(1)) \to H^0(\cO_{\OGr(k,\cQ)}(1)) \to 0\]
 		is exact. But one can see by writing 
 		\[\Sym^2 \cU_k(1) \simeq \bbS^{(1,\dots,1,-1)}\cU_k^\vee\]
 		that the Borel-Weil-Bott theorem immediately implies that this vector bundle has no global sections. \par 
 		For $\tau_*(\cU^\vee(1))$, the argument is similar. We can calculate that the pushforward of $\cU_k^\vee(1) = \bbS^{(2,1,\dots,1)}\cU_k^\vee$ on the ambient Grassmannian to the base is isomorphic to $\bbS^{(2,1,\dots,1)}\cE^\vee$, and therefore restriction gives a map
 		\[\bbS^{(2,1,\dots,1)}\cE^\vee \to \tau_*(\cU_k^\vee(1)).\]
 		If we consider the Koszul resolution for the total space of the orthogonal Grassmannian fibration $\OGr(k,\cQ)$ inside of the Grassmannian bundle $\Gr(k,\cE)$, then our previous arguments imply that the map above is surjective, and the kernel is computed by the image of the pushforward of the map 
 		\[\cU_k^\vee(1) \otimes \tau^*\cL \otimes \Sym^2 \cU_k \to \cU_k^\vee(1), \]
 		which is just multiplication with the defining section of the orthogonal Grassmannian. This map factors as 
 		\[\cU_k^\vee(1) \otimes \tau^*\cL \otimes \Sym^2 \cU_k \to \cU_k^\vee(1) \otimes \Sym^2 \cE^\vee \otimes \Sym^2 \cU_k \to \cU_k^\vee(1) \otimes \Sym^2 \cU_k^\vee \otimes \Sym^2 \cU_k\to \cU_k^\vee(1),\]
 		where the first map comes from the quadratic form $\tau^*\cL\to \Sym^2 \cE^\vee$, the second comes from restriction from $\cE$ to $\cU_k$, and the third comes from the evaluation map $\Sym^2 \cU_k^\vee \otimes \Sym^2 \cU_k \to \cO_{\Gr(k,\cE)}$. We can then check that there is an identification \[\cU_k^\vee(1) \otimes \Sym^2 \cU_k \simeq \bbS^{(2,1,\dots,1,-1)}\cU_k^\vee \oplus \bbS^{(1,\dots,1,0)}\cU_k^\vee,\]
 		and the Borel-Weil-Bott theorem applied on the Grassmannian fibers implies that the pushforward of the first summand to $S$ is trivial, while the pushforward of the second is precisely $\bigwedge^{k-1} \cE^\vee$. Pushing forward this sequence of maps, we get an identification of $\tau_*(\cU_k^\vee(1))$ with the image of 
 		\[\bigwedge^{k-1} \cE^\vee \otimes \tau^*\cL \to \bigwedge^{k-1} \cE^\vee \otimes \Sym^2 \cE^\vee \to \bbS^{(2,1,\dots,1)} \cE^\vee,\]
 		proving the claim. 
 	\end{proof}
 	To prove fully faithfulness for the main embedding functor we will specialize to $t = -1$, while the cases $t > -1$ will be useful in proving semiorthogonality. \par 
 	When $k = 2$, tensor products and plethysms of Schur functors $\bbS^{\lambda} \cU_2$ are much simpler, and as a consequence, we are able to prove stronger vanishing results. We do so, case by case, in the following series of lemmas. 
 	\begin{lemma}\label{vanishinglemma3}
 		On $\OGr(2,\cQ)$, for any $0 \leq t \leq n - 4$, we have 
 		\begin{equation}
 			R^i\tau_*(\Sym^p \cU_2 \otimes \Sym^q \cU_2^\vee(-t)) = 0
 		\end{equation}
 		for all $i > p + t$. If in addition $p < n - 4 - t$, then $R^i\tau_*(\Sym^p \cU_2 \otimes \Sym^q \cU_2^\vee(-t)) = 0$ for all $i > 0$.
 	\end{lemma}
 	\begin{proof}
 		The proof is similar to Lemma \ref{vanishinglemma1}, so we will immediately reduce to calculating the cohomology of this bundle over the fiber of some $s \in S$. Since $\cU_2$ is rank 2, there is a natural isomorphism \[\cU_2 \simeq \cU_2^\vee(-1),\]
 		so in particular we may rewrite $\Sym^p \cU_2 \simeq \Sym^p \cU_2^\vee (-p)$. The first part of the lemma then reduces to showing that 
 		\[H^i(\Sym^p \cU_2^\vee \otimes \Sym^q \cU_2^\vee(-t-p)|_{\OGr}) = 0.\]
 		for all $i > p + t$. Using the Pieri rule, we have a decomposition 
 		\[\Sym^p \cU_2^\vee \otimes \Sym^q \cU_2^\vee(-t-p)\simeq \bigoplus_{\max(p,q) \leq a \leq p + q}\bbS^{(a-t-p,q-a-t)}\cU_2^\vee,\]
 		so we may check the vanishing of cohomology for each summand. We then resolve on the Grassmannian by tensoring with the Koszul resolution
 		\[0 \to \det \Sym^2 \cU_2 \simeq \cO(-3)\to \bigwedge^2 \Sym^2 \cU_2 \simeq \Sym^2 \cU_2(-1) \to \Sym^2 \cU_2 \to \cO_{\Gr(k,V)}\to \cO_{\OGr(k,Q)} \to 0.\]
 		By applying the Pieri rule after tensoring the Koszul resolution with $\bbS^{(a-t-p,q-a-t)}\cU_2^\vee$, we see that the bundle $\bbS^{(a-t-p,q-a-t)}\cU_2^\vee|_{\OGr}$ is resolved by 
 		\[0 \to \bbS^{(a-t-p-3,q-a-t-3)}\cU_2^\vee \to \substack{\bbS^{(a-t-p-3,q-a-t-1)}\cU_2^\vee \\ \oplus \\ \bbS^{(a-t-p-2,q-a-t-2)}\cU_2^\vee \\ \oplus \\ \bbS^{(a-t-p-1,q-a-t-3)}\cU_2^\vee} \to \substack{\bbS^{(a-t-p-2,q-a-t)}\cU_2^\vee \\ \oplus \\ \bbS^{(a-t-p-1,q-a-t-1)}\cU_2^\vee \\ \oplus \\ \bbS^{(a-t-p,q-a-t-2)}\cU_2^\vee} \to \bbS^{(a-t-p,q-a-t)}\cU_2^\vee,\]
 		where for the moment we let $\bbS^{(\lambda_1,\lambda_2)} \cU_2^\vee = 0$ whenever $\lambda_1 < \lambda_2$. \par 
 		By the hypercohomology spectral sequence, the claim will follow if we show that 
 		\[H^i(\bbS^{(a-t-p,q-a-t)}\cU_2^\vee \otimes \bigwedge^j \Sym^2 \cU_2) = 0\]
 		whenever $i > p + t + j$, so we analyze the summands of the terms of the Koszul resolution using the Borel-Weil-Bott theorem once more. To understand the cohomology for any equivariant vector bundle of the form $\bbS^{(\lambda_1,\lambda_2)}\cU_2^\vee$ on the Grassmannian, we need only examine the sequence of integers 
 		\[\lambda_1 + n, \lambda_2 + n - 1, n - 2, \dots, 1.\]
 		By Borel-Weil-Bott, such an equivariant vector bundle can have higher cohomology only in degree $n - 2$, if $ \lambda_1 + n > n - 2 > 1 > \lambda_2 + n - 1$, or in degree $2n - 4$, if $1 > \lambda_1 + n$. In fact the second case cannot happen, since for all Schur functors arising in the Koszul resolution we have 
 		\[\lambda_1 + n \geq a - t - p - 3 + n \geq a - p + 1 \geq 1,\]
 		since $t \leq n - 4$ and $a - p \geq 0$. \par 
 		Suppose on the other hand that a Schur functor arising in the degree $j$ term of the Koszul resolution has cohomology in degree $n - 2$. By inspection of the sequence above, we see that $\lambda_2 \geq q - a - t - j - 1$, so in particular we have the inequality 
 		\[0 \geq \lambda_2 + n - 1 \geq q - a - t - j - 1 + n - 1,\]
 		or equivalently 
 		\[n - 2 \leq a - q + t + j \leq p + t + j,\]
 		as $a \leq p + q$. It follows that the cohomology of this Schur functor arises only in degrees $\leq p + t + j$, proving the first part of the lemma. \par 
 		If we also assume that $p < n - 4 - t$, then 
 		\[\lambda_2 + n - 1 \geq q - a - t - 3 + n - 1 \geq p - t + n - 4 > 0,\]
 		so each term of the Koszul resolution has no higher cohomology, finishing the proof of the lemma. 
 	\end{proof}
 	Under similar but slightly different hypotheses, we can show that the pushforwards by $\tau$ completely vanish. The following proposition should be thought of as essentially an analogue of \cite[Lemma 3.11]{kuznetsoveven} in the relative, potentially singular setting.
 	\begin{lemma}\label{vanishinglemma4}
 		Suppose that $n > 4$. On $\OGr(2,\cQ)$, for any $0 \leq t \leq n - 4$ and any $0 \leq p,q \leq n/2 - 2$, we have 
 		\begin{equation}
 			R\tau_*(\Sym^p \cU_2 \otimes \Sym^q \cU_2^\vee(-t)) = 0,
 		\end{equation}
 		unless $t = 0$ and $p \leq q$ or $p = q = n/2 - 2$ and $t = n/2 - 2$ or $n/2 - 1$. 
 	\end{lemma}
 	\begin{proof}
 		As before, we may pass to a fiber and instead calculate 
 		\[R\Gamma(\Sym^p \cU_2 \otimes \Sym^q \cU_2^\vee(-t)|_{\OGr}) = R\Hom(\Sym^p \cU_2^\vee(t)|_{\OGr}, \Sym^q \cU_2^\vee|_{\OGr})\]
 		on a single orthogonal Grassmannian $\OGr(k,Q)$. By the calculation of the dualizing line bundle \eqref{dualizing}, there is an isomorphism 
 		\[R\Hom(\Sym^p \cU_2^\vee(t)|_{\OGr}, \Sym^q \cU_2^\vee|_{\OGr}) \simeq R\Hom(\Sym^q \cU_2^\vee(n-3-t)|_{\OGr}, \Sym^p \cU_2^\vee|_{\OGr})^\vee,\]
 		so to check the vanishing we may assume that $0 \leq t < n/2 - 1$. \par 
 		By resolving with the Koszul resolution and using the hypercohomology spectral sequence, we again reduce to understanding the cohomology of the bundles in the sequence 
 		\[0 \to \bbS^{(a-t-p-3,q-a-t-3)}\cU_2^\vee \to \substack{\bbS^{(a-t-p-3,q-a-t-1)}\cU_2^\vee \\ \oplus \\ \bbS^{(a-t-p-2,q-a-t-2)}\cU_2^\vee \\ \oplus \\ \bbS^{(a-t-p-1,q-a-t-3)}\cU_2^\vee} \to \substack{\bbS^{(a-t-p-2,q-a-t)}\cU_2^\vee \\ \oplus \\ \bbS^{(a-t-p-1,q-a-t-1)}\cU_2^\vee \\ \oplus \\ \bbS^{(a-t-p,q-a-t-2)}\cU_2^\vee} \to \bbS^{(a-t-p,q-a-t)}\cU_2^\vee,\]
 		where $\max(p,q) \leq a \leq p + q$. In fact, we will show that every such bundle has no cohomology under the conditions above. \par 
 		Suppose first that $1 \leq t < n/2 - 2$. Taking $\bbS^{(\lambda_1,\lambda_2)} \cU_2^\vee$ above where $q - a - t \geq \lambda_2 \geq q - a - t - 3$, 
 		we observe that 
 		\[n - 1 + \lambda_2 \geq n - 1 + q - a - t - 3 \geq n - p - t - 4 > n - n/2+2 - n/2+2 - 4 = 0,\]
 		as well as 
 		\[n - 1 + \lambda_2 \leq n - 1 + q - a - t \leq n - 1  - t \leq n - 2,\]
 		so the sequence of integers 
 		\[n + \lambda_1, n - 1 + \lambda_2, n - 2, \dots, 1\]
 		are not all distinct, and $\bbS^{(\lambda_1,\lambda_2)}\cU_2^\vee$ has no cohomology by Borel-Weil-Bott. \par 
 		Now suppose that $t = 0$ and $p > q$. In the Koszul resolution above, with $\max(p,q) \leq a \leq p + q$, it follows that the inequality $q - a < 0$ must be strict, so for any $\bbS^{(\lambda_1,\lambda_2)}\cU_2^\vee$ as above where $q - a \geq \lambda_2 \geq q - a - 3$, we have 
 		\[n - 1 + \lambda_2 \geq n - 1 + q - a - 3\geq n - p - 4 \geq n - n/2 + 2 - 4 > 0\]
 		as well as 
 		\[n - 1 + \lambda_2 \leq n - 1 + q - a < n - 1,\]
 		or equivalently $n - 1 + \lambda_2 \leq n - 2$, so Borel-Weil-Bott theorem again guarantees any such vector bundle is acyclic in the Koszul resolution. \par 
 		In the final case for the computation of acyclicity of these vector bundles, we may assume that the inequality $n/2 - 2 \leq t < n/2 - 1$ holds and one of $p, q$ is strictly less than $n/ 2 - 2$; hence $a \leq p + q < n - 4$. Again taking $\bbS^{(\lambda_1,\lambda_2)}\cU_2^\vee$ where $q - a - t \geq \lambda_2 \geq q - a - t - 3$, observe first that 
 		\[n - 1 + \lambda_2 \leq n - 1 + q - a - t \leq n - 1 - t \leq n - 2,\]
 		but we only have 
 		\[n - 1 + \lambda_2 \geq n - 1 + q - a - t - 3 \geq n - p - t - 4 > n - n/2+2-n/2+1 - 4 = -1,\]
 		so this vector bundle may have cohomology if and only if $n - 1 + \lambda_2 = 0$. By the chain of inequalities above, this forces $\lambda_2 = q - a - t - 3$ as well as $p + t = n - 4$. By inspection of the terms in the Koszul resolution, if $\lambda_2 = q - a - t - 3$, then $a - t - p - 3\leq \lambda_1 \leq a - t - p - 1$, so we may observe
 		\[n + \lambda_1 \leq n + a - t - p - 1 \leq a + 3 < n - 1\]
 		or equivalently $n + \lambda_1 \leq n - 2$, as well as 
 		\[n + \lambda_1 \geq n + a - t - p - 3 \geq n - t - 3 > n - n/2 + 1 - 3 = n/2-2 > 0,\]
 		proving by Borel-Weil-Bott the acyclicity of any vector bundle arising in the Koszul resolution under the conditions of the lemma.
 		
 	\end{proof}
 	The previous vanishing lemmas will reduce all of the necessary calculations for Theorem \ref{sodtheorem} to a finite collection of explicit calculations, which will need as input the following isomorphisms. 
 	\begin{lemma}\label{vanishinglemma5}
 		On $\OGr(2,\cQ)$, we have the following isomorphismsm, where $0 \leq t \leq q$ and $t < n - 4$: 
 		\begin{align}
 			\tau_*(\Sym^t \cU_2^\vee \otimes \Sym^q \cU_2^\vee(-t)) &\simeq \tau_*(\Sym^{q-t} \cU_2^\vee) \label{vl5e1}\\
 			&\simeq \Sym^{q-t} \cE^\vee / \im(\Sym^{q-t-2}\cE^\vee \otimes \tau^*\cL \to \Sym^{q-t} \cE^\vee)\nonumber \\
 			\tau_*(\Sym^{t+1} \cU_2^\vee \otimes \Sym^q \cU_2^\vee(-t)) &\simeq \tau_*(\cU_2^\vee \otimes \Sym^{q-t}\cU_2^\vee ) \label{vl5e2}\\
 			&\simeq \cE^\vee \otimes \Sym^{q-t} \cE^\vee / \im(\left(\substack{\cE^\vee \otimes\Sym^{q-t-2} \cE^\vee \\ \oplus \\ \Sym^{q-t-1}\cE^\vee}\right) \otimes \tau^*\cL \to \cE^\vee \otimes \Sym^{q-t}\cE^\vee) \nonumber 
		\end{align}
 	\end{lemma}
 	\begin{proof}
 		In the first case, we first observe by the Pieri rule that 
 		\[\Sym^t \cU_2^\vee \otimes \Sym^q \cU_2^\vee(-t) \simeq \bigoplus_{j = 0}^{t} \Sym^{q+t-2j} \cU_2^\vee(j-t).\]
 		By Lemma \ref{vanishinglemma2}, all of the summands have vanishing $R\tau_*$ except for $\Sym^{q-t} \cU_2^\vee$. Moreover, the proof of Lemma \ref{vanishinglemma2} shows that every term in the Koszul resolution for $\Sym^{q-t} \cU_2^\vee$ has vanishing higher direct images. Using the Borel-Weil-Bott theorem on the ambient Grassmannian bundle to calculate the pushforward of the bundle $\Sym^{q-t} \cU_2^\vee$ gives a canonical surjective map 
 		\[\Sym^{q-t} \cE^\vee \to \tau_*(\Sym^{q-t}\cU_2^\vee),\]
 		whose kernel is computed by the image of the pushforward of the map \[\Sym^{q-t} \cU_2^\vee \otimes \tau^*\cL\otimes \Sym^2 \cU_2\to \Sym^{q-t}\cU_2^\vee\] given by multiplication by the equation for the quadric. It is then straightforward to check that this map pushes forward to the map $\Sym^{q-t-2} \cE^\vee \otimes \tau^*\cL\to \Sym^{q-t} \cE^\vee$. \par 
 		The case of the pushforward $\tau_*(\Sym^{t+1} \cU_2^\vee \otimes \Sym^q \cU_2^\vee(-t))$ is similar, with two additional subtleties. First, there are two Schur functors which contribute to the pushforward, which are precisely the summands coming from the subbundle \[\cU_2^\vee \otimes \Sym^{q-t} \cU_2^\vee \subset \Sym^{t+1}\cU_2^\vee \otimes \Sym^q \cU_2^\vee(-t).\] Second, there are three Schur functors which contribute to the pushforward of the bundle $\cU_2^\vee \otimes \Sym^{q-t} \cU_2^\vee \otimes \tau^*\cL \otimes \Sym^2 \cU_2$ determining the relations: two of these combine to push forward to the map \[\cE^\vee \otimes \Sym^{q-t-2} \cE^\vee \otimes \tau^*\cL\to \cE^\vee \otimes \Sym^{q-t} \cE^\vee\] which is multiplication by the quadric relation in the right entry, and the last pushes forward to the map \[\Sym^{q-t-1} \cE^ \vee \otimes \tau^*\cL\to \cE^\vee \otimes \Sym^{q-t} \cE^\vee\] which splits the multiplication by the quadric relation between the left and right entry. Otherwise, the proof goes through essentially without change. 
 	\end{proof}
 	To conclude this section, we explain on a fiber over $s \in S$ how to interpret the isomorphism \eqref{U1calc} of Lemma \ref{vanishinglemma2}. If we define $Q \coloneq \cQ_s$ as the quadric on $V \coloneq \cE|_s$, by diagonalization of quadratic forms we may assume that we have a orthogonal basis $e_i$ of $V$ such that $Q(e_i,e_i) = \lambda_i$. As a consequence, we can write the defining equation of $Q$ as $\lambda_1x_1^2 + \cdots + \lambda_n x_n^2$. As a first step, it is most natural to view $\bbS^{(2,1\dots,1)} V^\vee$ as a quotient of $\bigwedge^k V^\vee \otimes V^\vee$: in fact, by the Pieri rule, we have an isomorphism 
 	\[\bbS^{(2,1,\dots,1)} V^\vee \simeq \left(\bigwedge^k V^\vee \otimes V^\vee \right)/ \bigwedge^{k+1} V^\vee,\]
 	where $\bigwedge^{k+1} V^\vee \to \bigwedge^k V^\vee \otimes V^\vee$ embeds via the comultiplication map. \par 
 	Explicitly, this can be viewed as being spanned by elements of the form $x_{i_1} \wedge x_{i_2} \wedge \cdots \wedge x_{i_k} \otimes x_{i_{k+1}}$ where we additionally impose the \textit{cyclic rotation relations} 
 	\begin{multline} \label{cyclicrotations}
 		x_{i_1} \wedge \cdots \wedge x_{i_k} \otimes x_{i_{k+1}} + (-1)^{k} x_{i_2} \wedge \cdots \wedge x_{i_{k+1}} \otimes x_{i_1} + \cdots + (-1)^{k^2}x_{i_{k+1}} \wedge \cdots \wedge x_{i_{k-1}}\otimes x_{i_k} \\ 
 		= \sum_{j=1}^{k+1} \sgn(\tau^j)x_{i_{\tau^j(1)}} \wedge \cdots \wedge x_{i_{\tau^j(k)}} \otimes x_{i_{\tau^j(k+1)}},
 	\end{multline}
 	where we let $\tau$ denote the permutation $(1 \;\; 2 \;\; \cdots \;\; k+1)$; these relations form a basis for the image of $\bigwedge^{k+1}V^\vee \to \bigwedge^k V^\vee \otimes V^\vee$. However, we also take a further quotient by the image of
 	\[\bigwedge^{k-1} V^\vee \to \bigwedge^{k-1} V^\vee \otimes \Sym^2 V^\vee \to \bbS^{(2,1,\dots,1)} V^\vee.\]
 	Having fixed our choice of basis diagonalizing the quadratic form, this can be given as the map
 	\[ \alpha \mapsto \sum_{i=1}^n \lambda_i \alpha \wedge x_i \otimes x_i,\]
 	giving the \textit{quadric relations}
 	\begin{equation} \label{quadricrelations}
 		\sum_{i=1}^n \lambda_i x_{i_1} \wedge x_{i_2} \wedge \cdots \wedge x_{i_{k-1}} \wedge x_i \otimes x_i
 	\end{equation}
 	for any $i_1 < i_2 < \cdots < i_{k-1}$. \par 
 	Finally, note that on smooth fibers, this isomorphism is precisely an identification of $H^0(\cU^\vee(1))$ with the \textit{orthogonal} Schur functor associated to weight $(2,1,\dots, 1)$. This can be checked by appealing to the Borel-Weil-Bott theorem on the orthogonal Grassmannian itself which is homogeneous for the action of $\Spin(V)$.
	\section{The FM kernel over smooth fibers, and a digression on spinors} \label{smoothsection}
	To make the construction discussed above more explicit, we discuss the more familiar situation where $S = \Spec \bbC$ is a single point and the quadric fibration is a single smooth quadric $Q \subset\bbP(V)$ where $\dim V = n$. Note that since the quadrics, orthogonal Grassmannians, and sheaves $\cF$ are all flat over the base $S$, restricting to a smooth quadric fiber reduces to this case. \par 
	When $k < n/2$, both the quadric $Q$ itself and the orthogonal Grassmannian $\OGr(k,Q)$ are homogeneous spaces under the action of $\SO(n)$, and therefore also under the action of its simply connected double cover $\Spin(n)$. When $n$ is even and $2k = n$, it is well-known that $\OGr(k,Q)$ has two connected components, both of which are homogeneous spaces under the action of $\SO(n)$ (and hence $\Spin(n)$). In either case, the group $\Spin(n)$ embeds as a subgroup of $\textcl_{\text{even}}^\times$. \par 
	Since the base is a single point, the FM kernel $\cF$ is defined by the right exact sequence of right $\textcl_{\text{even}}$-modules on $\OGr(k,Q)$ given by 
	\begin{equation}
		\cU_k \otimes \textcl_{\text{odd}} \to \textcl_{\text{even}} \to \cF \to 0. \label{rightexactF}
	\end{equation}
	Let us fix an maximal isotropic subspace $W \subset V$, in which case $\dim W = \lfloor n/2 \rfloor$. We briefly recall that one construction of the \textit{half-spinor (resp. spinor) representation} when $n$ is even (resp. odd) is given by taking the left ideal 
	\[I_W := \textcl \cdot \bigwedge^n W\]
	and then taking the even piece $I_W^\text{even} \subset \textcl_{\text{even}}$ or the odd pieces $I_W^\text{odd} \simeq \textcl_{\text{odd}} \otimes_{\textcl_{\text{even}}}I_W^\text{even}$. When $n$ is even, $I_W^\text{odd}$ is another half-spinor representation distinct from $I_W^\text{even}$, while when $n$ is odd, they are isomorphic $\Spin(n)$-representations. \par 
	One way to check that these are indeed the half-spinor (resp. spinor) representations is by observing that $I_W^\text{even}$ and $I_W^\text{odd}$ are clearly invariant under the action of $\textcl_{\text{even}}^\times$, and so define a $\textcl_{\text{even}}^\times$-representation. We can therefore restrict them to $\Spin(n)$-representations; then to check that they define the usual half-spinor (resp. spinor) representations, pick a convenient maximal torus in $\Spin(n)$ and check that they have the correct highest weight and dimension.  \par 
	When $n$ is odd, the action of $\textcl_{\text{even}}$ on the module $I_W^{\text{even}}$ realizes an isomorphism 
	\begin{equation} 
		\textcl_{\text{even}} \simeq \End_\bbC(I_W^\text{even}), \label{ClEnd1}
	\end{equation}
	so in particular $\textcl_{\text{even}}$ is a central simple algebra. Then we have the following well-known theorem of Morita equivalence: 
	\begin{proposition}\label{morita}
		Suppose that $M$ is a free $R$-module of finite rank. Then the functors between the categories of right modules
		\begin{align*}
			\End_R(M)\text{-}\mathrm{mod} &\to R\text{-}\mathrm{mod} \\
			N &\mapsto N \otimes_{\End_R(M)} M 
		\end{align*}
		and
		\begin{align*}
			R\text{-}\mathrm{mod} &\to \End_R(M)\text{-}\mathrm{mod} \\
			N &\mapsto N \otimes_R \Hom_R(M,R)
		\end{align*}
		are mutually inverse exact equivalences. 
	\end{proposition}
	In particular, if we apply this to the sequence \eqref{rightexactF} of right $\textcl_{\text{even}}$-modules on $\OGr(k,Q)$, we see that it is carried to the sequence of sheaves on $\OGr(k,Q)$ given by
	\[\cU_k \otimes I_W^\text{odd} \to I_W^\text{even} \to \cF \otimes_{\textcl_{\text{even}}} I_W^\text{even}\to 0\] 
	where one can check that this cokernel is a $\Spin(n)$-equivariant vector bundle with highest weight $(1/2,\cdots,1/2)$ and of the correct rank. In particular, we see that we have an isomorphism \begin{equation} \label{oddspinor}
		\cF \otimes_{\textcl_{\text{even}}} I_W^\text{even} \simeq \cS
	\end{equation} relating $\cF$ to the usual spinor bundle on $\OGr(k,V)$ when $\dim V$ is odd, which is exceptional by \cite[Proposition 6.8]{kuznetsovodd}. In particular, under the equivalence \[\Db(\OGr(k,Q),\textcl_{\text{even}}) \simeq \Db(\OGr(k,Q))\] coming from Proposition \ref{morita}, the exceptionality of $\cS$ implies the exceptionality of $\cF$ as a module over $\textcl_{\even}$. Moreover, the inverse equivalence realizes an isomorphism 
	\[
		\cF \simeq \cS \otimes (I_W^\text{even})^\vee, \]
	so if we forget the $\textcl_{\text{even}}$-module structure then purely as sheaves on the orthogonal Grassmannian we have the equivalence 
	\[R\Hom(\cF,\cF)\simeq R\Hom(\cS, \cS) \otimes \End(I_W^\text{even},I_W^\text{even}) \simeq \textcl_{\text{even}}.\]
	\par 
	When $n$ is even, the action of the full Clifford algebra $\textcl$ on $I_W$ realizes an isomorphism \cite[II.2.1]{chevalleyclifford}
	\[\textcl \simeq \End_\bbC(I_W),\]
	which induces by restriction to $\textcl_{\text{even}}$ an isomorphism 
	\begin{equation}
		\textcl_{\text{even}} \simeq \End_\bbC(I_W^\text{even}) \times \End_\bbC(I_W^\text{odd}). \label{ClEnd2}
	\end{equation}
	If we let $Z := Z(\textcl_{\text{even}})$ denote the center of the even Clifford algebra, then $Z$ has a natural action on $I_W$ which preserves the even and odd part as well. More precisely: if we pick a basis $f_1,\dots, f_n$ for $W$ and choose vectors $e_1, \dots, e_n$ such that $Q(e_i,f_j) = \delta_{ij}$, then the center $Z$ is two-dimensional over $\bbC$ and spanned by $1$ and the element 
	\[d := \frac{1}{2^n}(e_1+f_1)(e_1-f_1)\cdots(e_n+f_n)(e_n-f_n),\]
	and it is an easy explicit calculation in the Clifford algebra that $d$ acts by $1 \cdot$ on  $I_W^\text{even}$ and by $-1\cdot$ on $I_2^\text{odd}$. In particular, we can reinterpret the question of preserving even or odd pieces as being $Z$-linear, and hence reinterpret the isomorphism of \eqref{ClEnd2} as a $Z$-linear isomorphism 
	\begin{equation}
		\textcl_{\text{even}} \simeq \End_Z(I_W);
	\end{equation}
	As a consequence, we may view $\textcl_{\text{even}}$ as being a trivial Azumaya algebra over $Z$. On the other hand, $d^2 = 1$, so $\Spec Z = \Spec \bbC[d]/(d^2-1) \simeq \bbC \sqcup \bbC$. Taking the fibers of $I_W$ over the two points recovers $I_W^\text{even}$ and $I_W^\text{odd}$, and the fibers of $\textcl_\text{even} \simeq \End_Z(I_W)$ are $\End_\bbC(I_W^\text{even})$ and $\End_\bbC(I_W^\text{odd})$. By the same considerations as in the odd case, we see that Proposition \ref{morita} carries the sequence of right $\textcl_{\text{even}}$-modules on $\OGr(k,Q)$ to a sequence of right $Z$-modules on $\OGr(k,Q)$ given by 
	\[\cU_k \otimes I_W \to I_W \to \cF\otimes_{\textcl_{\text{even}}} I_W \to 0,\]
	which under the equivalence of $\Spec_{\OGr(k,Q)} Z \simeq \OGr(k,Q) \sqcup \OGr(k,Q)$ gives isomorphisms 
	\begin{equation}
		\label{evenspinor} \cF \otimes_{\textcl_{\text{even}}} I_W^{\text{even}} \simeq \cS_+ ,\qquad \qquad \cF \otimes_{\textcl_{\text{even}}} I_W^\text{odd} \simeq \cS_-
	\end{equation} with the spinor bundles, whose exceptionality is shown for example in \cite[Proposition 6.8]{kuznetsovodd}. When $k < n/2$, these are both vector bundles of the same rank on the connected variety $\OGr(k,Q)$, while when $k = n/2$ both $\cS_+$ and $\cS_-$ are line bundles supported on different connected components of $\OGr(k,Q)$. In particular, in the category $\Db(\OGr(k,Q), \textcl_{\text{even}})$, we can deduce
	\[R\End_{\textcl_{\text{even}}}(\cF,\cF) = \bbC \times \bbC\]
	and that the even and odd pieces of $\cF$ are exceptional. \par 
	On the other hand, applying the inverse equivalence gives an isomorphism
	\[\cF \simeq (\cS_+ \oplus \cS_-) \otimes_\cZ I_W^\vee \simeq \cS_+ \otimes (I_W^\text{even})^\vee \oplus \cS_- \otimes (I_W^\text{odd})^\vee,\]
	and forgetting the $\textcl_{\text{even}}$-module structure we see that as sheaves of $\cO$-modules on the orthogonal Grassmannian we have 
	\begin{equation}
		\begin{aligned}
			R\Hom(\cF,\cF) &\simeq R\Hom(\cS_+,\cS_+) \otimes \Hom(I_W^\text{even}, I_W^\text{even}) \oplus R\Hom(\cS_-,\cS_-) \otimes \Hom(I_W^\text{odd},I_W^\text{odd}) \\
			&\simeq \End(I_W^\text{even}) \oplus \End(I_W^\text{odd}) \\
			&\simeq \End_Z(I_W) \\
			&\simeq \textcl_{\text{even}}.
		\end{aligned}
	\end{equation}
	In particular, we may observe that we have proved the following proposition. 
	\begin{proposition}
		Suppose that $Q \subset \bbP(V)$ is a smooth quadric hypersurface, and take the sheaf of right $\textcl_{\text{even}}$-modules $\cF$ as defined above on $\OGr(k,Q)$, i.e. 
		\[\cF := \coker(\cU \otimes \textcl_{\text{odd}} \to \textcl_{\text{even}}).\]
		Then as sheaves of $\cO$-modules on $\OGr(2,Q)$, we have 
		\[R\Hom(\cF,\cF) \simeq \textcl_{\text{even}}. \]
	\end{proposition}
	But the flatness of $\cF$ allows us to propagate this property to any fiber in our fibration regardless of singularities, at least on the level of Euler characteristics. 
	\begin{corollary}
		\label{constancy}
		Suppose that $p : \cQ \to S$ is a flat quadric fibration such that every fiber is of corank $\leq n - 2k + 1$, and consider the fibration $\tau: \OGr(k,\cQ) \to S$. Then for all $s \in S$, 
		\begin{equation}
			\chi(\cF|_s,\cF|_s) := \chi\left(R\Hom(\cF|_s,\cF|_s)\right) = \on{rk}\cl_0,
		\end{equation}
		and $\dim \Hom(\cF|_s,\cF|_s) \geq \on{rk}\cl_0$. 
	\end{corollary}
	\begin{proof}
		If we knew that the quadric fibration contained a point $s \in S$ where the fiber was smooth, we would be done by the semicontinuity theorem and the constancy of Euler characteristic for the flat sheaf $\cF^\vee \otimes \cF$ over the flat projective family $p : \OGr(k,\cQ) \to S$. The only nontrivial content to the corollary is the fact that the quadric fibration need not have a point $s \in S$ such that $\cQ_s$ is smooth! \par 
		However, shrinking $S$ if necessary, there is a natural map $S \to \abs{\cO_{\bbP(V)}(2)}$ to the linear system of all quadrics; moreover, $S$ lands in the open subset of the linear system where each fiber is of corank $\leq n - 2k$, which in particular contains all smooth quadrics. Working instead on this larger family, we deduce the claim. 
	\end{proof}
	We will also find it useful to observe that similar results hold for Euler pairings between twists of $\cF$ and certain tautological vector bundles on relative orthogonal Grassmannians, at least when $k = 2$. 
	\begin{proposition}\label{constancy2}
		Suppose that $p : \cQ\to S$ is a flat quadric fibration such that every fiber is of corank $\leq n - 3$, and consider the fibration $\OGr(2,\cQ) \to S$. For any $s \in S$, whenever $0 \leq t \leq n - 4$ and $0 \leq q \leq n/2 - 2$, 
		we have equalities 
		\begin{equation}
				\chi(\cF|_s\otimes \cO(-t),\cF|_s) = R\Hom(\cF|_s \otimes \cO(-t),\cF|_s) = 0
		\end{equation}
		and 
		\begin{equation}
			\chi(\cF|_s \otimes \cO(-t),\Sym^2 \cU_2^\vee) = R\Hom(\cF|_s \otimes \cO(-t),\Sym^2\cU_2^\vee) = 0.
		\end{equation}
	\end{proposition}
	\begin{proof}
		Over smooth fibers, $\cF$ is a direct sum of copies of spinor bundles, and in this case the vanishing follows from the semiorthogonality results in \cite{kuznetsovodd} and \cite{kuznetsoveven}. By a similar degeneration argument to Corollary \ref{constancy}, the same holds for any possibly singular fiber. 
	\end{proof}
	\begin{remark} \label{spinorsheafremark}
		If we allow the quadric to be singular, then one can take $\cF \otimes_{\textcl_{\even}} I_W^\even$ and $\cF \otimes_{\textcl_\even}{I_W^{\mathrm{odd}}}$ to be a \textit{definition} for spinor sheaves, analogous to the one presented in Addington's thesis  \cite{singularquadrics}. \par 
		Moreover, when we choose our family of quadrics to be a pencil of quadrics with smooth base locus in an odd dimensional projective space, the central reduction of \cite{quadfibs} takes the Fourier-Mukai functor \[\Phi_{\cF}: \Db(S,\cl_0) \to \Db(\OGr(k,\cQ))\] to the functor $\Db(C) \to \Db(\OGr(k,\cQ))$ with Fourier-Mukai kernel given by varying the spinor sheaf with a maximal isotropic subspace $W$, as is done in \cite{singularquadrics} for the case of quadrics.
	\end{remark}
	\section{Proof of Theorem \ref{embeddingtheorem}} \label{proofsection}
	In order to actually show that the functor $\Phi_\cF$ defines a fully faithful functor, we first observe that the right adjoint $\Psi_\cF$ of $\Phi_\cF$ is given by 
	\begin{equation}
		\begin{aligned}
			\Psi_\cF : \Db(\OGr(k,\cQ)) &\to \Db(S, \cl_0) \\ 
			\cG &\mapsto R\tau_*(\cF^\vee \otimes \cG). 
		\end{aligned}
	\end{equation}
	where $\cF^\vee$ denotes the $\cO$-linear dual; this follows from an application of derived tensor-hom adjunction, using that $\cF$ is locally free as a $\cO$-module, followed by the derived pullback-pushforward adjunction. \par 
	Using the lemmas proved in the sections above, we can deduce the main ingredient in the theorem. 
	\begin{lemma}\label{keylemma}
		We have an isomorphism of $\cl_0$-bimodules
		\begin{equation}
			R\tau_*(\cF^\vee \otimes \cF) \simeq \cl_0.
		\end{equation}
	\end{lemma}
	\begin{proof}
		Observe that we have a left and a right resolution for $\cF^\vee \otimes \cF$ coming from the left and right resolution \eqref{leftres} and \eqref{rightres} of $\cF$: 
		\begin{gather}
			\cdots \to \begin{matrix}
				\tau^*\cl_{k+1}^\vee \otimes \cU_k(-1) \otimes \tau^*\cl_0 \\
				\oplus \\
				\tau^*\cl_k^\vee \otimes \cU_k(-1) \otimes \tau^*\cl_{-1}
			\end{matrix}\to \tau^*\cl_k^\vee \otimes \cO_{\OGr(k,\cQ)}(-1) \otimes \tau^* \cl_0 \to \cF^\vee \otimes \cF \to 0  \label{lefttensorres}\\
			0 \to \cF^\vee \otimes \cF \to \tau^*\cl_0^\vee \otimes \cO_{\OGr(k,\cQ)}(1) \otimes \tau^*\cl_k \to \begin{matrix}
				\tau^*\cl_{-1}^\vee \otimes \cU_k^\vee(1) \otimes \tau^*\cl_k \\
				\oplus \\
				\tau^*\cl_0^\vee \otimes \cU_k^\vee(1) \otimes \tau^*\cl_{k+1}
			\end{matrix} \to \cdots \label{righttensorres}
		\end{gather}
		We first use the resolution \eqref{lefttensorres}, which we denote as $A^\bullet_{\text{left}}$. Observe that the $\ell$th term $A^{-\ell}_{\text{left}}$ is always of the form \[\bigoplus_{\substack{\text{finitely}\\ \text{many }i,j}}\tau^*\cl_i^\vee \otimes \Sym^p \cU_k \otimes \Sym^q \cU_k(-1) \otimes \tau^*\cl_j,\]
		where $p + q = \ell$, but by Lemma \ref{vanishinglemma1}, we know that $R^i\tau_*(\Sym^p \cU_k \otimes \Sym^q \cU_k(-1)) = 0$ whenever $i > \ell + 1$. If we allow $\sigma^{\geq -n} A^\bullet_{\text{left}}$ to denote the stupid truncations of $A^\bullet_{\text{left}}$, then it follows from the hypercohomology spectral sequence that 
		\[R^i\tau_*(\sigma^{\geq -n} A^\bullet_{\text{left}}) = 0\]
		for $i \geq 2$. But if we then observe that $A^\bullet_{\text{left}} = \colim_{n\to\infty} \sigma^{\geq -n}A^\bullet_{\text{left}}$ is a directed colimit of complexes, using that $R\tau_*$ is a left adjoint to the Grothendieck duality functor we see that \[R^i\tau_*(A^\bullet_{\text{left}}) = R^i\tau_*(\cF^\vee \otimes \cF) = 0 \]
		whenever $i \geq 2$. But since $\cF^\vee \otimes \cF$ is just a sheaf, certainly $R\tau_*(\cF^\vee \otimes \cF)$ has cohomology only in nonnegative degrees, so the only nonzero terms are 
		\[\tau_*(\cF^\vee \otimes \cF)\quad \text{and}\quad R^1\tau_*(\cF^\vee \otimes \cF).\]
		\par 
		In order to conclude our proof, we will make use of the Lemma \ref{constancy} on the constancy of the Euler characteristic and the theorems of cohomology and base change in order to reduce this question to yet another fiberwise computation. To be more precise: we will construct a map 
		\begin{equation} 
			\label{cl0isombasic}
			\cl_0 \to \tau_*(\cF^\vee \otimes \cF) 
		\end{equation}
		which, on the fibers of the map gives an isomorphism 
		\[\textcl_{even} \to \tau_*(\cF^\vee \otimes \cF) \otimes \kappa(s) \to  H^0(\cF|_s^\vee \otimes \cF|_s) = \End(\cF|_s),\]
		where $\cF|_s \coloneq \cF \otimes \kappa(s)$ is the fiber of $\cF$ at $s \in S$. But by the constancy of Euler characteristic, this implies that $H^1(\cF|_s^\vee \otimes \cF|_s) = 0$ for all $s \in S$, and the usual theorems on cohomology and base change allow us to conclude that we have an isomorphism
		\[\cl_0 \iso \tau_*(\cF^\vee \otimes \cF) \simeq R\tau_*(\cF^\vee \otimes \cF).\]
		To begin, observe that by the exact sequence \eqref{righttensorres} and the projection formula, we have an identification 
		\begin{equation}
			\tau_*(\cF^\vee \otimes \cF) \simeq \ker\left(\cl_0^\vee \otimes \tau_*(\cO_{\OGr(k,\cQ)}(1)) \otimes \cl_k \to \begin{matrix}
				\cl_{-1}^\vee \otimes \tau_*(\cU_k^\vee(1)) \otimes \cl_k \\
				\oplus \\
				\cl_0^\vee \otimes \tau_*(\cU_k^\vee(1)) \otimes \cl_{k+1}
			\end{matrix}\right).
		\end{equation}
		Let us define by analogy to the smooth fibers (see the remark following Lemma \ref{vanishinglemma2}) the notation
		\[\bbS^{(2,1,\dots,1)}_{\SO} \cE^\vee := \tau_*(\cU^\vee(1))\simeq \bbS^{(2,1,\dots,1)} \cE^\vee / \im\left(\bigwedge^{k-1}\cE^\vee \otimes \tau^*\cL \to \bigwedge^{k-1}\cE^\vee \otimes \Sym^2 \cE^\vee \to \bbS^{(2,1,\dots,1)} \cE^\vee\right)\]
		which is the isomorphism \eqref{U1calc} of Lemma \ref{vanishinglemma2}. Then applying this along with the other isomorphism \eqref{O1calc} of Lemma \ref{vanishinglemma2} gives an isomorphism 
		\begin{equation}
			\label{mapkernel}
			\tau_*(\cF^\vee \otimes \cF) \simeq \ker\left(\cl_0^\vee \otimes \bigwedge^k\cE^\vee \otimes \cl_k \to \begin{matrix}
				\cl_{-1}^\vee \otimes \bbS^{(2,1,\dots,1)}_{\SO}\cE^\vee \otimes \cl_k \\
				\oplus \\
				\cl_0^\vee \otimes \bbS^{(2,1,\dots,1)}_{\SO}\cE^\vee \otimes \cl_{k+1}
			\end{matrix}\right),
		\end{equation}
		where the first map is given by 
		\[
		\cl_0^\vee \otimes \bigwedge^k\cE^\vee \otimes \cl_k \to \cl_0^\vee \otimes \cE \otimes \cE^\vee \otimes \bigwedge^k \cE^\vee \otimes \cl_k\to\cl_{-1}^\vee \otimes \bbS^{(2,1,\dots,1)} \cE^\vee \otimes \cl_k \to \cl_{-1}^\vee \otimes \bbS^{(2,1,\dots,1)}_{\SO}\cE^\vee \otimes \cl_k
		\]
		consisting of the monoidal coevaluation $\cO \to \cE \otimes \cE^\vee$ followed by the Clifford multiplication $ \cl_0^\vee \otimes \cE \to \cl_{-1}^\vee $
		and the canonical surjection $\cE^\vee \otimes \bigwedge^k \cE^\vee \to \bbS^{(2,1,\dots,1)}_{\SO} \cE^\vee \to \bbS^{(2,1,\dots,1)}_{\SO} \cE^\vee,$ and the second map is given by 
		\[
		\cl_0^\vee \otimes \bigwedge^k\cE^\vee \otimes \cl_k \to \cl_0^\vee \otimes \bigwedge^k \cE^\vee \otimes \cE^\vee \otimes \cE \otimes \cl_k \to\cl_0^\vee \otimes \bbS^{(2,1,\dots,1)} \cE^\vee \otimes \cl_{k+1} \to \cl_0^\vee \otimes \bbS^{(2,1,\dots,1)}_{\SO}\cE^\vee \otimes \cl_{k+1}
		\]
		consisting again of the monoidal coevaluation $\cO\to \cE^\vee \otimes \cE$ followed by the canonical surjection \[\bigwedge^k \cE^\vee \otimes \cE^\vee \to \bbS^{(2,1,\dots,1)} \cE^\vee \to \bbS^{(2,1,\dots,1)}_{\SO} \cE^\vee\] along with the Clifford multiplication $\cE \otimes \cl_k \to \cl_{k+1}$. \par
		On the other hand, there is a natural map of sheaves 
		\begin{equation}
			\label{kerguess}
			\cl_0 \to \cl_0^\vee \otimes \cl_0 \to \cl_0^\vee \otimes \bigwedge^k \cE^\vee \otimes \bigwedge^k \cE \otimes \cl_0 \to \cl_0^\vee \otimes \bigwedge^k \cE^\vee \otimes \cl_k
		\end{equation}
		where the first map is the natural map $\cl_0 \to \Endsheaf_\cO(\cl_0)$ given by right multiplication, the second is again monoidal coevaluation $\cO \to \bigwedge^k \cE^\vee \otimes \bigwedge^k \cE$, and the last map is the composition $\bigwedge^k \cE \otimes \cl_0 \to \cE \otimes \cE \otimes \cl_0 \to \cl_k$. Recall that by Proposition \ref{endomorphismprop} we have an identification of the image of the map $\cl_0 \to \Endsheaf_\cO(\cl_0)$ with the left $\cl_0$-linear endomorphisms of $\cl_0$. 
		\par
		Now, we check on fibers that the composition is an isomorphism onto the kernel of the homomorphism \eqref{mapkernel}. As before, we fix a point $s \in S$ and consider the quadric $Q = \cQ_s$ on $V := \cE|_s$; the fibers of $\cl_i$ restrict to $\textcl_{\text{even}}$ when $i$ is even and $\textcl_{\text{odd}}$ when $i$ is odd. We write $\parity(k)$ to denote whether $k$ is even or odd. Let us fix an orthogonal basis $e_1, \dots, e_n$ for $V$ such that $Q(e_i,e_i) = \lambda_i$; this in particular implies that $e_ie_j = -e_je_i$ whenever $i \neq j$. By the assumption that the corank is $\leq n - 2k + 1$ and therefore the rank is at least $2k  - 1 \geq  k$, we may assume that $\lambda_1 = \lambda_2 = \cdots = \lambda_{k} = 1$. Examining \eqref{mapkernel}, we have an identification of $\End(\cF(s))$ with the kernel of the homomorphism  
		\begin{equation}
			\label{mapkernelfiber}
			\bigwedge^k V^\vee  \otimes \Hom(\textcl_{\text{even}},\textcl_{\parity(k)}) \to \begin{matrix}
				\bbS^{(2,1,\dots,1)}_{\SO}V^\vee \otimes \Hom(\textcl_{\text{odd}}, \textcl_{\parity(k)}) \\
				\oplus \\
				\bbS^{(2,1,\dots,1)}_{\SO}V^\vee  \otimes \Hom(\textcl_{\text{even}}, \textcl_{\parity(k+1)})
			\end{matrix} 
		\end{equation}
		which takes 
		\[\omega \otimes \phi \mapsto \left(\sum_{j=1}^n \omega \otimes x_j \otimes \phi\left(e_j \emptyinput\right), \sum_{j=1}^n \omega \otimes x_j \otimes e_j\phi\right).\]
		On the other hand, the restriction of homomorphism \eqref{kerguess} to the fiber can be viewed as a map 
		\begin{equation}\label{kerguessfiber}
			\textcl_{\text{even}} \to \bigwedge^k V^\vee \otimes \Hom(\textcl_{\text{even}},\textcl_{\parity(k)}), 
		\end{equation}
		and should be interpreted as the composition of the map
		\[\xi \mapsto (\alpha \mapsto \alpha\xi)\]
		followed by the map
		\[\phi \mapsto \sum_{1 \leq i_1 < \cdots < i_k \leq n } x_{i_1} \wedge \cdots \wedge x_{i_k}\otimes \left(\sum_{\sigma \in S_k} \sgn(\sigma)e_{i_{\sigma(1)}}\cdots e_{i_\sigma(k)}\right)\phi.\]
		However, since we are assuming that the basis is orthogonal so that $e_ie_j = -e_je_i$ for $i \neq j$, we may rewrite $\sum_{\sigma \in S_k} \sgn(\sigma)e_{i_{\sigma(1)}}\cdots e_{i_{\sigma_k}} = k! \cdot e_{i_1} \cdots e_{i_k}$, and therefore the composition is nothing but the map 
		\[\xi \mapsto \sum_{1 \leq i_1 < \cdots < i_k \leq n} x_{i_1} \wedge \cdots \wedge x_{i_k} \otimes (\alpha \mapsto k! \cdot e_{i_1}\cdots e_{i_k}\alpha\xi).\]
		To characterize the kernel of \eqref{mapkernelfiber}, we therefore take an arbitrary element 
		\[\sum_{1 \leq i_1 < \cdots < i_k \leq n} x_{i_1} \wedge \cdots \wedge x_{i_k} \otimes \phi_{i_1\cdots i_k} \in \bigwedge^k V^\vee \otimes \Hom(\textcl_{\text{even}}, \textcl_{\parity(k)})\]
		and check what conditions lying in the kernel impose on this element. For convenience's sake, we define as notation $\phi_{i_{\sigma(1)}i_{\sigma(2)}\cdots i_{\sigma(k)}} = \sgn(\sigma)\phi_{i_1i_2\cdots i_k}$ for any permutation $\sigma \in S_k$. Applying the map $\bigwedge^k V^\vee \otimes \Hom(\textcl_{\even},\textcl_{\parity(k)}) \to \bigwedge^kV^\vee \otimes V^\vee \otimes \Hom(\textcl_{\text{odd}},\textcl_{\parity(k)})$ gives
		\begin{equation}
			\sum_{j=1}^n \sum_{1 \leq i_1 < \cdots < i_k \leq n} x_{i_1} \wedge \cdots \wedge x_{i_k} \otimes x_j \otimes \phi_{i_1\cdots i_k}(e_j\emptyinput).
		\end{equation}
		By inspection, one can break this sum up and reindex to produce the summation
		\begin{multline}
			\sum_{1 \leq i_1 < \cdots < i_{k-1} \leq n}^n \sum_{j \neq i_\ell} x_{i_1} \wedge \cdots \wedge x_{i_{k-1}} \wedge x_j \otimes x_j \otimes \phi_{i_1\cdots i_{k-1}i_j}(e_j\emptyinput) \\
			+ \sum_{1 \leq i_1 < \cdots < i_k < i_{k+1} \leq n} \sum_{j=1}^{k+1}\left(x_{i_{\tau^j(1)}} \wedge \cdots \wedge x_{i_{\tau^j(k)}} \otimes x_{i_{\tau(k+1)}} \otimes \phi_{i_{\tau^j(1)}\cdots i_{\tau^j(k)}}(e_{i_{\tau^j(k+1)}}\emptyinput)\right),
		\end{multline}
		where we again let $\tau$ denote the cyclic permutation $(1 \;\; 2 \;\; \cdots \;\; k + 1)$. Composing with the surjection 
		\[\bigwedge^k V^\vee \otimes V^\vee \otimes \Hom(\textcl_{\even},\textcl_{\parity(k)}) \to \bbS^{(2,1,\dots,1)}_{\SO}V^\vee \otimes \Hom(\textcl_{\text{odd},\textcl_{\parity(k)}}\]
		gives the image of our element under the top map of \eqref{mapkernelfiber}. But by assumption, if our element lies in the kernel of \eqref{mapkernelfiber}, then this must vanish when viewed as an element of $\bbS^{(2,1,\dots,1)}_{\SO} V^\vee \otimes \Hom(\textcl_{\text{odd}},\textcl_{\parity(k)})$, so as an element of $\bigwedge^k V^\vee \otimes V^\vee \otimes \Hom(\textcl_{\text{odd}},\textcl_{\parity(k)})$, it lies in the span of the cyclic rotation relations \eqref{cyclicrotations} and quadric relations \eqref{quadricrelations} tensored with $\Hom(\textcl_{\text{odd}},\textcl_{\parity(k)})$. In particular, this implies that there exists for each tuple $1 \leq i_1 < \cdots < i_{k-1} \leq n$ some $\varphi_{i_1i_2\cdots i_{k-1}} \in \Hom(\textcl_{\text{odd}},\textcl_{\parity(k)})$ such that 
		\[ \sum_{\substack{j=1 \\ j\neq  i_\ell}}^n x_{i_1} \wedge \cdots \wedge x_{i_{k-1}} \otimes x_j \otimes \phi_{i_1\cdots i_{k-1}j}(e_j\emptyinput) = \sum_{\substack{j = 1 \\ j \neq  i_\ell}}^n \lambda_j x_{i_1} \wedge \cdots \wedge x_{i_{k-1}} \otimes x_j \otimes \varphi_{i_1\cdots i_{k-1}}, \]
		which forces us to conclude that 
		\begin{equation} \label{innerquadric}
			\phi_{i_1\cdots i_{k-1}j}(e_j\emptyinput) = \lambda_j \varphi_{i_1\cdots i_{k-1}}
		\end{equation}
		for all $1 \leq i_1 < \cdots < i_{k-1} \leq n$ and $j \neq i_\ell$. If we likewise define, $\varphi_{i_{\sigma(1)}} \cdots i_{\sigma(k-1)} = \sgn(\sigma)\varphi_{i_1\cdots i_{k-1}}$, then it is equivalent for this identity to hold for any permutation of the $i_1, \dots, i_{k-1}$. Similarly, we can conclude that there exist elements $\psi_{i_1i_2\cdots i_{k+1}} \in \Hom(\textcl_{\text{odd}},\textcl_{\parity(k)})$ where the $i_\ell$ are all distinct such that 
		\begin{equation}\label{innerrotation}
			\psi_{i_1i_2\cdots i_{k+1}} = \phi_{i_1i_2\cdots i_k}(e_{i_{k+1}}\emptyinput) = \sgn(\tau^j)\phi_{i_{\tau^j(1)}i_{\tau^j(2)}\cdots i_{\tau^j(k)}}(e_{\tau^j(k+1)}\emptyinput)
		\end{equation} 
		for $\tau = (1 \;\; 2 \;\; \cdots \;\; k+1)$. This can be more explicitly written as equalities
		\[\phi_{i_1i_2\cdots i_k}(e_{k+1}\emptyinput) = (-1)^{k}\phi_{i_2i_3\cdots i_{k+1}}(e_{i_1}\emptyinput) = (-1)^{2k}\phi_{i_3i_4\cdots i_{1}}(e_{i_2}) =\cdots = (-1)^{k^2} \phi_{i_{k+1}i_1\cdots i_{k-1}}(e_{i_k}\emptyinput).\]
		\par 
		On the other hand, applying the second map of \eqref{mapkernelfiber} and observing the relations lying in the kernel impose, we find elements $\varphi_{i_1\cdots i_{k-1}}',\psi_{i_1\cdots i_{k+1}}'\in \Hom(\textcl_{\text{even}}, \textcl_{\parity(k+1)})$ such that 
		\begin{equation} \label{outerquadric}
			e_j\phi_{ji_1\cdots i_{k-1}} = \lambda_j \varphi_{i_1\cdots i_{k-1}}'
		\end{equation}
		for $j, i_1, \cdots, i_{k-1}$ all distinct, and 
		\begin{equation}\label{outerrotation}
			\psi_{i_1\cdots i_{k+1}}' = e_{i_1}\phi_{i_2i_3\cdots i_{k+1}} = \sgn(\tau^j)e_{i_{\tau^j(1)}}\phi_{i_{\tau^j(2)}i_{\tau^j(3)}\cdots i_{\tau^j(k+1)}}
		\end{equation}
		whenever $i_1, \dots, i_{k+1}$ are all distinct. \par 
		If we take an element in the image of the map \eqref{kerguessfiber}, this amounts to setting
		\[\phi_{i_1i_2\cdots i_k}(\alpha) = k! \cdot e_{i_1}e_{i_2}\cdots e_{i_k}\alpha\xi,\]
		for some element $\xi \in \textcl_{\text{even}}$, and it is immediate to see that we can take 		
		\[\varphi_{i_1\cdots i_{k-1}}(\alpha) = \varphi_{i_1\cdots i_{k-1}}'(\alpha) = k!e_{i_1}\cdots e_{i_{k-1}}\alpha\xi \qquad \qquad \psi_{i_1\cdots i_{k+1}}(\alpha) = \psi_{i_1\cdots i_{k+1}}'(\alpha) = k! e_{i_1}\cdots e_{i_{k+1}}\alpha\xi,
		\]
		and therefore conclude that any element in the image of \eqref{kerguessfiber} must lie in the kernel of \eqref{mapkernelfiber}. \par 
		Conversely, suppose that we have an arbitrary element $\sum_{1 \leq i_1 < \cdots < i_k \leq n} x_{i_1}\wedge \cdots \wedge x_{i_k} \otimes \phi_{i_1\cdots i_k}$ in the kernel, and recall that the $\phi_{i_1\cdots i_k}$ satisfy the identities involving the $\varphi_{i_1\cdots i_{k-1}}, \psi_{i_1\cdots i_{k+1}}, \varphi_{i_1\cdots i_{k-1}}', \psi_{i_1\cdots i_{k+1}}'$ as above. It then suffices by Proposition \ref{endomorphismprop} for us to find a left $\textcl_{\text{even}}$-linear endomorphism $\phi$ of $\textcl_{\text{even}}$ such that $\phi_{i_1\cdots i_k} = k!\cdot e_{i_1}\cdots e_{i_k}\phi$. Keeping in mind our assumption that $\lambda_1 = \lambda_2 = \cdots = \lambda_{2k} = 1$, we define 
		\[\phi := \frac{e_ke_{k-1}\cdots e_1}{k!} \cdot \phi_{123\cdots k}.\]
		We then check that $k!\cdot e_{i_1}e_{i_2}\cdots e_{i_k}\phi = \phi_{i_1\cdots i_k}$ for $i_1, \cdots, i_k$ all distinct. In particular, this amounts to checking that 
		\[\phi_{i_1i_2\cdots i_k} = e_{i_1}e_{i_2}\cdots e_{i_k}e_{k}e_{k-1}\cdots e_1 \phi_{123\cdots k}.\]
		We will show this inductively. It suffices to show that given any length $\ell$ subsequence of distinct integers $j_1,j_2,\dots, j_\ell$ of $1,2, \dots, k$ which is pairwise distinct from $i_{\ell+1}, \cdots, i_{k}$, we can find a subsequence $m_1, m_2, \dots, m_{\ell - 1} \leq n$ of $j_1, \dots, j_\ell$ pairwise distinct from $i_{\ell}, i_{\ell+1}, \cdots, i_k$ such that 
		\[e_{i_1}e_{i_2}\cdots e_{i_\ell}e_{j_\ell}\cdots e_{j_2}e_{j_1}\phi_{j_1\cdots j_\ell i_{\ell+1}\cdots i_k} = e_{i_1}e_{i_2}\cdots e_{i_{\ell-1}}e_{m_{\ell-1}}\cdots e_{m_2}e_{m_1}\phi_{k_1\cdots k_{\ell-1}i_{\ell+1}\cdots i_k}.\]
		If we know that $i_\ell \neq j_i$ for any $1 \leq i \leq \ell$, then this is straightforward: using the identity \eqref{outerrotation} once and the fact that $e_{j_1}e_{j_1} = \lambda_{j_1} = 1$ since $j_1 \in \{1,2,\dots, 2k\}$, we can check 
		\begin{equation}
			\begin{aligned}
				e_{i_1}e_{i_2}\cdots e_{i_\ell}e_{j_\ell}\cdots e_{j_2}e_{j_1}\phi_{j_1\cdots j_\ell i_{\ell+1}\cdots i_k} &= (-1)^\ell e_{i_1} \cdots e_{i_{\ell-1}} e_{j_\ell} \cdots e_{j_2}e_{j_1}e_{i_\ell}\phi_{j_1\cdots j_\ell i_{\ell+1}\cdots i_k} \\
				&= (-1)^{\ell + k}e_{i_1}\cdots e_{i_{\ell-1}}e_{j_\ell} \cdots e_{j_{2}}e_{j_1}e_{j_1}\phi_{j_2 \cdots j_{\ell}i_{\ell+1}\cdots i_k i_\ell} \\
				&= e_{i_1}\cdots e_{i_{\ell-1}}e_{j_\ell}\cdots e_{j_2}\phi_{j_2\cdots j_{\ell}i_\ell i_{\ell+1}\cdots i_k}
			\end{aligned}
		\end{equation}
		and we can take $m_1 = j_2, m_2 = j_3, \cdots, m_{\ell - 1} = j_\ell$. Otherwise, there exists some $j_i$ such that $i_\ell = j_i$, and there is an equality 
		\begin{equation}
			\begin{aligned}
				e_{i_1}e_{i_2}\cdots e_{i_\ell}e_{j_\ell}\cdots e_{j_i}\cdots e_{j_2}e_{j_1}\phi_{j_1\cdots j_i \cdots j_\ell i_{\ell+1}\cdots i_k} &= e_{i_1} \cdots e_{i_{\ell-1}} e_{i_\ell} e_{j_i}e_{j_\ell} \cdots e_{j_{i+1}}e_{j_{i-1}}\cdots e_{j_1}\phi_{j_1\cdots j_{i-1}j_{i+1}\cdots j_\ell j_ii_{\ell+1}\cdots i_k} \\
				&= e_{i_1}\cdots e_{i_{\ell-1}}e_{j_\ell}\cdots e_{j_{i+1}}e_{j_{i-1}}\cdots e_{j_1} \phi_{j_1\cdots j_{i-1}j_{i+1}\cdots j_\ell i_\ell i_{\ell+1}\cdots i_k},
			\end{aligned}
		\end{equation}
		where $e_{i_\ell}e_{j_i} = e_{j_i}e_{j_i} = \lambda_{j_i} = 1$ since $j_i \in \{1,\dots, 2k\}$. Hence in this case we can take $m_1 = j_1, m_2 = j_2, \dots, m_{i-1} = j_{i-1}, m_i = j_{i+1}, \dots, m_{\ell-1} = j_{\ell}$. This proves that $\phi \in \Hom(\textcl_{\even}, \textcl_{\parity(k)})$ as defined above satisfies the equalities $k!\cdot e_{i_1}\cdots e_{i_k}\phi = \phi_{i_1\dots i_k}$, and so under the natural composition \[\End(\textcl_{\even}) \to \bigwedge^k V^\vee \otimes \Hom(\textcl_{\even},\textcl_{\parity(k)}),\] the element $\phi$ maps to $\sum_{1 \leq i_1 < \cdots < i_k \leq n}\phi_{i_1\cdots i_k}$. \par 
		To show that our element $\sum_{1 \leq i_1 < \cdots < i_k \leq n}\phi_{i_1\cdots i_k}$ actually is in the image of \eqref{kerguessfiber}, Proposition \ref{endomorphismprop} shows that it suffices to show that $\phi$ is a left $\textcl_{\even}$-linear endomorphism. To check this left $\textcl_{\even}$-linearity, it suffices to check that 
		\[e_ie_j\phi = \phi(e_ie_j\emptyinput)\]
		for any $i \neq j$, as these terms generate the algebra $\textcl_{\text{even}}$. If we fix some $i \neq j$, then since the rank of the quadratic form is at least $2k$, it follows that there are at least $k - 1$ elements of the basis distinct from $e_i$ and $e_j$ which square to 1; for simplicity, we may assume that $1, 2, \dots , k - 1, i, j$ are all pairwise distinct. Then by the identity \eqref{innerrotation}, we have the calculation
		\begin{align*}
			e_1e_2 \cdots e_{k-1}e_i\phi(e_j\emptyinput) &= \frac{1}{k!}\cdot\phi_{12\cdots(k - 1)i}(e_j\emptyinput) \\
			&= \frac{(-1)^k}{k!}\cdot \phi_{j12 \cdots (k-1)}(e_i\emptyinput) \\
			&= (-1)^ke_je_1e_2\cdots e_{k-1}\phi(e_i\emptyinput) \\
			&= -e_1e_2\cdots e_{k-1}e_j\phi(e_i\emptyinput),\\
		\end{align*}
		so left-multiplying by $e_{k-1}\cdots e_2e_1$ we conclude that 
		\[e_i\phi(e_j\emptyinput) = -e_j\phi(e_i\emptyinput).\]
		Then using this identity, we can compute: 
		\begin{align*}
			e_ie_j\phi &= e_ie_j\phi(e_1e_1\emptyinput) \\
			&= -e_ie_1\phi(e_je_1\emptyinput) \\
			&= -e_1e_i\phi(e_1e_j\emptyinput) \\
			&= e_1e_1\phi(e_ie_j\emptyinput) \\
			&= \phi(e_ie_j\emptyinput).
		\end{align*}
		It follows that this result holds for any $i \neq j$. \par 
		As a consequence, we see that the map \eqref{kerguessfiber} is a surjection onto the kernel of homomorphism \eqref{mapkernelfiber}, or equivalently a surjection 
		\[\textcl_{\text{even}} \simeq \cl_0 \otimes \kappa(s) \to \End(\cF|_s).\]
		By Corollary \ref{constancy}, $\dim \End(\cF(s)) \geq \dim \textcl_{\text{even}}$, so in fact this must be an isomorphism, proving the lemma.  
	\end{proof}
	Now the proof of the theorem is easy. 
	\begin{proof}[Proof of Theorem \ref{embeddingtheorem}]
		It is enough to check that the composition $\Psi_\cF \circ \Phi_\cF : \Db(S,\cl_0) \to \Db(S,\cl_0)$ is the identity functor. 
		We calculate: \begin{align*}
			\Psi_\cF \circ \Phi_\cF(\cG) &\simeq R\tau_*(\cF^\vee \otimes \cF \otimes^L_{\tau^*\cl_0} \tau^*\cG) \\
			&\simeq R\tau_*(\cF^\vee \otimes \cF) \otimes^L_{\cl_0} \cG \\
			&\simeq \cl_0 \otimes^L_{\cl_0} \cG \\
			&\simeq \cG,
		\end{align*}
		where we have used the projection formula and Lemma \ref{keylemma}. Then it follows from general category theory that the functor $\Phi_\cF$ is fully faithful. 
	\end{proof}
	\section{Semiorthogonality} \label{semiorthogonalitysection}
	We now specialize to $k = 2$, keeping our running assumptions that $2 \leq k \leq n / 2$ and that every fiber of $\cQ \to S$ has rank at least $2k - 1$. As a preliminary step in proving Theorem \ref{sodtheorem}, we must first show that the other components which arise in the decomposition arise from fully faithful embeddings $\Db(S) \hookrightarrow \Db(\OGr(2,\cQ))$. We therefore prove: 
	\begin{proposition} \label{embeddingprop}
		The functors \begin{equation}
			\begin{aligned}
				\Phi_{\Sym^p\cU_2^\vee} : \Db(S) &\to \Db(\OGr(2,\cQ)) \\
				\cG &\mapsto \tau^*\cG \otimes \Sym^p \cU_2^\vee
			\end{aligned}
		\end{equation} 
		are fully faithful when $p \leq n/2 - 2$. 
	\end{proposition}
	When $S = \Spec(\bbC)$ is a point, this is simply the statement that $\Sym^p \cU_k^\vee$ is exceptional, and in fact the proof will reduce to showing fiberwise exceptionality. 
	\begin{proof}
		As in the proof of Theorem \ref{embeddingtheorem}, we may calculate the composition with the adjoint, which is easily seen to be the functor which sends $\cG \in \Db(S)$ to 
		\begin{equation*}
			R\tau_*(\tau^*\cG \otimes \Sym^p \cU_2^\vee \otimes \Sym^p \cU_2) \simeq \cG \otimes R\tau_*(\Sym^p \cU_2^\vee \otimes \Sym^p \cU_2).
		\end{equation*}
		To show fully faithfulness it therefore suffices to check that $\cO\simeq R\tau_*(\Sym^p \cU_2 \otimes \Sym^p \cU_2^\vee)$, which follows from the isomorphism \eqref{vl5e1} of Lemma \ref{vanishinglemma5} by taking $t = q = p$.
	\end{proof}
	To check semiorthogonality between the various components, we will need to separate the cases between the Clifford components, which are twists of the embedding $\Phi_\cF$ of Theorem \ref{embeddingtheorem} by a line bundle, and the non-Clifford components, which are twists of the embeddings $\Phi_{\Sym^p \cU_2^\vee}$ of Proposition \ref{embeddingprop} by a line bundle. 
	\subsection{Non-Clifford component with non-Clifford component}
	The easiest case to consider is the following: 
	\begin{proposition} \label{noncliffnoncliffprop}
		Let $n > 4$. For any $\cG_1, \cG_2 \in \Db(S)$, if $0 \leq p, q \leq n - 2$ and $0 \leq t \leq n - 4$, then
		\begin{equation}
			R\Hom(\tau^*\cG_1 \otimes \Sym^p \cU_2^\vee(t), \tau^*\cG_2 \otimes \Sym^q \cU_2^\vee) = 0
		\end{equation}
		on $\OGr(2,\cQ)$, unless $p \leq q$ and $t = 0$ , or $p = q = n/2 - 2$ and $t = n/2 - 2$ or $n/2 - 1$. 
	\end{proposition}
	\begin{proof}
		Under the isomorphism 
		\[R\Hom(\tau^*\cG_1 \otimes \Sym^p \cU_2^\vee(t), \tau^*\cG_2 \otimes \Sym^q \cU_2^\vee) \simeq R\Hom(\cG_1,\cG_2 \otimes R\tau_*(\Sym^p \cU_2 \otimes \Sym^q \cU_2^\vee(-t))),\]
		this reduces to the statement of Lemma \ref{vanishinglemma4}.
	\end{proof}
	\subsection{Non-Clifford component with Clifford component}
	The second case is more subtle. As key input, we need an analogous vanishing statement between the kernels for the Clifford and non-Clifford components. 
	\begin{lemma} \label{cliffnonclifflemma}
		Suppose $n > 4$. For any $0 \leq q \leq n/2 - 2$ and any $0 \leq t \leq n - 4$, 
		\begin{equation}
			R\tau_*(\cF^\vee(-t) \otimes \Sym^q \cU_2^\vee) = 0.
		\end{equation}
	\end{lemma}
	\begin{proof}
		As in the proof of fully faithfulness for the Clifford component, we may exploit the left and right resolutions \eqref{leftres} and \eqref{rightres} (or rather their duals) in order to prove the claim. \par 
		Suppose first that $0 \leq t \leq n - 5$ (which is a non-empty condition since $n > 4$). Then $\cF^\vee(-t) \otimes \Sym^q \cU_2^\vee$ admits the left resolution 
		\begin{equation} \label{cliffnoncliffleftres}
			\cdots \to \tau^*\cl_{3}^\vee \otimes \cU_2 \otimes \Sym^q \cU_2^\vee(-t-1) \to \tau^*\cl_2^\vee \otimes \Sym^q \cU_2^\vee(-t-1) \to \cF^\vee \otimes \Sym^q \cU_2^\vee(-t)\to 0,
		\end{equation}
		which is in the $p$th degree given by $\tau^*\cl_{2+p}^\vee \otimes \Sym^p \cU_2 \otimes \Sym^q \cU_2^\vee(-t-1)$. By Lemma \ref{vanishinglemma3}, we know that $R^i\tau_*(\Sym^p \cU_2 \otimes \Sym^q\cU_2^\vee(-t-1)) = 0$ whenever $i > p + t + 1$. As in the proof of Lemma \ref{keylemma}, it follows that $R^i\tau_*(\cF^\vee(-t) \otimes \Sym^q \cU_2^\vee) = 0$ for all $i > t + 1$ by checking this for each stupid truncation of the resolution above. \par 
		Let us pass now to the right resolution of $\cF^\vee(-t) \otimes \Sym^q\cU_2^\vee(-t-1)$, which is given by 
		\begin{equation} \label{cliffnoncliffrightres}
			0 \to \cF^\vee \otimes \Sym^q \cU_2^\vee(-t) \to \tau^*\cl_0^\vee \otimes \Sym^q \cU_2^\vee(-t) \to \tau^*\cl_{-1}^\vee \otimes \cU_2^\vee \otimes \Sym^q \cU_2^\vee(-t) \to \cdots ,
		\end{equation}
		whose $p$th degree is $\tau^*\cl_{-p}^\vee \otimes \Sym^p \cU_2^\vee \otimes \Sym^q \cU_2^\vee(-t)$. Now by Lemma \ref{vanishinglemma2}, the higher derived pushforwards vanish for every term in the resolution, so it may be used to compute the derived pushforward of $\cF^\vee \otimes \Sym^q\cU_2^\vee(-t)$, and moreover $\tau_*(\Sym^p \cU_2^\vee \otimes \Sym^q \cU_2^\vee(-t)) = 0$ for $p < t$. We therefore conclude that $R^i\tau_*(\cF^\vee(-t) \otimes \Sym^q \cU_2^\vee) = 0$ for all $i \neq t, t+1$. \par 
		In fact, it is enough to show that $R^t\tau_*(\cF^\vee(-t) \otimes \Sym^q \cU_2^\vee) = 0$ as then Grothendieck's theorem for cohomology and base change combined with Proposition \ref{constancy2} will show that $R^{t+1}\tau_*(\cF^\vee(-t)\otimes \Sym^q\cU_2^\vee)$ vanishes as well. By considering the resolution \eqref{cliffnoncliffrightres}, it suffices to show that the map 
		\begin{equation}\label{cliffnoncliffmap}
			\cl_{-t}^\vee \otimes \tau_*(\Sym^t \cU_2^\vee \otimes \Sym^q \cU_2^\vee(-t)) \to \cl_{-t-1}^\vee \otimes\tau_*(\Sym^{t+1}\cU_2^\vee \otimes \Sym^q \cU_2^\vee(-t))
		\end{equation}
		is injective. \par 
		If $q < t$, then the second part of Lemma \ref{vanishinglemma2} implies that in fact $\tau_*(\Sym^t \cU_2^\vee \otimes \Sym^q \cU_2^\vee(-t)) = 0$, so in this case the result follows automatically. \par 
		If $t \leq q$, then we may use the isomorphisms \eqref{vl5e1} and \eqref{vl5e2} of Lemma \ref{vanishinglemma5} to rewrite the map \eqref{cliffnoncliffmap} as  the natural map 
		\begin{multline*}
				\cl_{-t}^\vee \otimes \Sym^{q-t} \cE^\vee/\im(\Sym^{q-t-2}\cE^\vee \otimes \tau^*\cL \to \Sym^{q-t}\cE^\vee) \to \\
				\cl_{-t-1}^\vee \otimes \cE^\vee \otimes \Sym^{q-t} \cE^\vee /  \im(\left(\substack{\cE^\vee \otimes\Sym^{q-t-2} \cE^\vee \\ \oplus \\ \Sym^{q-t-1}\cE^\vee}\right) \otimes \tau^*\cL \to \cE^\vee \otimes \Sym^{q-t}\cE^\vee)
		\end{multline*}
		given by Clifford multiplication. We will show that this map is fiberwise injective. \par 
		Fix any point $s \in S$, and let $V \coloneq \cE|_s$. To show fiberwise injectivity, it is enough to show that the map 
		\begin{multline} \label{fiberwisecliffnoncliffmap}
		\textcl^\vee \otimes \Sym^{q-t} V^\vee/\im(\Sym^{q-t-2}V^\vee \to \Sym^{q-t}V^\vee) \to \\
		\textcl^\vee \otimes V^\vee \otimes \Sym^{q-t} V^\vee /  \im(\left(\substack{V^\vee \otimes\Sym^{q-t-2} V^\vee \\ \oplus \\ \Sym^{q-t-1}V^\vee}\right)\to V^\vee \otimes \Sym^{q-t}V^\vee),
		\end{multline}
		where $\textcl$ is as before the full Clifford algebra; the Clifford multiplication map above naturally permutes the $\bbZ/2$-grading, so taking graded pieces will recover the desired fiberwise result.\par 
		Fix a basis $e_1, \dots, e_n$ of $V$ which diagonalizes the quadratic form $Q \coloneq \cQ_s$, and let $x_1,\dots, x_n$ be the dual basis, so that we may write the quadratic form as $\sum_{i=1}^n \lambda_i x_i^2$. By the assumption that the rank of the quadratic form is $\geq 3$, we may assume that $\lambda_{n-2} = \lambda_{n-1} = \lambda_n = 1$. In the quotient \[\Sym^{q-t} V^\vee / \im(\Sym^{q-t-2} V^\vee \xrightarrow{\cdot \sum_{i=1}^n\lambda_i x_i^2} \Sym^{q-t}V^\vee),\] it follows that elements of the form
		\[x_{i_1} \cdots x_{i_{q-t}} \text{ and } x_{j_1}\cdots x_{j_{q-t-1}}x_n,\]
		where $1 \leq i_1 \leq i_2 \leq \cdots i_{q-t} \leq n -1 $ and $1 \leq j_1 \leq j_2 \leq \cdots j_{q-t-1} \leq n - 1$, form a basis; indeed, by using the quadratic form it is possible to rewrite any terms with degree $\geq 2$ in the variable $x_n$. \par 
		On the other hand,
		\[V^\vee \otimes \Sym^{q-t} V^\vee /  \im(\left(\substack{V^\vee \otimes\Sym^{q-t-2} V^\vee \\ \oplus \\ \Sym^{q-t-1}V^\vee}\right)\to V^\vee \otimes \Sym^{q-t}V^\vee)\]
		may be understood explicitly as the quotient of $V^\vee \otimes \Sym^{q-t}V^\vee$ by the relations 
		\begin{equation}\label{tensorprodrels}
			\sum_{i=1}^n x_j \otimes x_{i_1}\cdots x_{i_{q-t-2}}\lambda_ix_i^2 \qquad \text{ and } \qquad  \sum_{i=1}^n \lambda_i x_i \otimes x_i x_{j_1}\cdots x_{j_{q-t-1}}
		\end{equation}
		where $1 \leq j \leq n$, $1 \leq {i_1} \leq \cdots \leq {i_{q-t-2}} \leq n$ and $1 \leq j_1 \leq \cdots \leq j_{q-t-1} \leq n$. Using these relations, this quotient space is spanned by terms of the following three types: 
		\begin{equation} \label{bigbasis}
		\begin{gathered}
			x_i \otimes x_{j_1}\cdots x_{j_{q-t}} \text{ for } 1 \leq j_1 \leq \cdots \leq j_{q-t} \leq n-1 \text{ and }1 \leq i \leq n, \\
			x_k \otimes x_{\ell_1} \cdots x_{\ell_{q-t-1}}x_n \text{ for } 1 \leq \ell_1 \leq \cdots \leq \ell_{q-t-1} \leq n-2 \text{ and }1 \leq k \leq n - 1,\\
			x_r \otimes x_{s_1} \cdots x_{s_{q-t-2}}x_{n-1}x_n \text{ for } 1 \leq s_1 \leq \cdots \leq s_{q-t-2} \leq n-1 \text{ and }1 \leq r \leq n - 2.
		\end{gathered}
		\end{equation}
		Indeed, using the first relation of \eqref{tensorprodrels} we may ensure that the second term of the tensor product has degree $\leq 1$ in the variable $x_n$; using the second, it is possible to rewrite terms of the form 
		\[x_n \otimes x_{\ell_1} \cdots x_{\ell_{q-t-1}}x_n\]
		to remove the occurrences of $x_n$; and combining both relations, it is possible to rewrite terms of the form 
		\[x_{n-1} \otimes x_{s_1}\cdots x_{s_{q-t-2}}x_{n-1}x_n\]
		to remove the occurrence of $x_{n-1}$ in the first term of the tensor product and the occurrence of $x_n$ in the second term of the tensor product. As a dimension count, one can observe that the predicted dimension of this quotient, by the Koszul resolution, is 
		\[n{n-1 + q - t \choose q - t} - {n-1+q-t-1 \choose q-t-1}-n{n-1+q-t-2\choose q-t-2} + {n-1+q-t-3\choose q-t-3}\]
		which iterated application of Pascal's identity coincides with the length of the spanning set above, \[n{n-2+q-t \choose q-t} + (n-1){n-3+q-t-1\choose q-t-1} + (n-2){n-3+q-t-2\choose q-t-2}.\]
		In particular this spanning set is a basis. To prove injectivity of \eqref{fiberwisecliffnoncliffmap}, we therefore write the map in terms of this basis. \par 
		Pick an arbitrary element 
		\[\sum_{1 \leq i_1 \leq \cdots \leq i_{q-t} \leq n - 1} \varphi_{i_1\cdots i_{q-t}} \otimes x_{i_1}\cdots x_{i_{q-t}} + \sum_{1 \leq j_1 \leq \cdots \leq j_{q-t-1} \leq n - 1} \psi_{j_1\cdots j_{q-t-1}}\otimes x_{j_1}\cdots x_{j_{q-t-1}}x_n\]
		in the kernel of \eqref{fiberwisecliffnoncliffmap}, and observe that under this map the element above is sent to 
		\[\sum_{\substack{1 \leq i_1 \leq \cdots \leq i_{q-t} \leq n - 1 \\1 \leq j \leq n}} \varphi_{i_1\cdots i_{q-t}}e_j \otimes x_j \otimes x_{i_1}\cdots x_{i_{q-t}} + \sum_{\substack{1 \leq j_1 \leq \cdots \leq j_{q-t-1} \leq n - 1 \\ 1 \leq j \leq n}} \psi_{j_1\cdots j_{q-t-1}}e_j \otimes x_j\otimes x_{j_1}\cdots x_{j_{q-t-1}}x_n.\]
		One can carefully rewrite this in terms of the basis above using relations \eqref{tensorprodrels}, to yield the following messy expression: 
		\begin{equation} \label{basisexpansion}
		\begin{gathered}
			\sum_{\substack{1 \leq i_1 \leq \cdots \leq i_{q-t} \leq n - 1 \\1 \leq j \leq n}} \varphi_{i_1\cdots i_{q-t}}e_j \otimes x_j \otimes x_{i_1}\cdots x_{i_{q-t}} \\
			+  \sum_{\substack{1 \leq j_1 \leq \cdots \leq j_{q-t-1} \leq n - 2 \\ 1 \leq j \leq n - 1}} \psi_{j_1\cdots j_{q-t-1}}e_j \otimes x_j\otimes x_{j_1}\cdots x_{j_{q-t-1}}x_n \\
			+ \sum_{\substack{1 \leq j_1 \leq \cdots \leq j_{q-t-2}\leq n -1 \\ 1 \leq j \leq n - 2}}\psi_{j_1\cdots j_{q-t-2},n-1}e_j \otimes x_j \otimes x_{j_1}\cdots x_{j_{q-t-2}}x_{n-1}x_n \\
			+ \sum_{\substack{1 \leq j_1 \leq \cdots \leq j_{q-t-1}\leq n -2 \\ 1 \leq j \leq n - 1}}-\lambda_j\psi_{j_1\cdots j_{q-t-1}}e_n \otimes x_j \otimes x_{j_1}\cdots x_{j_{q-t-1}}x_j \\
			+ \sum_{\substack{1 \leq j_1 \leq \cdots \leq j_{q-t-2}\leq n -1 \\ 1 \leq j \leq n - 2}}-\lambda_j\psi_{j_1\cdots j_{q-t-2},n-1}e_{n-1} \otimes x_j \otimes x_{j_1}\cdots x_{j_{q-t-2}}x_{j}x_n \\
			+ \sum_{\substack{1 \leq j_1 \leq \cdots \leq j_{q-t-2}\leq n -1 \\ 1 \leq j \leq n - 1}}-\lambda_j\psi_{j_1\cdots j_{q-t-2},n-1}e_{n} \otimes x_j \otimes x_{j_1}\cdots x_{j_{q-t-2}}x_{n-1}x_j \\
			+ \sum_{\substack{1 \leq j_1 \leq \cdots \leq j_{q-t-2}\leq n -1 \\ 1 \leq j \leq n - 2}}\lambda_j\psi_{j_1\cdots j_{q-t-2},n-1}e_{n-1} \otimes x_n \otimes x_{j_1}\cdots x_{j_{q-t-2}}x_j^2.
		\end{gathered}
		\end{equation}
		Since this element vanishes, it follows that it must vanish in each component coming from the basis \eqref{bigbasis}. \par We prove first that $\psi_{j_1\cdots j_{q-t-1}} = 0$ by induction on the number of occurrences of $n - 1$ in $j_1, \dots, j_{q-t-1}$. Observe first that the only terms attached to the tensor ${x_{n-1}}\otimes x_{j_1}\cdots x_{ j_{q-t-1}}x_n$ for $1 \leq j_1 \leq \cdots \leq j_{q-t-1} \leq n - 2$ arise in the second line of \eqref{basisexpansion}, so we find that $\psi_{j_1\cdots j_{q-t-1}}e_{n-1} = 0 \text{ for } 1 \leq j_1 \leq \cdots \leq j_{q-t-1} \leq n - 2$. Since $\lambda_{n-1} = e_{n-1}^2 = 1$, it follows that
		\[\psi_{j_1\cdots j_{q-t-1}} = 0 \text{ for } 1 \leq j_1 \leq \cdots \leq j_{q-t-1} \leq n - 2.\]
		On the other hand, observe that for any $0 \leq k \leq q - t - 2$, the only remaining terms which are attached to tensors of the form 
		$x_{n-2} \otimes x_{j_1}\cdots x_{j_k}x_{n-2}x_{n-1}^{q-t-2-k}x_n$ where $1 \leq j_1 \leq \cdots \leq j_k \leq n - 2$ are in the third and fifth lines of \eqref{basisexpansion}. If $k = q - t - 2$, then the only term attached to this tensor is 
		\[-\lambda_{n-2}\psi_{j_1\cdots j_{q-t-2},n-1}e_{n-1} \otimes x_{n-2} \otimes x_{j_1}\cdots x_{j_{q-t-2}}x_{n-2}x_n\]
		in the fifth row, so we find that $\lambda_{n-2}\psi_{j_1\cdots j_{q-t-2},n-1}e_{n-1} = 0$, and as $\lambda_{n-2} = \lambda_{n-1} = 1$ it follows that 
		\[\psi_{j_1\cdots j_{q-t-2},n-1} = 0 \text{ for } 1 \leq j_1 \leq \cdots \leq j_{q-t-2} \leq n - 2.\]
		If $k < q - t - 2$, then by examining terms attached to the tensor $x_{n-2} \otimes x_{j_1}\cdots x_{j_k}x_{n-2}x_{n-1}^{q-t-2-k}x_n$ we instead find 
		\[\psi_{j_1\cdots j_k,n-2,n-1,\cdots,n-1}e_{n-2} - \lambda_{n-2}\psi_{j_1\cdots,j_{k},n-1,\cdots,n-1}e_{n-1} = 0.\]
		Assuming by induction that the first term vanishes and using that $\lambda_{n-2} = \lambda_{n-1} = 1$, we may conclude that
		\[\psi_{j_1\cdots j_k,n-1,\cdots,n-1} = 0 \text{ for } 1 \leq j_1 \leq \cdots \leq j_k \leq n - 2.\]
		Hence it follows $\psi_{j_1,\cdots,j_{q-t-1}} = 0 $ whenever $1 \leq j_1 \leq \cdots \leq j_{q-t-1} \leq n - 1$. The expression \eqref{basisexpansion} therefore reduces simply to 
		\[\sum_{\substack{1 \leq i_1 \leq \cdots \leq i_{q-t}\leq n- 1 \\ 1 \leq j \leq n}}\varphi_{i_1\cdots i_{q-t}}e_j\otimes x_j \otimes x_{i_1\cdots i_{q-t}},\]
		and as every basis element occurs at most one, we may conclude that $\varphi_{i_1\cdots i_{q-t}}e_n = 0$. Since $\lambda_n = 1$, it follows that $\varphi_{i_1\cdots i_{q-t}} = 0$. In particular, the map \eqref{fiberwisecliffnoncliffmap} is injective, which proves the lemma whenever $0 \leq t \leq n - 5$. \par 
		We are left only to handle the case when $t = n - 4$. Then by Grothendieck duality and the calculation of the dualizing line bundle \eqref{dualizing}, there is a quasi-isomorphism 
		\[R\tau_*(\cF^\vee(-n+4) \otimes \Sym^q \cU_2^\vee) \simeq \left(R\tau_*(\cF(-1) \otimes \Sym^q \cU_2)\right)^\vee[-2n+7] \otimes \cL^{\otimes 3},\]
		so it suffices to show that 
		\begin{equation}
			R\tau_*(\cF(-1) \otimes \Sym^q \cU_2) = R\tau_*(\cF \otimes \Sym^q \cU_2^\vee(-q-1)) = 0
		\end{equation}
		whenever $0 \leq q \leq n/2 - 2$. Using the left and right resolution \eqref{leftres} and \eqref{rightres},  $\cF\otimes \Sym^q \cU_2^\vee(-q-1)$ admits left and right resolutions 
		\[\cdots \to \tau^*\cl_{-1} \otimes \cU_2 \otimes \Sym^q \cU_2^\vee(-q-1) \to \tau^*\cl_0 \otimes \Sym^q\cU_2^\vee(-q-1) \to \cF \otimes \Sym^q \cU_2^\vee(-q-1) \to 0\]
		and 
		\[0 \to \cF \otimes \Sym^q \cU_2^\vee(-q-1) \to \tau^*\cl_2 \otimes \Sym^q \cU_2^\vee(-q) \to \tau^*\cl_{3} \otimes \cU_2^\vee \otimes \Sym^q \cU_2^\vee(-q) \to \cdots.\]
		Comparing these resolutions to the resolutions \eqref{cliffnoncliffleftres} and \eqref{cliffnoncliffrightres} above, one can see that the same arguments go through, proving the lemma. 
	\end{proof}
	Semiorthogonality then follows easily.
	\begin{proposition} \label{cliffnoncliffprop}
		Suppose $n > 4$. For any $\cG_1 \in \Db(S,\cl_0)$ and $\cG_2 \in \Db(S)$, if $0 \leq q \leq n/2 - 2$ and $0 \leq t \leq n - 4$, then 
		\[R\Hom((\cF \otimes_{\tau^*\cl_0}\tau^*\cG_1)(t), \tau^*\cG_2 \otimes \Sym^q\cU_2^\vee) = 0\]
		on $\OGr(2,\cQ)$. 
	\end{proposition}
	\begin{proof}
		Under the isomorphisms
		\begin{align*}
			R\Hom((\cF \otimes_{\tau^*\cl_0}\tau^*\cG_1)(t), \tau^*\cG_2 \otimes \Sym^q\cU_2^\vee) &\simeq R\Hom_{\tau^*\cl_0}(\tau^*\cG_1, \cF^\vee \otimes \tau^*\cG_2 \otimes \Sym^q \cU_2^\vee(-t)) \\
			&\simeq R\Hom_{\cl_0}(\cG_1, \tau_*(\cF^\vee \otimes \Sym^q \cU_2^\vee(-t)) \otimes \cG_2),
		\end{align*}
		this reduces to Lemma \ref{cliffnonclifflemma}.
	\end{proof}
	\subsection{Clifford component with Clifford component} The final case is very similar to the proof of fully faithfulness in Theorem \ref{embeddingtheorem}. As in the previous two cases, the meat of the argument will come from a vanishing statement for certain derived pushforwards. 
	\begin{lemma} \label{cliffclifflemma}
		For any $0 < t \leq n - 4$, 
		\begin{equation}
			R\tau_*(\cF^\vee \otimes \cF(-t))) = 0.
		\end{equation}
	\end{lemma}
	\begin{proof}
%		By Grothendieck duality and the calculation of the dualizing line bundle in \eqref{dualizing} there is an isomorphism 
%		\[R\tau_*(\cF^\vee \otimes \cF(-t)) \simeq (R\tau_*(\cF^\vee \otimes \cF(t-n+3)))^\vee[-2n+7] \otimes \cL^{\otimes 3},\]
%		so to show vanishing we may assume that $t < n/2 - 1$. 
		By twisting the resolutions \eqref{lefttensorres} and \eqref{righttensorres} appearing in Lemma \ref{keylemma}, we observe that the above sheaf has resolutions given by
		\begin{gather}
			\cdots \to \begin{matrix}
				\tau^*\cl_{3}^\vee \otimes \cU_2(-t-1) \otimes \tau^*\cl_0 \\
				\oplus \\
				\tau^*\cl_2^\vee \otimes \cU_2(-t-1) \otimes \tau^*\cl_{-1}
			\end{matrix}\to \tau^*\cl_2^\vee \otimes \cO_{\OGr(2,\cQ)}(-t-1) \otimes \tau^* \cl_0 \to \cF^\vee \otimes \cF(-t) \to 0  \label{lefttwistedtensorres}\\
			0 \to \cF^\vee \otimes \cF(-t) \to \tau^*\cl_0^\vee \otimes \cO_{\OGr(2,\cQ)}(-t+1) \otimes \tau^*\cl_2 \to \begin{matrix}
				\tau^*\cl_{-1}^\vee \otimes \cU_2^\vee(-t+1) \otimes \tau^*\cl_2 \\
				\oplus \\
				\tau^*\cl_0^\vee \otimes \cU_2^\vee(-t+1) \otimes \tau^*\cl_{3}
			\end{matrix} \to \cdots \label{righttwistedtensorres}
		\end{gather}
		If we let $A^\bullet_{\text{left}}$ denote the left resolution \eqref{lefttwistedtensorres}, then the $\ell$th term $A^{-\ell}_{\text{left}}$ is always of the form
		\[\bigoplus_{\substack{\text{finitely}\\\text{many }i,j}}\tau^*\cl_i^\vee \otimes \Sym^p \cU_2 \otimes \Sym^q \cU_2(-t-1) \otimes \tau^*\cl_j\]
		where $p + q = \ell$, and using Lemma \ref{vanishinglemma1}, we find that $R^i\tau_*(A^{-\ell}_{\text{left}}) = 0$ for $i > \ell + 2t + 1$. As in the proof of Lemma \ref{keylemma}, by passing to the stupid truncations and then to the colimit, we see that $R^i\tau_*(\cF^\vee \otimes \cF(-t))$ whenever $i > 2t + 1$. \par 
		Examining the right resolution $A^\bullet_{\text{right}}$ of \eqref{righttwistedtensorres}, we see that 
		\[A^\ell_{\text{right}} = \bigoplus_{p=0}^{\ell}\tau^*\cl_{-p}^\vee \otimes \Sym^p \cU_2^\vee \otimes \Sym^{\ell-p} \cU_2^\vee(-t+1) \otimes \tau^*\cl_{2+\ell-p}.\]
		By Lemma \ref{vanishinglemma2}, $R^i\tau_*(A^\ell_{\text{right}}) = 0$ whenever $i > 0$, and $A^\bullet_{\text{right}}$ can be used to compute cohomology. If $\ell < 2t - 2$, then it follows that $\min(p,\ell-p) < t - 1$, so in this case $\tau_*(A^\ell_{\text{right}}) = 0$ vanishes completely. It follows that $R^i\tau_*(\cF^\vee \otimes \cF(-t)) = 0$ except for $i = 2t-2, 2t - 1, 2t, 2t + 1$. The lemma then amounts to showing the exactness of 
		\begin{equation*}
			0 \to \tau_*(A^{2t-2}_\text{right}) \to \tau_*(A^{2t-1}_\text{right}) \to \tau_*(A^{2t}_\text{right}) \to \tau_*(A^{2t+1}_{\text{right}})
		\end{equation*}
		when $0 < t \leq n - 4$, as this will imply that $R^i\tau_*(\cF^\vee \otimes \cF(-t)) = 0$ for $2t -2 \leq i \leq 2t$; then by Proposition \ref{constancy2} along with Grothendieck's theorems on cohomology and base change, $R^{2t+1}\tau_*(\cF^\vee \otimes \cF(-t))$ necessarily must vanish also. \par 
		As in the previous cases, showing vanishing for the remaining range will reduce to an explicit linear algebra calculation using $A^\bullet_{\text{right}}$. Examining the summands of $A^\ell_\text{right}$ for $\ell = 2t - 2, 2t - 1, 2t$ and $2t + 1$, using Lemma \ref{vanishinglemma2} again, and applying the projection formula, we see that we have identifications
		\begin{equation}
			\begin{aligned}
				\tau_*(A^{2t-2}_{\text{right}}) &\simeq \cl_{-t+1}^\vee \otimes \tau_*(\Sym^{t-1} \cU_2^\vee \otimes \Sym^{t-1}\cU_2^\vee(-t+1)) \otimes \cl_{1+t} \\ \\
				\tau_*(A^{2t-1}_{\text{right}}) &\simeq \begin{matrix}\cl_{-t+1}^\vee \otimes \tau_*(\Sym^{t-1} \cU_2^\vee \otimes \Sym^{t}\cU_2^\vee(-t+1)) \otimes \cl_{2+t} \\ \oplus \\ \cl_{-t}^\vee \otimes \tau_*(\Sym^t \cU_2^\vee \otimes \Sym^{t-1}\cU_2^\vee(-t+1)) \otimes \cl_{1+t}\end{matrix} \\ \\
				\tau_*(A^{2t}_{\text{right}}) &\simeq \begin{matrix}\cl_{-t+1}^\vee \otimes \tau_*(\Sym^{t-1} \cU_2^\vee \otimes \Sym^{t+1}\cU_2^\vee(-t+1)) \otimes \cl_{3+t} \\ \oplus \\ \cl_{-t}^\vee \otimes \tau_*(\Sym^t \cU_2^\vee \otimes \Sym^{t}\cU_2^\vee(-t+1)) \otimes \cl_{2+t} \\ \oplus \\ \cl_{-t-1}^\vee \otimes \tau_*(\Sym^{t+1}\cU_2^\vee \otimes \Sym^{t-1}\cU_2^\vee(-t+1)) \otimes \cl_{1+t} \end{matrix} \\\\ 
				\tau_*(A^{2t+1}_{\text{right}}) &\simeq \begin{matrix}\cl_{-t+1}^\vee \otimes \tau_*(\Sym^{t-1} \cU_2^\vee \otimes \Sym^{t+2}\cU_2^\vee(-t+1)) \otimes \cl_{4+t} \\ \oplus \\ \cl_{-t}^\vee \otimes \tau_*(\Sym^t \cU_2^\vee \otimes \Sym^{t+1}\cU_2^\vee(-t+1)) \otimes \cl_{3+t} \\ \oplus \\ \cl_{-t-1}^\vee \otimes \tau_*(\Sym^{t+1}\cU_2^\vee \otimes \Sym^{t}\cU_2^\vee(-t+1)) \otimes \cl_{2+t} \\ 
				\oplus \\ \cl_{-t-2}^\vee \otimes \tau_*(\Sym^{t+2}\cU_2^\vee \otimes \Sym^{t-1}\cU_2^\vee(-t+1)) \otimes \cl_{1+t}
			\end{matrix} 
			\end{aligned}
		\end{equation}
		Remarkably, under the isomorphisms of Lemma \ref{vanishinglemma5}, each of these simplifies significantly: 
		\begin{equation}
			\begin{gathered}
				\tau_*(A^{2t-2}_{\text{right}}) \simeq \cl_{-t+1}^\vee \otimes  \tau_*\cO \otimes \cl_{1+t} \qquad \qquad
				\tau_*(A^{2t-1}_{\text{right}}) \simeq \begin{matrix}\cl_{-t+1}^\vee \otimes \tau_*(\cU_2^\vee) \otimes \cl_{2+t} \\ \oplus \\ \cl_{-t}^\vee \otimes \tau_*(\cU_2^\vee) \otimes \cl_{1+t}\end{matrix} \\ \\
				\tau_*(A^{2t}_{\text{right}}) \simeq \begin{matrix}\cl_{-t+1}^\vee \otimes \tau_*(\Sym^2 \cU_2^\vee) \otimes \cl_{3+t} \\ \oplus \\ \cl_{-t}^\vee \otimes \tau_*(\cU_2^\vee \otimes \cU_2^\vee) \otimes \cl_{2+t} \\ \oplus \\ \cl_{-t-1}^\vee \otimes \tau_*(\Sym^2 \cU_2^\vee) \otimes \cl_{1+t} \end{matrix} \qquad\qquad 
				\tau_*(A^{2t+1}_{\text{right}}) \simeq \begin{matrix}\cl_{-t+1}^\vee \otimes \tau_*(\Sym^3 \cU_2^\vee) \otimes \cl_{4+t} \\ \oplus \\ \cl_{-t}^\vee \otimes \tau_*(\cU_2^\vee \otimes \Sym^2\cU_2^\vee) \otimes \cl_{3+t} \\ \oplus \\ \cl_{-t-1}^\vee \otimes \tau_*(\Sym^2\cU_2^\vee \otimes \cU_2^\vee) \otimes \cl_{2+t} \\ 
					\oplus \\ \cl_{-t-2}^\vee \otimes \tau_*(\Sym^{3}\cU_2^\vee) \otimes \cl_{1+t}
				\end{matrix} 
			\end{gathered}
		\end{equation}
		For convenience's sake, we define the notation
		\[\Sym^p_{\SO} \cE^\vee \coloneq \Sym^p \cE^\vee / \im(\Sym^{p-2}\cE^\vee \otimes \tau^*\cL\to \Sym^p \cE^\vee)\]
		and 
		\[(\cE^\vee \otimes \Sym^{p} \cE^\vee)_{\SO} \coloneq \cE^\vee \otimes \Sym^{p} \cE^\vee/\im(\left( \substack{\cE^\vee \otimes \Sym^{p-2}\cE^\vee \\ \oplus \\ \Sym^{p-1} \cE^\vee}
		\right) \otimes \tau^*\cL \to \cE^\vee \otimes \Sym^p \cE^\vee)\]
		by analogy to the orthogonal Schur functors (and to the notation used in the proof of Lemma \ref{keylemma}), so that by Lemma \ref{vanishinglemma5} we have $\tau_*(\Sym^p \cU_2^\vee) \simeq \Sym^p_{\SO} \cE^\vee$ and $\tau_*(\cU_2^\vee \otimes \Sym^p \cU_2^\vee) \simeq (\cE^\vee \otimes \Sym^p \cE^\vee)_{\SO}$. Then the proof of the lemma reduces to showing the exactness of the complex at the terms:
		\[ \small 0 \to \cl_{-t+1}^\vee \otimes \cl_{1+t} \xrightarrow{d_0} \begin{matrix}
			\cl_{-t+1}^\vee \otimes \cE^\vee \otimes \cl_{2+t} \\ 
			\oplus \\ 
			\cl_{-t}^\vee \otimes \cE^\vee \otimes \cl_{1+t} \\
			\end{matrix} \xrightarrow{d_1} \begin{matrix}
			\cl_{-t+1}^\vee \otimes \Sym^2_{\SO} \cE^\vee \otimes \cl_{3+t} \\
			\oplus \\
			\cl_{-t}^\vee \otimes (\cE^\vee \otimes \cE^\vee)_{\SO} \otimes \cl_{2+t} \\
			\oplus \\ 
			\cl_{-t-1}^\vee \otimes \Sym^2_{\SO} \cE^\vee \otimes \cl_{1+t} 
			\end{matrix} \xrightarrow{d_2} \begin{matrix}
			\cl_{-t+1}^\vee \otimes \Sym^3_{\SO} \cE^\vee \otimes \cl_{4+t} \\
			\oplus \\
			\cl_{-t}^\vee \otimes (\cE^\vee \otimes \Sym^2\cE^\vee)_{\SO} \otimes \cl_{3+t} \\
			\oplus \\ 
			\cl_{-t-1}^\vee \otimes (\Sym^2_{\SO} \cE^\vee \otimes \cE^\vee)\otimes \cl_{2+t} \\
			\oplus \\
			\cl_{-t-2}^\vee \otimes \Sym^3_{\SO} \cE^\vee \otimes \cl_{1+t}
		\end{matrix}\]
		where all maps are induced (up to sign, coming from the Koszul sign convention for the tensor product of complexes) by monoidal coevaluation followed by Clifford multiplication, either into the left or right factor. \par 
		It is enough to check the exactness fiberwise over points $s \in S$. We therefore fix a basis $e_1, \dots, e_n$ of $V \coloneq \cE|_s$ with dual basis $x_1,\dots,x_n$, such that the quadratic form diagonalizes as $Q \coloneq \cQ_s = \sum_{i=1}^n \lambda_i x_i^2$. By our assumption that the fibers of the quadratic form have rank at least 3, we may assume that \[\lambda_1 = \lambda_2 = \lambda_3  = 1.\] We proceed case-by-case: 
		\begin{itemize}
			\item \underline{Injectivity of $d_0$:} On fibers, we may realize $d_0$ as a map 
			\[d_0 : \Hom(\textcl_{\parity(-t+1)},\textcl_{\parity(1+t)}) \to \begin{matrix}
				\Hom(\textcl_{\parity(-t+1)}^\vee,\textcl_{\parity(2+t)})\otimes V^\vee  \\ 
				\oplus \\ 
				\Hom(\textcl_{\parity(-t)}^\vee,\textcl_{\parity(1+t)}) \otimes V^\vee
				\end{matrix}\]
				Suppose $\phi \in \Hom(\textcl_{\parity(-t+1)},\textcl_{\parity(1+t)})$, and observe that 
			\[d_0(\phi) = \left(\sum_{i=1}^n e_i\phi \otimes x_i, \sum_{i=1}^n \phi e_i \otimes x_i \right).\]
			If $d_0(\phi) = 0$, then evidently $\phi e_1 = 0$; but since $\lambda_1 = 1$, it follows immediately that $\phi = 0$. 
			\item \underline{$\im(d_0) = \ker(d_1)$}: On fibers, $d_1$ is the map 
			\[d_1 : \begin{matrix} \Hom(\textcl_{\parity(-t+1)},\textcl_{\parity(2+t)}) \otimes V^\vee \\ \oplus \\ \Hom(\textcl_{\parity(-t)},\textcl_{\parity(1+t)}) \otimes V^\vee \end{matrix} \to \begin{matrix}
				\Hom(\textcl_{\parity(-t+1)},\textcl_{\parity(3+t)} )\otimes \Sym^2_{\SO} V^\vee  \\
				\oplus \\
				\Hom(\cl_{-t},\textcl_{\parity(2+t)}) \otimes (V^\vee \otimes V^\vee)_{\SO} \\
				\oplus \\ 
				\Hom(\textcl_{\parity(-t-1)} ,\textcl_{\parity(1+t)})\otimes \Sym^2_{\SO} V^\vee 
			\end{matrix}\]
			where $\Sym^2_{\SO} V^\vee$ is the quotient of $\Sym^2 V^\vee$ by the element $\sum_{i=1}^n \lambda_ix_i^2$, and $(V^\vee \otimes V^\vee)_{\SO}$ is the quotient of $V^\vee \otimes V^\vee$ by $\sum_{i=1}^n \lambda_i x_i \otimes x_i$. Suppose
			\[\left(\sum_{i=1}^n \phi_i \otimes x_i , \sum_{i=1}^n \psi_i \otimes x_i\right) \in \begin{matrix} \Hom(\textcl_{\parity(-t+1)},\textcl_{\parity(2+t)}) \otimes V^\vee \\ \oplus \\ \Hom(\textcl_{\parity(-t)},\textcl_{\parity(1+t)}) \otimes V^\vee \end{matrix} \]
			is in $\ker(d_1)$. Applying $d_1$ to this element yields 
			\[\left(\sum_{1 \leq i,j \leq n}e_j\phi_i \otimes x_ix_j, \sum_{1 \leq i,j \leq n} (\phi_ie_j - e_i\psi_j) \otimes x_j \otimes x_i, \sum_{1 \leq i,j \leq n} \psi_ie_j \otimes x_ix_j\right).\]
			From the definition of the quotients $\Sym^2_{\SO} V^\vee$ and $(V^\vee \otimes V^\vee)_{\SO}$, it follows that as elements of 
			\[\begin{matrix}
				\Hom(\textcl_{\parity(-t+1)},\textcl_{\parity(3+t)} )\otimes \Sym^2 V^\vee  \\
				\oplus \\
				\Hom(\cl_{-t},\textcl_{\parity(2+t)}) \otimes (V^\vee \otimes V^\vee) \\
				\oplus \\ 
				\Hom(\textcl_{\parity(-t-1)} ,\textcl_{\parity(1+t)})\otimes \Sym^2 V^\vee 
			\end{matrix}\]
			we have 
			\begin{align*}
				\sum_{j=1}^n\sum_{i=1}^n e_j\phi_i\otimes x_ix_j &= \alpha \otimes \sum_{i=1}^n \lambda_i x_i^2, \\
				\sum_{j=1}^n\sum_{i=1}^n (\phi_ie_j - e_i\psi_j) \otimes x_j \otimes x_i&= \beta \otimes \sum_{i=1}^n \lambda_i x_i \otimes x_i, \\
				\sum_{j=1}^n \sum_{i=1}^n \psi_ie_j \otimes x_ix_j &= \gamma \otimes \sum_{i=1}^n \lambda_i x_i^2
			\end{align*}
			for appropriate $\alpha, \beta, \gamma$. \par 
			Examining the terms in the first equation, we see first that $e_j\phi_i = -e_i\phi_j$ for any $i \neq j$, and $e_i\phi_i = \lambda_i\alpha$. Since $\lambda_1 = 1$, we may observe that $\phi_1 = e_1\alpha$, and $-e_1\phi_j = e_j\phi_1 = e_je_1\alpha = -e_1e_j\alpha$, so that $\phi_j = e_j\alpha$ for $j > 1$ as well. \par 
			Examining the second equation, we see that $\phi_ie_j - e_i\psi_j = 0$ for $i \neq j$, while $\phi_ie_i - e_i\psi_i = \lambda_i\beta$. In particular, we see that $e_1\alpha e_j = \phi_1e_j = e_1\psi_j$ for $j \neq i$, and conclude since $\lambda_1 = 1$ that $\psi_j = \alpha e_j$ for $j \neq 1$. Similarly we may see that $e_2\alpha e_1 = \phi_2e_1 = e_2\psi_1$ and conclude that $\psi_1 = \alpha e_1$. In particular it follows that 
			\[\left(\sum_{i=1}^n \phi_i \otimes x_i , \sum_{i=1}^n \psi_i \otimes x_i\right) = \left(\sum_{i=1}^n e_i\alpha \otimes x_i, \sum_{i=1}^n \alpha e_i \otimes x_i\right) = d_0(\alpha).\]
			\item \underline{$\im(d_1) = \ker(d_2)$:} On fibers, 
			\[\small 
			d_2: \begin{matrix}
				\Hom(\textcl_{\parity(-t+1)},\textcl_{\parity(3+t)} )\otimes \Sym^2_{\SO} V^\vee \\
				\oplus \\
				\Hom(\textcl_{\parity(-t)}, \textcl_{\parity(2+t)}) \otimes (V^\vee \otimes V^\vee)_{\SO}\\
				\oplus \\ 
				\Hom(\textcl_{\parity(-t-1)},\textcl_{\parity(1+t)})\otimes \Sym^2_{\SO} V^\vee 
			\end{matrix} \to \begin{matrix}
				\Hom(\textcl_{\parity(-t+1)},\textcl_{\parity(4+t)}) \otimes \Sym^3_{\SO} V^\vee \\
				\oplus \\
				\Hom(\textcl_{\parity(-t)},\textcl_{\parity(3+t)})\otimes (V^\vee \otimes \Sym^2V^\vee)_{\SO} \\
				\oplus \\ 
				\Hom(\textcl_{\parity(-t-1)},\textcl_{\parity(2+t)})\otimes (\Sym^2 V^\vee \otimes V^\vee)_{\SO} \\
				\oplus \\
				\Hom(\textcl_{\parity(-t-2)},\textcl_{\parity(1+t)})\otimes \Sym^3_{\SO} V^\vee 
				\end{matrix}
				\]
				We therefore take an arbitrary element in the domain of $d_2$, which we may write as 
				\[\left(\sum_{1 \leq i \leq j \leq n} \phi_{ij} \otimes x_ix_j, \sum_{1 \leq i,j \leq n} \psi_{ij} \otimes x_i \otimes x_j, \sum_{1 \leq i \leq j \leq n } \varphi_{ij} \otimes x_{ij}\right),\]
				apply $d_2$, and check what relations are imposed upon our element. The result of applying $d_2$ is: 
				\begin{multline*}\small \left(\sum_{\substack{1 \leq i \leq j \leq n \\ 1 \leq k \leq n}} e_k\phi_{ij} \otimes x_ix_jx_k, \sum_{k=1}^n\left(\sum_{1 \leq i \leq j \leq n}\phi_{ij}e_k \otimes x_k \otimes x_i x_j - \sum_{1 \leq i,j \leq n}e_k\psi_{ij} \otimes x_i \otimes x_jx_k \right), \right. \\
				\left. \sum_{k=1}^n \left(\sum_{1 \leq i,j \leq n}\psi_{ij}e_k \otimes x_kx_{i} \otimes x_j + \sum_{1 \leq i \leq j \leq n} e_k\varphi_{ij} \otimes x_ix_j \otimes x_k\right), \sum_{1 \leq i \leq j \leq n}\varphi_{ij}e_k \otimes x_ix_jx_k\right).
				\end{multline*}
				If our element lies in $\ker(d_2)$, then it follows that as an element of 
				\[\begin{matrix}
					\Hom(\textcl_{\parity(-t+1)},\textcl_{\parity(4+t)}) \otimes \Sym^3 V^\vee \\
					\oplus \\
					\Hom(\textcl_{\parity(-t)},\textcl_{\parity(3+t)})\otimes (V^\vee \otimes \Sym^2V^\vee) \\
					\oplus \\ 
					\Hom(\textcl_{\parity(-t-1)},\textcl_{\parity(2+t)})\otimes (\Sym^2 V^\vee \otimes V^\vee) \\
					\oplus \\
					\Hom(\textcl_{\parity(-t-2)},\textcl_{\parity(1+t)})\otimes \Sym^3V^\vee 
				\end{matrix}
				\]
				it necessarily coincides with an element of the form 
				\begin{multline*}\left(\sum_{1 \leq i,j \leq n}\alpha_j \otimes \lambda_i x_i^2x_j, \sum_{1 \leq i,j \leq n}\gamma_j \otimes x_j \otimes \lambda_i x_i^2 + \sum_{1 \leq i,j \leq n} \delta_j \otimes \lambda_ix_i \otimes x_ix_j\right. \\ \left. \sum_{1 \leq i,j \leq n} \delta_j' \otimes \lambda_i x_ix_j \otimes x_i + \sum_{1 \leq i,j \leq n} \gamma_j' \otimes \lambda_i x_i^2 \otimes x_j, \sum_{1 \leq i,j \leq n}\beta_j \otimes  \lambda_i x_i^2x_j \right)
				\end{multline*}
				Examining the first component and using that $\lambda_1 = 1$, it is easy to check that $\phi_{11} = e_1\alpha_1$ by examining the coefficient of $x_1^3$. Then for any $i \neq 1$, by examining the coefficient of $x_1^2x_i$, we have an identification $e_1\phi_{1i} + e_i\phi_{11} = \alpha_i$, so $\phi_{1i} = e_1\alpha_i - e_1e_i\phi_{11} = e_1\alpha_i - e_1e_ie_1\alpha_1$, and we may conclude that $\phi_{1i} = e_1\alpha_i + e_i\alpha_1$. Finally, for any $1 < i < j$, examining the coefficient of $x_1x_ix_j$ gives $e_1\phi_{ij} + e_j\phi_{1i} + e_i\phi_{1j} = 0$, so that 
				\[e_1\phi_{ij} = -e_je_1\alpha_i-e_je_i\alpha_1 - e_ie_1\alpha_j - e_ie_j\alpha_1 = e_1e_j\alpha_i + e_1e_i\alpha_j,\]
				so that multiplication by $e_1$ gives 
				\begin{equation} \label{relation1}
					\phi_{ij} = e_i\alpha_j + e_j\alpha_i.
				\end{equation}
				Since $\lambda_2 = \lambda_3 = 1$ as well, repeating the same arguments shows that the same holds for any $i \neq j$ not necessarily distinct from 1. By looking at the coefficient to $x_i^2x_1$ or $x_i^2x_2$, a similar argument then shows that 
				\begin{equation} \label{relation2}
					\phi_{ii} = e_i\alpha_i
				\end{equation}
				for any $1 \leq i \leq n$. Examining the last component and repeating the same arguments, we see that \begin{equation} \label{relation3}
					\varphi_{ij} = \beta_ie_j + \beta_je_i
					\end{equation}
				for $1 \leq i \neq j \leq n$ and \begin{equation} \label{relation4}
					\varphi_{ii} = \beta_i e_i\end{equation}
				for $1 \leq i \leq n$. \par 
				Moving onto the second component and checking the component corresponding to $x_i \otimes x_1^2$ for $ i \neq 1$, we see that $\phi_{11}e_i-e_1\psi_{i1} = \gamma_i$, so that $\psi_{i1} = \alpha_1e_i - e_1\gamma_i$. Then for any $i \neq j$ distinct from 1, examining the component for $x_i \otimes x_1x_j$, we see $\phi_{1j}e_i - e_1\psi_{ij} - e_j\psi_{i1} = 0$, and conclude using left multiplication by $e_1$ that
				\begin{equation} \label{relation5}
					\psi_{ij} = \alpha_je_i - e_j\gamma_i.
				\end{equation}
				As before, since $\lambda_2 = \lambda_3 = 1$, the assumption that $ i, j$ are distinct from 1 can be removed. Now for any $i \neq 1$, examining the component associated to $x_i \otimes x_1x_i$, we see $\phi_{1i}e_i - e_i\psi_{i1} - e_1\psi_{ii} = \lambda_i\delta_1$, so 
				$\psi_{ii} = e_1\phi_{1i}e_i - e_1e_i\psi_{i1}-\lambda_ie_1\delta_1 = \alpha_ie_i - e_i\gamma_i - \lambda_ie_1\delta_1.$
				Looking at $x_1 \otimes x_1^2$, a similar argument shows that $\psi_{11} = \alpha_1e_1 - e_1\gamma_1 - \lambda_1e_1\delta_1$, so that 
				\begin{equation} \label{relation6}
					\psi_{ii} = \alpha_ie_i - e_i\gamma_i - \lambda_i e_1\delta_1
				\end{equation}
				for all $1 \leq i \leq n$. \par 
				We turn now to the third component. Repeating the arguments of the previous paragraph, we see that 
				\[\psi_{ij} = -e_j\beta_i + \gamma_j'e_i\]
				for $i \neq j$, and 
				\[\psi_{ii} =  -e_i\beta_i+\gamma_i'e_i+\lambda_i\delta_1'e_1.\]
				In particular, for any $i \neq 1$, 
				\[\alpha_1e_i - e_1\gamma_i = \psi_{i1} = -e_1\beta_i + \gamma_1'e_i,\]
				and by rearranging terms and right multiplying by $e_j$ for $j \neq 1,i$,  we find 
				\[e_1(-\beta_i+\gamma_i)e_j = (\alpha_1-\gamma_1')e_ie_j = -(\alpha_1-\gamma_1')e_je_i = -e_1(-\beta_j+\gamma_j)e_i.\]
				Left multiplication by $e_1$ and rearrangement then allows us to conclude that 
				\begin{equation} \label{relation7}
					\gamma_ie_j+\gamma_je_i = \beta_ie_j + \beta_je_i 
				\end{equation}
				for any $i \neq j$ distinct from $1$; as before we may remove this assumption by repeating the argument using that $\lambda_2 = 1$ or $\lambda_3 = 1$ instead. \par 
				On the other hand, right multiplying by $e_i$ rather than $e_j$ above before left multiplying by $e_1$ gives
				\[(-\beta_i+\gamma_i)e_i = \lambda_ie_1(\alpha_1-\gamma_1')\]
				for all $i \neq 1$. Since $\lambda_2 = 1 = \lambda_3$, we may in fact conclude that 
				\[(-\beta_2+\gamma_2)e_2 = e_1(\alpha_1-\gamma_1')=(-\beta_3+\gamma_3)e_3.\]
				However, the same argument starting with the relation coming from the two different expressions for $\psi_{i2}$ shows that 
				\[(-\beta_1+\gamma_1)e_1= e_2(\alpha_2-\gamma_2') = (-\beta_3+\gamma_3)e_3,\]
				so in fact it follows that 
				\[(-\beta_1+\gamma_1)e_1 = e_1(\alpha_1-\gamma_1'),\]
				and
				\begin{equation} \label{relation8}
					\gamma_ie_i = \beta_ie_i + \lambda_ie_1(\alpha_1-\gamma_1') 
				\end{equation}
				for all $1 \leq i \leq n$. \par 
				Consider now the image 
				\[d_1\left(\sum_{i=1}^n \alpha_i \otimes x_i, \sum_{i=1}^n \gamma_i \otimes x_i\right),\]
				which we may write as 
				\[\left(\sum_{1 \leq i,j \leq n}e_j\alpha_i \otimes x_ix_j, \sum_{1\leq i,j \leq n}(\alpha_je_i - e_j\gamma_i)\otimes x_i \otimes x_j, \sum_{1 \leq i,j \leq n}\gamma_ie_j \otimes x_ix_j\right),\]
				and observe that by \eqref{relation1} and \eqref{relation2}, the first term coincides with \[\sum_{1 \leq i \leq j \leq n}\phi_{ij} \otimes x_ix_j;\] 
				by \eqref{relation5} and \eqref{relation6}, we see that 
				\[\sum_{1 \leq i,j \leq n}(\alpha_je_i-e_j\gamma_i) \otimes x_i \otimes x_j = \sum_{1 \leq i,j \leq n} \psi_{ij} \otimes x_i \otimes x_j + e_1\delta_1 \otimes \sum_{1 \leq i \leq n}\lambda_i x_i \otimes x_i;\]
				and by combining relations \eqref{relation3}, \eqref{relation4}, \eqref{relation7} and \eqref{relation8}, 
				\begin{align*}
					\sum_{1 \leq i,j \leq n}\gamma_ie_j \otimes x_ix_j &= \sum_{1 \leq i,j \leq n}\beta_i e_j \otimes x_ix_j + e_1(\alpha_1-\gamma_1') \otimes \sum_{1 \leq i \leq n}\lambda_i x_i^2  \\
					&= \sum_{1 \leq i\leq j \leq n} \varphi_{ij} \otimes x_ix_j + e_1(\alpha_1-\gamma_1') \otimes \sum_{1 \leq i \leq n}\lambda_ix_i^2.
				\end{align*}
				In particular, up to the defining relations it follows that $\im(d_1) = \ker(d_2)$. 
%				Moving onto the second component, by examining coefficients for tensors of the form $x_i \otimes x_1^2$ for $i \neq 1$, we see that $\phi_{11}e_i - e_1\psi_{i1} = \beta_i,$ so that $e_1\alpha_1e_i - e_1\psi_{i1} = \beta_i$, so that we can conclude $\psi_{i1} = \alpha_1e_i - e_1\beta_i$. Examining tensors of the form $x_i \otimes x_1x_j$ for $i \neq j$ both distinct from $1$, we can see that $\phi_{1j}e_i - e_1\psi_{ij} - e_j\psi_{i1} = 0$, and so conclude that 
%				\[e_1\psi_{ij} = e_1\alpha_je_i + e_j\alpha_1e_i - e_j\alpha_1e_i + e_je_1\beta_i = e_1\alpha_je_i - e_1e_j\beta_i,\]
%				so that $\psi_{ij} = \alpha_je_i - e_j\beta_i$. Again, using that $\lambda_2 = \lambda_3 = 1$ as well and repeating the same arguments, the same holds for any $i \neq j$ not necessarily distinct from $1$. On the other hand, if $i \neq 1, 2$, then examining coefficients for $x_i \otimes x_ix_1$, we see that 
%				\[\phi_{1i}e_i - e_i\psi_{i1} - e_1\psi_{ii} = \gamma_1,\]
%				so that $e_1\psi_{ii} = e_1\alpha_ie_i + e_i\alpha_1e_i - e_i\alpha_1e_i+e_ie_1\beta_i - \gamma_1$, so that multiplying on the left by $e_1$ gives $\psi_{ii} = \alpha_ie_i - e_i\beta_i - e_1\gamma_1$. \par 
%				Skipping the third component and passing directly to the last component, using an argument similar to the last component we see that $\varphi_{ij} = \omega_ie_j + \omega_je_i$ for $i \neq j$, and $\varphi_{ii} = \omega_ie_i$. 
		\end{itemize}
	\end{proof}
	\begin{proposition} \label{cliffcliffprop}
		For any $\cG_1,\cG_2 \in \Db(S,\cl_0)$, if $0 < t \leq n - 4$, then 
		\[R\Hom((\cF \otimes_{\tau^*\cl_0} \tau^*\cG_1)(t),\cF \otimes_{\tau^*\cl_0}\tau^*\cG_2) = 0.\]
	\end{proposition}
	\begin{proof}
		By applying various adjunctions, there is an isomorphism
		\[
			R\Hom((\cF \otimes_{\tau^*\cl_0}\tau^*\cG_1)(t),\tau^*\cG_2) \simeq R\Hom_{\cl_0}(\tau^*\cG_1,\tau_*(\cF^\vee \otimes \cF(-t)) \otimes \cG_2),\]
			and we conclude by Lemma \ref{cliffclifflemma}.
	\end{proof}
	Combining the various cases of semiorthogonality proven already yields the second main result of this paper. 
	\begin{proof}[Proof of Theorem \ref{sodtheorem}]
		When $n = 4$, the statement of Theorem \ref{sodtheorem} reduces to the embedding of Theorem \ref{embeddingtheorem}, so we can and do assume $n > 4$. All of the embeddings can be deduced up to twisting by a line bundle from Theorem \ref{embeddingtheorem} and Proposition \ref{embeddingprop}. By combining Propositions \ref{noncliffnoncliffprop}, \ref{cliffnoncliffprop}, and \ref{cliffcliffprop} and applying Serre (or Grothendieck) duality if necessary, the necessary semiorthogonality follows. 
	\end{proof}
	\section{Residual categories: pencils and smooth quadric fibrations} \label{residualsection}
	Recall that $p : \cQ \to S$ is the quadric fibration of relative dimension $n - 2$ and $\tau : \OGr(2,\cQ) \to S$ is the relative orthogonal Grassmannian fibration. \par In some situations, it is possible to understand the categories $\cR_{\text{even}}$ and $\cR_{\text{odd}}$ appearing in Theorem \ref{sodtheorem}. The first is the case of a smooth fibration, which fiber-by-fiber recovers the decompositions of \cite{kuznetsovodd} and \cite{kuznetsoveven}. \par 
	On the other hand, motivated by the conjectural semiorthogonal decompositions appearing in the author's previous work \cite{fanoschemequaddecomp} and \cite{flips}, we also consider the case where $S = \bbP^1$ and the quadric fibration $\cQ \to S$ arises as the total space of a pencil of quadrics $\langle Q_1, Q_2 \rangle$ with smooth base locus $Q_1 \cap Q_2$. 
	\subsection{Odd-dimensional case}
	Suppose that $n = 2m + 1$ is odd. \par 
	\begin{proposition}\label{oddresidual}
		When $\cQ \to S$ is a smooth quadric fibration of relative dimension $2m - 1$, the category $\cR_{\text{odd}}$ of Theorem \ref{sodtheorem} is trivial.
	\end{proposition}
	\begin{proof}
	Since the decomposition of Theorem \ref{sodtheorem} is $S$-linear, pulling back to any point $s \in S$ yields a semiorthogonal decomposition, and for any $s \in S$ the isomorphism \eqref{oddspinor} discussed in Section \ref{kernelsection} shows that the embedding functor $\Db(\bbC)\simeq \Db(\textcl_{\text{odd}}) \hookrightarrow \Db(\OGr(2,Q_s))$ of Theorem \ref{embeddingtheorem} has image generated by the spinor bundle. In particular, the semiorthogonal sequence 
	\[\langle \cB, \cB(1), \cdots, \cB(2m-3)\rangle\]
	of Theorem \ref{sodtheorem} reduces to the full exceptional collection of \cite[Theorem 7.1]{kuznetsovodd}, and the fiber over $s \in S$ of the $S$-linear residual category $\cR_{\text{odd}}$ is trivial. It follows that $\cR_{\text{odd}}$ is trivial as well when $\cQ \to S$ is smooth.
	\end{proof}
	When $\cQ \to S$ is not smooth, the category $\cR_{\text{odd}}$ need not be trivial, and can have nonzero homological invariants. For the following, we suppose that $S = \bbP^1$, and $\cQ = \langle Q_1,Q_2 \rangle $ is the total space of a pencil of quadrics in $\bbP(V)$ with smooth base locus $X = Q_1 \cap Q_2$, where $\dim V = 2m + 1$. By \cite[Theorem 23]{flips}, there is a standard flip
	\begin{equation}
		\begin{tikzcd}
			\bbP(\Sym^2 \cU_2^\vee) \arrow[r, hook] \arrow[d] & X^{[2]} \arrow[r, leftrightarrow, dashed, "\text{flip}"] & \OGr(2,\cQ)  \arrow[d] & F_1(X) \times \bbP^1 \arrow[d] \arrow[l, hook'] \\
			F_1(X) & & \bbP^1 & F_1(X)
		\end{tikzcd}
	\end{equation}
	which by \cite[Corollary 3]{flips} induces a semiorthogonal decomposition (where we suppress the embedding functors)
	\[\Db(X^{[2]}) = \langle \Db(\OGr(2,\cQ)),\Db(F_1(X))\rangle.\]
	Using this geometric construction, we can in fact calculate invariants for $\cR_{\text{odd}}$. 
	\begin{proposition}
		When $\cQ \to \bbP^1$ is the total space of a pencil of quadrics in $\bbP^{2m}$ with smooth base locus, the category $\cR_{\text{odd}}$ has 
		\[\mathrm{HH}_0(\cR_{\text{odd}}) = 2(2m+1),\]
		and $\mathrm{HH}_i(\cR_{\text{odd}}) = 0$ for all $i \neq 0$. 
		Moreover, if \cite[Conjecture D]{fanoschemequaddecomp} holds for $F_1(X)$ then
		\[[\cR_{\text{odd}}] = 2(2m+1)[\Db(\bbC)]\]
		in the Grothendieck ring of $\bbC$-linear categories $K_0(\Cat_{\bbC})$. 
	\end{proposition}
	\begin{proof}
		By \cite[Theorem B]{kps}, there is a semiorthogonal decomposition
		\begin{align*}
			\Db(X^{[2]}) &= \langle \Db_{\bbZ/2}(X \times X), \underbrace{\Db(X),\cdots,\Db(X)}_{\dim X - 2 \text{ times}}\rangle \\
			&= \langle \Sym^2 \Db(X), \underbrace{\Db(X),\cdots,\Db(X)}_{\dim 2m- 4 \text{ times}}\rangle
		\end{align*}
		From \cite[Corollary 5.7]{quadfibs}, it is possible to write a semiorthogonal decomposition \[\Db(X) = \langle\Db(\cC),\underbrace{\Db(\bbC),\dots,\Db(\bbC)}_{2m-3 \text{ times}}\rangle\] where $\cC$ is a certain root stack with $\bbZ/2$-stabilizers at $2m+1$ points on $\bbP^1$. In particular, by the semiorthogonal decomposition for a root stack \cite[Theorem 1.6]{rootstackdecomp}, $\Db(X)$ has a full exceptional collection of length $4m$. \par 
		Using \cite[Corollary 1.3]{koseki}, it follows that $\Sym^2 \Db(X)$ also admits a full exceptional collection of length ${4m \choose 2} + 8m$, so we may conclude that $\Db(X^{[2]})$ has a full exceptional collection as well, which will be of length \[{4m \choose 2} + 8m + 4m(2m-4) = 16m^2 - 10m.\] 
		As a consequence, we deduce that the Hochschild homology of $X^{[2]}$ is precisely 
		\[\mathrm{HH}_0(X^{[2]}) = 16m^2 - 10m, \qquad \mathrm{HH}_i(X^{[2]}) = 0\text{ for all }i\neq 0.\]
		By \cite[Proposition 4.2]{fanoschemequaddecomp}, we may also calculate the Hochschild homology of the Fano variety of lines as 
		\[\mathrm{HH}_0(F_1(X)) = \chi(F_1(X)) = 16{m \choose 2} = 8m^2 - 8m, \qquad \mathrm{HH}_i(F_1(X)) = 0 \text{ for all }i \neq 0,\]
		and by using the additivity of Hochschild homology under semiorthogonal decomposition conclude that 
		\[\mathrm{HH}_0(\OGr(2,\cQ)) = 8m^2 - 2m, \qquad \mathrm{HH}_i(\OGr(2,\cQ)) = 0\text{ for all }i\neq 0.\]\par 
		On the other hand, the categories $\cB$ in the decomposition of Theorem \ref{sodtheorem} are composed of $m - 1$ copies of $\Db(\bbP^1)$ along with one copy of $\Db(\bbP^1,\cl_0)$, which by \cite[Corollary 3.16]{quadfibs} coincides with the derived category $\Db(\cC)$ of the stacky curve above. In particular, there is a semiorthogonal decomposition 
		\begin{align*}
			\Db(\OGr(2,\cQ)) &= \langle \cR_{\text{odd}}, \underbrace{\Db(\cC),\dots,\Db(\cC)}_{2m-2 \text{ times}}, \underbrace{\Db(\bbP^1),\dots,\Db(\bbP^1)}_{(2m-2)(m-1) \text{ times}} \rangle \\
			&= \langle \cR_{\text{odd}}, \underbrace{\Db(\bbC),\dots,\Db(\bbC)}_{(2m+3)(2m-2)+2(2m-2)(m-1) \text{ times}}\rangle.
		\end{align*}
		From the identity 
		\[(2m+3)(2m-2) + 2(2m-2)(m-1) = 8m^2 - 6m - 2,\]
		it follows again from the additivity of Hochschild homology that 
		\[\mathrm{HH}_0(\cR_{\text{odd}}) = 4m + 2, \qquad \mathrm{HH}_i(\cR_{\text{odd}}) = 0 \text{ for all }i\neq 0.\]
		If \cite[Conjecture D]{fanoschemequaddecomp} holds, then $\Db(F_1(X))$ admits a full exceptional collection of $8m^2 - 8m$ objects, and the rest of the argument can be promoted to hold on the level of $K_0(\Cat_{\bbC})$. 
	\end{proof}
	In particular, in this situation $\cR_{\text{odd}}$ is a $\bbP^1$-linear category which is generically trivial over $\bbP^1$ and supported over the ramification divisor of the pencil. We expect that each singular fiber contributes 2 exceptional objects to the full exceptional collection for the pencil. 
	\subsection{Even-dimensional case}
	Suppose that $n = 2m$ is even. In this case, the residual category $\cR_{\text{even}}$ is not trivial even for smooth fibrations. \par 
	When $S = \Spec(\bbC)$ is a point and the quadric $Q \subset \bbP(V)$ is smooth, \cite[Theorem 3.1]{kuznetsoveven} shows that $\cR_{\text{even}}$ consists of two exceptional $\Spin(2m)$-equivariant vector bundles $\cU^{2\omega_{m-1}}(-1)$ and $\cU^{2\omega_m}(-1)$ on $\OGr(2,Q)$, which by (3.16) and (3.26) in \textit{op. cit.} are the direct summands of $\bigwedge^{m-2}(\cU_2^\perp/\cU_2)$, where $\cU_2^\perp$ is the orthogonal complement of $\cU_2 \subset V$ with respect to $Q$. \par 
	For smooth fibrations, it is possible to relativize this construction in a natural way. For the relative orthogonal Grassmannian $\OGr(2,\cQ)$, we will endow $\bigwedge^{m-2}(\cU_2^\perp/\cU_2)$ with the structure of a module over the sheaf of algebras $\tau^*\cZ(\cl_0)$ and use it to prove the following result. 
	\begin{proposition}\label{evenresidual}
		Suppose $\cQ \to S$ is a smooth quadric fibration of relative dimension $2m - 2$, the category $\cR_{\text{even}}$ is equivalent to $\Db(S,\cZ(\cl_0))$, where $\cZ(\cl_0)$ is the center for the sheaf of even parts of the Clifford algebra of $\cQ$. 
	\end{proposition}
	\begin{proof}
		Recall that the center of the Clifford algebra is given by the subalgebra \[\cZ(\cl_0) = \cO_S \oplus \det(\cE) \otimes \cL^{\otimes m} \subset \cl_0.\]
		There is a natural morphism 
		\begin{equation}\label{hodgestar}
			\det(\cE) \otimes \cL^{\otimes m} \otimes \bigwedge^{m-2} (\cU_2^\perp/\cU_2) \to \bigwedge^{m-2}(\cU_2^\perp/\cU_2) 
		\end{equation}
		called the \textit{Hodge star operator}, such that its square coincides with the tensor product of $\bigwedge^{m-2}(\cU_2^\perp/\cU_2)$ with the map
		\[\det(\cE)^{\otimes 2} \otimes \cL^{\otimes 2m} \to \cO\]
		defining the algebra structure on $\cZ(\cl_0)$; this is enough to define the structure of a module over $\tau^*\cZ(\cl_0)$ on $\bigwedge^{m-2}(\cU_2^\perp/\cU_2)$. We outline how to construct the Hodge star as follows: \par 
		Recall that the non-degenerate quadratic form $\Sym^2 \cE \to \cL^\vee$ induces an isomorphism $\cE \to \cL^\vee \otimes \cE^\vee$, and that by definition
		\[\cU^\perp_2 \coloneq \ker(\cE \iso \cL^\vee \otimes \cE^\vee \to \cL^\vee \otimes \cU_2^\vee).\]
		Then the restriction 
		\[\Sym^2 \cU_2^\perp  \hookrightarrow \Sym^2 \cE \to \cL^\vee\]
		naturally descends to a nondegenerate quadratic form $\Sym^2 (\cU_2^\perp/\cU_2) \to \cL^\vee$, and there is a canonical induced isomorphism $\cU_2^\perp/\cU_2 \to \cL^\vee \otimes  (\cU_2^\perp/\cU_2)^\vee$ so that the following diagram commutes: 
		\begin{equation*}
			\begin{tikzcd}
				\cU_2^\perp/\cU_2 \arrow[r,"\sim"]  & \cL^\vee \otimes (\cU_2^\perp/\cU_2)^\vee \arrow[d,hook]\\
				\cU_2^\perp \arrow[r] \arrow[d,hook] \arrow[u,two heads]& \cL^\vee \otimes (\cU_2^\perp)^\vee \\
				\cE \arrow[r,"\sim"] & \cL^\vee \otimes \cE^\vee \arrow[u,two heads]
			\end{tikzcd}
		\end{equation*}
		As a consequence, we deduce the existence of a map (in fact, an isomorphism)
		\[\cL^{\otimes {(m-2)}} \otimes \bigwedge^{m-2}(\cU_2^\perp/\cU_2) \to \bigwedge^{m-2}(\cU_2^\perp/\cU_2)^\vee.\]
		On the other hand, there is another perfect pairing arising from the determinant: 
		\[\bigwedge^{m-2}(\cU_2^\perp/\cU_2) \otimes \bigwedge^{m-2}(\cU_2^\perp/\cU_2) \to \det(\cU_2^\perp/\cU_2),\]
		inducing an isomorphism 
		\[\det(\cU_2^\perp/\cU_2) \otimes \bigwedge^{m-2}(\cU_2^\perp/\cU_2)^\vee \iso \otimes \bigwedge^{m-2} (\cU_2^\perp/\cU_2).\]
		Up to twisting, composition gives a map 
		\[\det(\cU_2^\perp/\cU_2)^\vee \otimes \cL^{\otimes(m-2)} \otimes \bigwedge^{m-2}(\cU_2^\perp/\cU_2) \to \bigwedge^{m-2}(\cU_2^\perp/\cU_2).\]
		In fact, there is an isomorphism $\det(\cE) \otimes \cL^{\otimes m} \to \det(\cU_2^\perp/\cU_2)^\vee \otimes \cL^{\otimes(m-2)}$: indeed, by construction there is a short exact sequence 
		\[0 \to \cU_2^\perp \to \cE \to \cL^\vee \otimes \cU_2^\vee \to 0,\]
		and taking determinants gives $\det(\cE) \simeq \det(\cU_2^\perp) \otimes \det(\cU_2^\vee) \otimes (\cL^\vee)^{\otimes2} \simeq \det(\cU_2^\perp/\cU_2) \otimes (\cL^\vee)^{\otimes 2}$. Composing with this isomorphism gives the desired map, which is fiberwise identified with the usual Hodge star operator; moreover, its square can be checked locally to coincide with the multiplication map $\det(\cE)^{\otimes 2} \otimes \cL^{2m} \to \cO$. \par
		As mentioned above, since the quadric fibration is smooth, the bundle $\bigwedge^{m-2}(\cU_2^\perp/\cU_2)$ fiberwise decomposes as a direct sum of two exceptional $\Spin(2m)$-equivariant irreducible vector bundles $\cU^{2\omega_{m-1}}(-1)$ and $\cU^{2\omega_m}(-1)$ \cite[(3.16) and (3.26)]{kuznetsoveven}. Since the Hodge star operator is naturally $\Spin(2m)$-equivariant by construction, the eigenspaces of the Hodge star operator $\bigwedge^{m-2}(\cU_2^\perp/\cU_2)$ are equivariant subbundles, which necessarily must coincide with the exceptional vector bundles $\cU^{2\omega_{m-1}}(-1)$ and $\cU^{2\omega_m}(-1)$. \par
		We next verify that the functor \begin{align*}
			\Phi_{\bigwedge^{m-2}(\cU_2^\perp/\cU_2)}: \Db(S,\cZ(\cl_0)) &\to \Db(\OGr(2,\cQ)) \\
			\cG &\mapsto \tau^*\cG \otimes_{\tau^*\cZ(\cl_0)} \bigwedge^{m-2} (\cU_2^\perp/\cU_2)
		\end{align*}
		defines a fully faithful functor. By the same trick of composing with the adjoint used in the proof of Theorem \ref{embeddingtheorem}, it suffices to show that there is an isomorphism 
		\[\cZ(\cl_0) \simeq R\tau_*\left(\bigwedge^{m-2}(\cU_2^\perp/\cU_2)^\vee \otimes \bigwedge^{m-2}(\cU_2^\perp/\cU_2)\right).\]
		By Grothendieck's theorem on cohomology and base change, to show that the higher derived pushforwards vanish, it is enough to observe that fiberwise over $s \in S$ this vector bundle has no higher cohomology. Using the decomposition into exceptional vector bundles $\cU^{2\omega_{m-1}}(-1) \oplus \cU^{2\omega_m}(-1)$, this follows from the proof of \cite[Lemma 3.12]{kuznetsoveven}. To conclude it suffices to show that the natural map 
		\[\cZ(\cl_0) \to \tau_*\left(\bigwedge^{m-2}(\cU_2^\perp/\cU_2)^\vee \otimes \bigwedge^{m-2}(\cU_2^\perp/\cU_2)\right) \simeq \tau_*\homsheaf\left(\bigwedge^{m-2}(\cU_2^\perp/\cU_2),\bigwedge^{m-2}(\cU_2^\perp/\cU_2)\right)\]
		induced by the $\cZ(\cl_0)$-module structure on $\bigwedge^{m-2}(\cU_2^\perp/\cU_2)$ is an isomorphism. We may also check this isomorphism fiberwise. But since the action of the Hodge star operator \eqref{hodgestar} on $\bigwedge^{m-2}(\cU_2^\perp/\cU_2)$ acts with different eigenvalues on the two eigenspaces, it follows that the action of the nontrivial element of the center $Z(\textcl_{\text{even}})$ is not the identity map. It follows that on the fibers, the map 
		\[Z(\textcl_{\text{even}}) \to \Hom\left(\bigwedge^{m-2}(\cU_2^\perp/\cU_2),\bigwedge^{m-2}(\cU_2^\perp/\cU_2)\right) \simeq \Hom\left(\cU^{2\omega_{m-1}},\cU^{2\omega_{m-1}}\right) \oplus \Hom(\cU_2^{\omega_m},\cU_2^{\omega_m})\]
		is an injective map of two-dimensional vector spaces, and therefore an isomorphism, so in particular the functor $\Phi_{\bigwedge^{m-2}(\cU_2^\perp/\cU_2)}$ is fully faithful. \par 
		Next observe that the image of this functor naturally lands inside of $\cR_{\text{even}}$. Indeed, for any object $\cG_1 \in \Phi_{\bigwedge^{m-2}(\cU_2^\perp/\cU_2)}(\Db(S,\cZ(\cl_0)))$ and any object $\cG_2 \in \langle \cA, \cB(1),\cdots \cB(m-2), \cA(m-1), \cdots \cA(2m-4) \rangle$ in the semiorthogonal sequence of Theorem \ref{sodtheorem}, we may observe that the restriction of $R\tau_*R\homsheaf(\cG_2,\cG_1)$ vanishes upon restriction to any smooth fiber, since under the equivalence $\Db(\bbC) \times \Db(\bbC) \simeq \Db(Z(\textcl_{\text{even}}))$ the category $\Phi_{\bigwedge^{m-2}(\cU_2^\perp/\cU_2)}(\Db(Z(\textcl_{\text{even}})))$ is generated by the two exceptional vector bundles $\cU_2^{2\omega_{m-1}}$ and $\cU_2^{2\omega_m}$ which are left orthogonal to the exceptional objects generating $\cG_1$ (e.g. the spinor bundles and symmetric powers of $\cU_2$), c.f. \cite[Theorem 3.1]{kuznetsoveven}. It follows that $R\Hom(\cG_2,\cG_1) = 0$. Likewise, the fullness of the resulting semiorthogonal sequence 
		\[\langle \Db(S,\cZ(\cl_0)), \cA,\cB(1),\cdots,\cB(m-2),\cA(m-1),\cdots,\cA(2m-4)\rangle\]
		follows from fiberwise restriction as in the proof of Proposition \ref{oddresidual}. 
	\end{proof}
	\begin{remark}
		Even when $\cQ \to S$ is a smooth fibration, the categories $\Db(S,\cZ(\cl_0))$ and $\Db(S,\cl_0)$ can differ. In general for a quadric fibration with at worst simple nodal degenerations, the category $\Db(S,\cl_0)$ can be identified with a Brauer twist of $\Db(\underline{\Spec}_S(\cZ(\cl_0))) \simeq \Db(S,\cZ(\cl_0))$, as $\cl_0$ is \'etale-locally a matrix algebra over its center \cite[Proposition 3.13]{quadfibs}.
	\end{remark}
	It is not clear to the author what is the best way to extend the embedding kernel of Proposition \ref{evenresidual} to singular fibers. However, if we turn to the case where $\cQ \to \bbP^1$ is the total space of a pencil of quadrics in $\bbP^{2m-1}$ with smooth base locus, then the standard flip of \cite[Theorem 23]{flips} as discussed above still holds, and a similar calculation to the odd case can be used to produce a prediction for the structure of $\cR_{\text{even}}$. 
	\begin{proposition}[Section 5 of \cite{flips}]
		When $\cQ \to \bbP^1$ is the total space of a pencil of quadrics in $\bbP^{2m-1}$ with smooth base locus, if \cite[Conjecture A]{fanoschemequaddecomp} holds for $F_1(X)$, then 
		\[[\cR_{\text{even}}] = [\Db(C)]\]
		in the Grothendieck ring of $\bbC$-linear categories $K_0(\Cat_\bbC)$, where $C$ is the double cover of $\bbP^1$ branched over the discriminant divisor of the pencil. 
	\end{proposition}
	Together with Proposition \ref{evenresidual}, it seems likely that there should be an equivalence $\cR_{\text{even}} \simeq \Db(C)$ for pencils of quadrics with smooth base locus. That claim generalizes to the following.  
	\begin{conjecture}
		For any quadric fibration $\cQ \to S$ which is not too singular, there exists a natural extension of the embedding kernel of Proposition \ref{evenresidual} giving an equivalence $\Db(S,\cZ(\cl_0)) \simeq \cR_{\text{even}}$. 
	\end{conjecture}
	Optimistically, one might expect that the conjecture holds when the fibers of the quadric fibration $\cQ \to S$ all have rank $\geq 3$. We certainly expect it will hold when $\cQ \to S$ has at worst simple nodal degenerations, from which the case of a pencil of quadrics with smooth base locus would follow. 
	\printbibliography
\end{document}